\documentclass[10pt,reqno]{amsart}

\usepackage{listings}
\usepackage{amsmath,amsthm,amsfonts,amssymb,amscd}
\usepackage{varioref}
\usepackage{mathabx}
\usepackage{graphicx}
\usepackage[usenames]{color}
\usepackage[mathscr]{euscript}
\usepackage{stmaryrd}
\usepackage{bm}
\usepackage{mathrsfs}
\usepackage[colorlinks,backref]{hyperref}
\usepackage{upref}

\DeclareMathAlphabet{\mathpzc}{OT1}{pzc}{m}{it}

\newtheorem{theorem}{Theorem}[section]
\newtheorem{proposition}[theorem]{Proposition}
\newtheorem{lemma}[theorem]{Lemma}
\newtheorem{corollary}[theorem]{Corollary}

\newtheorem{definition}[theorem]{Definition}
\newtheorem{remark}[theorem]{Remark}

\DeclareMathOperator{\spt}{spt}

\def\bbN{{\mathbb N}}

\def\bbR{{\mathbb R}}

\def\D{{\mathcal D}}

\def\F{{\mathcal F}}
\def\G{{\mathcal G}}

\def\O{{\mathcal O}}
\def\P{{\mathcal P}}

\def\U{{\mathcal U}}
\def\V{{\mathcal V}}

\def\Varphi{{\boldsymbol\varphi}}
\def\Varpsi{{\boldsymbol\psi}}
\def\sbF{\text{\bf\emph{F}}}

\def\rA{{\rm A}}

\def\rE{{\rm E}}

\def\rH{{\rm H}}

\def\rO{{\rm O}}
\def\rP{{\rm P}}

\def\rT{{\rm T}}
\def\rU{{\rm U}}

\def\re{{\rm e}}

\def\rg{{\rm g}}

\def\rk{{\rm k}}

\def\rw{{\rm w}}

\def\mC{{\mathscr C}}

\definecolor{grey}{rgb}{0.5,0.5,0.5}
\definecolor{lightgrey}{rgb}{0.9,0.9,0.9}
\definecolor{darkgreen}{rgb}{0,0.6,0}
\definecolor{orange}{rgb}{1,0.5,0}
\definecolor{lightpink}{rgb}{1,0.714,0.757}
\definecolor{lightorange}{rgb}{1,0.855,0.725}

\def\Holder{{H\"{o}lder}}
\def\Poincare{{Poincar\'e}}

\def\div{{\operatorname{div}}}

\def\curl{{\operatorname{curl}}}

\def\bdy #1{\partial\hspace{1pt} #1}
\def\cls #1{\overline {#1}}

\def\id{{\text{Id}}}
\def\supp{{\text{spt}}}

\def\cptsubset{\hspace{1pt}{\subset\hspace{-2pt}\subset}\hspace{1pt}}

\def\contsubset{\hspace{1pt}{\hookrightarrow}\hspace{1pt}}

\def\bp{{\overline{\partial}\hspace{1pt}}}
\def\p{{\partial\hspace{1pt}}}
\def\bdygrad {\nabla^{\hspace{-.5pt}\scriptscriptstyle{\p\hspace{-1pt}\Omega}}}
\def\bdydiv {\div^{\scriptscriptstyle{\p\hspace{-1pt}\Omega}}}
\def\bpD{{\nabla_{\text{\tiny\hspace{-3pt}$\bdy\hspace{-1pt}\D$}}}}

\def\n{{\rm n}}
\def\cptspt{{c}}

\def\Forall{\forall\hspace{2pt}}

\def\Rn{{\bbR^{\hspace{0.2pt}\n}}}

\def\comm#1#2{{\llbracket#1,#2\rrbracket}}
\def\bigcomm#1#2{{\big\llbracket#1,#2\big\rrbracket}}

\def\({{(\hspace{-2pt}(}}
\def\){{)\hspace{-2pt})}}

\def\texorpdfstring#1#2{{#1}} 

\def\smallexp#1{{\text{\small #1}}}
\def\dfrac#1#2{\smallexp{$\displaystyle{}\frac{#1}{#2}$}}

\def\footnoteexp#1{{\text{\footnotesize #1}}}

\def\cvl{{\convolution}}
\def\novertwo{\text{\small $\displaystyle{}\frac{\n}{2}$}}

\def\XXint#1#2#3{{\setbox0=\hbox{$#1{#2#3}{\int}$}
\vcenter{\hbox{$#2#3$}}\kern-.5\wd0}}

\title[Regularity Theory for Elliptic Equations of Hodge-Type]{Solvability and regularity for an elliptic system prescribing the curl, divergence, and partial trace of a vector field on Sobolev-class domains}

\author[C.H. A. Cheng]{C.H. Arthur Cheng}
\address{Department of Mathematics, National Central University, Jhongli City, Taoyuan County, 32001, Taiwan ROC}
\email{cchsiao@math.ncu.edu.tw}

\author[S. Shkoller]{Steve Shkoller}
\address{Mathematical Institute,
University of Oxford,
Andrew Wiles Building,
Radcliffe Observatory Quarter,
Woodstock Road,
Oxford, OX2 6GG, UK}
\email{shkoller@maths.ox.ac.uk}


\date{July 31, 2014}

\begin{document}

\maketitle
{\small
\tableofcontents}

\def\f{\text{\bf\emph{f}}}
\def\g{\text{\bf\emph{g}}}
\def\nn{\text{\bf\emph{n}}}
\def\u{\text{\bf\emph{u}}}
\def\v{\text{\bf\emph{v}}}
\def\w{\text{\bf\emph{w}}}
\def\bfw{{\mathbf w}}
\def\bfn{{\mathbf n}}
\def\h{\text{\bf\emph{h}}}
\def\bN{{\bf N}}
\def\bT{{\bf T}}
\def\wtrN{{\widetilde{\bN}}}

\section{Introduction}\label{sec:introduction}
\subsection{The main results}
Given a sufficiently smooth Sobolev-class bounded domain $ \Omega \subseteq \bbR^\n $ and forcing functions $\f $and $g$ in $\Omega$ together with
boundary data given by
either $h$ or $\h$ on
$\bdy \Omega$, we establish the basic elliptic estimates for the vector elliptic system of Hodge-type:
\begin{subequations}\nonumber
\begin{alignat}{2}
\operatorname{curl} \v &= \f \qquad&&\text{in}\quad \Omega\,,\\
\operatorname{div} \v &= g &&\text{in}\quad \Omega\,,
\end{alignat}
\end{subequations}
with boundary conditions given by either
$$
\v \cdot \bN =h \ \text{ or } \ \v \times \bN = \h \ \ \ \text{on}\quad \bdy\Omega \,,
$$
where $\bN$ is the unit normal vector on $\bdy \Omega $.   When the domain $\Omega$ is of class $\mC^{k+1}$, elliptic estimates for solutions $\v$ in
$H^{k+1}(\Omega)$ are now classical.  We extend this well-known theory to the case of domains $\Omega$ of Sobolev class $H^{k+1}$.

We first establish the following
\begin{theorem}\label{thm:main_thm1}
Let $\Omega\subseteq \bbR^3$ be a bounded $H^{\rk+1}$-domain with $\rk > \smallexp{$\displaystyle{}\frac{3}{2}$}$. Given $\f,g\in H^{\ell-1}(\Omega)$ with
$ \operatorname{div} \f=0$,
consider the equations
\begin{subequations}\label{div_curl}
\begin{alignat}{2}
\operatorname{curl} \v &=\f \qquad&&\text{in}\quad \Omega\,,\\
\operatorname{div} \v &= g &&\text{in}\quad \Omega\,.
\end{alignat}
\end{subequations}
\begin{enumerate}
\item[{\rm(1)}] If $\f$ satisfies
    \begin{equation}\label{solvability_condition_for_normal_trace_problem}
    \int_\Gamma \f\cdot \bN \,dS = 0 \quad\text{for each connected component $\Gamma$ of $\bdy\Omega$}\,,
    \end{equation}
    and $h\in H^{\ell-0.5}(\bdy\Omega)$ satisfies
    $ \smallexp{$\displaystyle{}\int_{\bdy\Omega}$} h\, dS= \smallexp{$\displaystyle{}\int_\Omega$} g\, dx $,
    then, for $1\le \ell \le \rk$, there exists a solution $\v\in H^\ell(\Omega)$ to {\rm(\ref{div_curl})} with boundary condition 
    \begin{equation}\label{bc1}
    \v \cdot \bN = h \qquad\text{on}\quad \bdy\Omega\,,
    \end{equation}
    such that
    $$
    \qquad\quad \|\v\|_{H^\ell(\Omega)} \le C(|\bdy\Omega|_{H^{\rk+0.5}}) \Big[\|\f\|_{H^{\ell-1}(\Omega)} + \|g\|_{H^{\ell-1}(\Omega)} + \|h\|_{H^{\ell-0.5}(\bdy\Omega)} \Big]\,.
    $$
    The solution is unique if $\Omega$ is the disjoint union of simply connected open sets.\vspace{.1cm}

\item[{\rm(2)}] If $\f$ satisfies {\rm(\ref{solvability_condition_for_normal_trace_problem})} and $\f \cdot \bN = \bdydiv  \h$ on $ \bdy\Omega$ {\rm(}where $\bdydiv $ denotes the surface divergence operator {\rm}defined in Definition \ref{defn:tangential_derivative}{\rm)}, $\h\in H^{\ell-0.5}(\bdy\Omega)$ satisfies $\h\cdot \bN = 0$ as well as
    \begin{equation}\label{solvability_condition_for_tangential_trace_problem}
    \begin{array}{c}
    \qquad\smallexp{$\displaystyle{}\int_\Sigma$} \f \cdot \bfn \, dS = \smallexp{$\displaystyle{}{\oint}_{\hspace{-3pt}\bdy\Sigma}$} (\bN \times \h)\cdot d {\bf r} \text{ if $\Sigma \subseteq \cls\Omega$ has piecewise smooth boundary $\bdy \Sigma \subseteq \bdy\Omega$}\\
    \qquad\qquad\qquad\qquad\qquad\qquad\text{ with unit normal $\bfn$, compatible with the orientation of $\bdy\Sigma$,}
    \end{array}
    \end{equation}
    then, for $1\le \ell \le \rk$, there exists a solution $\v\in H^\ell(\Omega)$ to {\rm(\ref{div_curl})} with boundary condition 
    \begin{equation}\label{bc2}
    \v \times \bN =\h \qquad\text{on}\quad \bdy\Omega\,,
    \end{equation}
    such that
    $$
    \|\v\|_{H^\ell(\Omega)} \le C(|\bdy\Omega|_{H^{\rk+0.5}}) \Big[\|\f\|_{H^{\ell-1}(\Omega)} + \|g\|_{H^{\ell-1}(\Omega)} + \|\h\|_{H^{\ell-0.5}(\bdy\Omega)} \Big]\,.
    $$
 The solution is unique if each connected component of $\Omega$  has a connected boundary.
\end{enumerate}
\end{theorem}

\begin{remark}\label{rmk:solvability_condition}
To explain condition {\rm(\ref{solvability_condition_for_normal_trace_problem})}, let $\Omega$ be a connected bounded open set, and let $\{\Gamma_i\}_{i=0}^I$ denote the connected components of $\bdy\Omega$ in which $\Gamma_0$ is the boundary of the unbounded connected component of $\Omega^\complement$.  For each $i=1,...,I$, let  $q_i$ be the solution to
\begin{subequations}\label{q_eq}
\begin{alignat}{2}
\Delta q_i &= 0 &&\text{in}\quad\Omega\,,\\
q_i &= \delta_{ij} \qquad&&\text{on}\quad \Gamma_j\,.
\end{alignat}
\end{subequations}
Then if $\u$ satisfies $\curl \u = \f$ in $\Omega$ for some divergence-free vector $\f$,  applying the divergence theorem and then integrating-by-parts, shows
that
\begin{align*}
\int_{\Gamma_i} \f \cdot \bN dS &= \int_{\Gamma_i} (\f \cdot \bN) q_i dS = \int_{\bdy\Omega} (\f \cdot \bN) q_i dS = \int_\Omega q_i \div \f dx + \int_\Omega \f \cdot \nabla q_i dx \\
&= \int_\Omega \curl \u \cdot \nabla q_i dx = \int_{\bdy\Omega} (\bN \times \u) \cdot \nabla q_i dS = \int_{\bdy\Omega} (\bN \times \u) \cdot \bdygrad q_i dS = 0\,.
\end{align*}
Therefore, for  $\curl \u = \f$ to be solvable, it is necessary that
$$
\int_{\Gamma_i} \f\cdot \bN \,dS = 0 \qquad\Forall i \in \{1,\cdots,I\}\,.
$$
Then, an application of the divergence theorem, together with the fact that $ \operatorname{div} \f=0$, shows that
$\smallexp{$\displaystyle{}\int_{\Gamma_0}$} \f \cdot \bN \, dS = 0$.  In other words, {\rm(\ref{solvability_condition_for_normal_trace_problem})} is a
necessary condition for the solvability of {\rm(\ref{div_curl}a)}. We will show that it is also one of the sufficient conditions to solve {\rm(\ref{div_curl})} with
the boundary conditions  {\rm(\ref{bc1})} or {\rm(\ref{bc2})}.
\end{remark}

The  problem (\ref{div_curl}) with either boundary conditions (\ref{bc1}) or (\ref{bc2}) has been well studied. The characterization of the kernel of both
problems, the solvability conditions, and the existence theory has been developed in a number of papers;  see, for example, \cite{AgDoNi1964},
\cite{FoTe1978}, \cite{Ge1979}, \cite{Picard1982}, \cite{BeDoGa1985}, \cite{AmBeDaGi1998}, \cite{KoYa2009}, \cite{AmSe2013}, and the references therein.   The
inequalities given in Theorem \ref{thm:main_thm1} are new for Sobolev-class domains.

Motivated by the analysis of the free-boundary problems which arise in inviscid fluid dynamics,
we next state a theorem which provides two fundamental elliptic estimates set on Sobolev-class domains:

\begin{theorem}\label{thm:main_thm2}
Let $\Omega\subseteq \Rn$, $\n=2$ or $3$, be a bounded $H^{\rk+1}$-domain with $\rk > \novertwo$\,. Then there exists a generic constant $C$ depending on $|\bdy\Omega|_{H^{\rk+0.5}}$ such that for all $\u\in H^{\rk+1}(\Omega)$,
\begin{align}
\|\u\|_{H^{\rk+1}(\Omega)}&\le C \Big[\|\u\|_{L^2(\Omega)} + \|\curl \u\|_{H^\rk(\Omega)} + \|\div \u\|_{H^\rk(\Omega)}
+ \|\bdygrad \u\cdot \bN\|_{H^{\rk-0.5}(\bdy\Omega)}\Big]\,, \label{Hodge_elliptic_estimate1} \\
\|\u\|_{H^{\rk+1}(\Omega)} &\le C \Big[\|\u\|_{L^2(\Omega)} + \|\curl \u\|_{H^\rk(\Omega)} + \|\div \u\|_{H^\rk(\Omega)}
+ \|\bdygrad \u\times \bN\|_{H^{\rk-0.5}(\bdy\Omega)}\Big]\,, \label{Hodge_elliptic_estimate2}
\end{align}
where $\bdygrad \u$ is the tangential derivative on $\bdy \Omega$ {\rm(}defined in Definition \ref{defn:tangential_derivative}{\rm)}.
\end{theorem}

\begin{remark}
The inequalities {\rm(\ref{Hodge_elliptic_estimate1})} and {\rm(\ref{Hodge_elliptic_estimate2})} play a fundamental role in the regularity theory of the Euler equations with moving interfaces; see, for example, \cite{CoSh2007} for the incompressible setting and \cite{CoSh2012} for the compressible problem with vacuum. The use of the norm $ \|\bdygrad \u\cdot \bN\|_{H^{\rk-0.5}(\bdy\Omega)}$ rather than $ \| \u\cdot \bN\|_{H^{\rk+0.5}(\bdy\Omega)}$ is crucial,
as the regularity of the normal vector to field to $\bdy \Omega $ is often worse than the regularity of the velocity vector $\u$.

On the other hand,
if $\Omega$ is at least of class $H^{k+2}$ then the inequalities {\rm(\ref{Hodge_elliptic_estimate1})} and {\rm(\ref{Hodge_elliptic_estimate2})} can be replaced, respectively, by
\begin{align}
\|\u\|_{H^{\rk+1}(\Omega)} &\le C \Big[\|\u\|_{L^2(\Omega)} + \|\curl \u\|_{H^\rk(\Omega)} + \|\div \u\|_{H^\rk(\Omega)}
+ \| \u\cdot \bN\|_{H^{\rk+0.5}(\bdy\Omega)}\Big] \label{Hodge_elliptic_estimate1_smooth} \\
\|\u\|_{H^{\rk+1}(\Omega)} &\le C \Big[\|\u\|_{L^2(\Omega)} + \|\curl \u\|_{H^\rk(\Omega)} + \|\div \u\|_{H^\rk(\Omega)}
+ \| \u\times \bN\|_{H^{\rk+0.5}(\bdy\Omega)}\Big] \label{Hodge_elliptic_estimate2_smooth}
\end{align}
\end{remark}

\begin{remark} Recently, Amrouche \& Seloula \cite{AmSe2013} established the inequalities {\rm(\ref{Hodge_elliptic_estimate1_smooth})} and
 {\rm(\ref{Hodge_elliptic_estimate2_smooth})} in the $L^p$ framework
and for domains $\Omega$ of class $\mC^{\rk+1}$; see Corollary 3.5 in \cite{AmSe2013}. Of course, in the case of a $\mC^{\rk+1}$ domain $\Omega$,
the inequalities {\rm(\ref{Hodge_elliptic_estimate1})} and
 {\rm(\ref{Hodge_elliptic_estimate2})} follow immediately from {\rm(\ref{Hodge_elliptic_estimate1_smooth})} and
 {\rm(\ref{Hodge_elliptic_estimate2_smooth})}, respectively.
 \end{remark}

When $ \Omega $ is very close to a $ \mC^\infty$-domain, we can obtain these inequalities for fractional-order Sobolev spaces, as in the following
\begin{theorem}\label{thm:main_thm3}
Let $\Omega\subseteq \Rn$, $\n=2$ or $3$, be a bounded $H^{s+1}$-domain with $s \in \bbR $ such that $s > \novertwo$\,, and let $ \mathcal{D} $ denote a
$\mC^ \infty $-domain such that the distance between $\bdy \mathcal{D} $ and $ \bdy \Omega $ in the $H^{s+0.5}$-norm is less than $ \epsilon $
for $ 0 < \epsilon \ll 1$.
 Then there exists a generic constant $C$ depending only on $|\bdy \mathcal{D} |_{H^{s+0.5}}$, such that for all $\u\in H^{s+1}(\Omega)$,
\begin{align}
\|\u\|_{H^{s+1}(\Omega)}&\le C \Big[\|\u\|_{L^2(\Omega)} + \|\curl \u\|_{H^s(\Omega)} + \|\div \u\|_{H^s(\Omega)}
+ \|\bdygrad \u\cdot \bN\|_{H^{s-0.5}(\bdy\Omega)}\Big]\,, \label{Hodge_elliptic_estimate1b} \\
\|\u\|_{H^{s+1}(\Omega)} &\le C \Big[\|\u\|_{L^2(\Omega)} + \|\curl \u\|_{H^s(\Omega)} + \|\div \u\|_{H^s(\Omega)}
+ \|\bdygrad \u\times \bN\|_{H^{s-0.5}(\bdy\Omega)}\Big]\,, \label{Hodge_elliptic_estimate2b}
\end{align}
where $\bdygrad \u$ is the tangential derivative on $\bdy \Omega$.
\end{theorem}
The inequalities (\ref{Hodge_elliptic_estimate1b}) and (\ref{Hodge_elliptic_estimate2b}) set in fractional-order
Sobolev spaces are fundamental to the analysis of Euler-type free-boundary problems. We remark that
$\bdy \Omega $ is assumed to be in a small tubular neighborhood of the normal bundle over $\bdy \mathcal{D} $; hence, there is height function
$h(x,t)$ such that each point on $\bdy \Omega$ is given by $x+ h(x)\text{\bf\emph{n}}(x)$, $x \in \bdy \mathcal{D} $, where $\text{\bf\emph{n}}$ is the outward-pointing unit normal to $\bdy\mathcal{D} $. The assumption that the distance between $\bdy \mathcal{D} $ and $ \bdy \Omega $ in the $H^{s+0.5}$-norm is less than $ \epsilon \ll1$ means that we assume that $\|h\|_{H^{s+0.5}(\bdy\D)} < \epsilon \ll 1$

\subsection{Outline of the paper} In Section \ref{sec::notation}, we introduce our notation as well as a number of elementary technical lemmas,
whose proofs we include (for completeness) in Appendix \ref{appendixA}. Section \ref{sec:vector-valued_elliptic_eq} is devoted to the analysis of
the vector-valued elliptic system (\ref{vector-valued_elliptic_eq}a) with mixed-type boundary conditions (\ref{vector-valued_elliptic_eq}b) and
(\ref{vector-valued_elliptic_eq}c), which is fundamental to the proof of our two main theorems; in particular, we prove Theorem \ref{thm:vector-valued_elliptic_eq_Sobolev_coeff} which establishes the elliptic estimate for (\ref{vector-valued_elliptic_eq}) when the coefficients are of Sobolev-class.
 As a corollary to this theorem, we state in Corollary
\ref{cor:scalar_elliptic_eq_Sobolev_coeff} the basic elliptic estimates for both the Dirichlet and Neumann problems, again with Sobolev class regularity.
Finally, for coefficients which are close to the identity, we give an improved estimate in Theorem \ref{thm_linear_est} for solutions
to (\ref{vector-valued_elliptic_eq}), which is linear in the highest
derivatives of the coefficient matrix. This latter theorem is essential for estimates in fractional-order Sobolev spaces via linear interpolation.

In Section \ref{sec:Hodge_elliptic_estimate}, we prove Theorem \ref{thm:main_thm2}, using the elliptic regularity theory developed for the elliptic system
 (\ref{vector-valued_elliptic_eq}). Then, in Section \ref{sec:vector_decomp}, we prove Theorem \ref{thm:main_thm1}. Our proof relies on some basic
 geometric identities involving the mean curvature of $\bdy\Omega$, together with the elliptic regularity theory established in Section \ref{sec:vector-valued_elliptic_eq}. Finally, in Section \ref{sec6}, we prove Theorem \ref{thm:main_thm3}.

\subsection{A brief history of prior results} In addition to the recent work of Amrouche \& Seloula \cite{AmSe2013} noted above, there have been many
other methods and results to study such elliptic systems on smooth domains. The elliptic system (\ref{div_curl}) can be viewed as a particular example
of the systems studied by
Agmon, Douglis \& Nirenberg \cite{AgDoNi1964}, wherein both Schauder-type estimates and $L^p$-estimates can be found.

In \cite{Wa1992}, 
von Wahl proved that if the normal or the tangential trace of a vector field vanishes, and for bounded or unbounded $ \Omega $, the inequality
$\|\nabla u\|_{L^p(\Omega)} \le C \big(\|\div u\|_{L^p(\Omega)} + \|\curl u\|_{L^p(\Omega)})$
holds when the first Betti number or the second Betti number, respectively, is equal to zero.

Vector potentials and the characterization of the kernel of problem (\ref{div_curl}) with boundary conditions (\ref{bc1}) or (\ref{bc2}) have been obtained
by Foias \& Temam \cite{FoTe1978}, Georgescu \cite{Ge1979}, Bendali, Dom{\'{\i}}nguez, \& Gallic \cite{BeDoGa1985},
Amrouche, Bernardi, Dauge \& Girault \cite{AmBeDaGi1998}, and Amrouche, Ciarlet \& Ciarlet Jr. \cite{AmCiCi2010}.

Amrouche \& Girault \cite{AmGi1994} derived the $L^p$-regularity theory of the steady Stokes equation by establishing the equivalency between the Sobolev space $W^{m,r}$ and the direct sum of $W^{m,r}$ by divergence-free vector fields and the gradients of $W^{m+1,r}$ functions.

Schwarz \cite{Sc1995}, studied the Hodge decomposition on manifolds with boundaries and showed that a differential $k$-forms can be written as the sum of an exact form, a coexact form, and a harmonic form.

Bolik \& von Wahl \cite{BoWa1997} derived $\mC^\alpha$-estimates of the gradient of a vector field whose curl, divergence, and normal or tangential traces are prescribed. Mitrea, Mitrea \& Pipher \cite{MiMiPi1997} studied the vector potential theory on non-smooth domains in {$\bbR^3$} with applications to electromagnetic scattering.

Buffa and Ciarlet Jr. \cite{BuCi2001a} and \cite{BuCi2001b} established the Hodge decomposition of tangential vector fields defined on polyhedron domains, and studied the tangential trace and tangential components of vectors belonging to the space $H(\curl,\Omega) := \big\{u\in L^2(\Omega;\bbR^3)\,\big|\, \curl u \in L^2(\Omega;\bbR^3)\big\}$.

In \cite{KoYa2009}, Kozono and Yanagisawa proved the decomposition of a divergence-free vector-field as the sum of the curl of a vector-field and a vector-field which is solenoidal, irrotational and has zero normal trace.

\section{Notation and Preliminary Results}\label{sec::notation}
The Einstein summation convention is used throughout the paper. In particular, repeated Latin indices are summed
from $1$ to $\n$, and repeated Greek indices are summed from $1$ to $\n-1$. For example, $f_i g_i = \sum\limits_{i=1}^\n f_ig_i$ and
$ f_\alpha g_\alpha = \sum\limits_{i=1}^{\n-1} f_\alpha g_\alpha $.   The gradient operator is denoted by $ \nabla = ( \p_1\,, ... \,, \p_\n)$.
Below, we shall also define various tangential derivative operators.

\subsection{$\mC^\rk$-domain}
We recall that a domain $\Omega \subseteq \bbR^\n$ is said to be of class $\mC^r$ if $\bdy\Omega$ is an $(\n-1)$-dimensional $\mC^r$-manifold; that is, there exists an open cover $\{\U_m\}_{m=1}^K \subseteq \bbR^\n$ of $\bdy\Omega$ and a collection of $\mC^r$-maps $\{\phi_m\}_{m=1}^K$ such that for each $1\le m\le K$,
$$
\phi_m: \U_m \cap \bdy\Omega \to \V_m \subseteq \bbR^{\n-1}
$$
is one-to-one, onto, and has a $\mC^r$-inverse map for some open subset $\V_m$ of $\bbR^{\n-1}$. For practical point of view, we have the following
\begin{proposition}\label{prop:Cr_domain_charts}
Let $\Omega\subseteq \bbR^\n$ be a $\mC^\rk$-domain for some $\rk \in \bbN$, and $\varepsilon > 0$ be given.
Then there exists a collection of open sets $\{\U_m\}_{m=0}^K$ with each $\U_m \subseteq  \bbR^\n$ , a collection of $\mC^\rk$-maps $\{\vartheta_m\}_{m=1}^K$, and
positive numbers $\{r_m\}_{m=1}^K$ such that
$$
\Omega \subseteq \bigcup\limits_{m=0}^K \U_m \quad\text{and}\quad \bdy\Omega \subseteq \bigcup\limits_{m=1}^K \U_m\,,
$$
and for each $1 \le m\le K$,
\begin{enumerate}
\item[\rm1.] $\vartheta_m: B(0, r_m) \to \U_m \text{ is a } \mC^\rk\text{-diffeomorphism} $;
\item[\rm2.] $\vartheta_m: B(0,r_m) \cap \{y_n=0\} \to \U_m \cap \p \Omega $;
\item[\rm3.] $\vartheta_m: B^+_m \equiv B(0,r_m) \cap \{y_\n>0\} \to \U_m \cap \Omega $;
\item[\rm4.] $\det (\nabla \vartheta_m) =1 $;
\item[\rm5.] $\|\nabla \vartheta_m - \text{\rm Id}\|_{L^\infty(B(0,r_m))} \le \varepsilon$.
\end{enumerate}
\end{proposition}
The proof of this proposition is given in  \ref{appendixA}.

\subsection{$H^s$-domain}\label{sec:Hs-domains}
In order to make our presentation self-contained, in this section, we collect a
number of useful technical lemmas. These lemmas are well-known when the domains are smooth, but we shall need these basic results for Sobolev class domains.
The proofs will be collected in Appendix \ref{appendixA}. For the remainder of this section, when not explicitly stated, $s$ will denote a real number, while $0\le \rk,
\ell$ will denote integers. We use the term domain to mean an open connected subset of $ \Rn$.

\begin{definition}\label{defn:Hs_domain}
Let $\Omega \subseteq \Rn$ be a bounded domain, and $s > \novertwo + 1$ be a real number. $\Omega$ is said to be an $H^s$-domain, or of class $H^s$, if there exists a bounded $\mC^\infty$-domain $\rO$ and a map $\psi$ such that $\psi:\cls{\rO} \to \cls{\Omega}$ is an $H^s$-diffeomorphism; that is,
\begin{enumerate}
\item $\psi:\cls{\rO} \to \cls{\Omega}$ is continuous;
\item $\psi:\rO \to \Omega$ is one-to-one and onto, with differentiable inverse map $\psi^{-1}: \Omega \to \rO$;
\item $\psi:\bdy\rO \to \bdy\Omega$ is one-to-one and onto, with differentiable inverse map $\psi^{-1}:\bdy\Omega \to \bdy\rO$;
\item $\psi\in H^{s}(\rO; \Omega)$ and $\psi ^{-1} \in H^{s}(\Omega ; \rO)$.
\end{enumerate}
By the trace theorem, $\psi|_{\bdy \rO} \in H^{s-0.5}(\bdy \rO; \bdy \Omega)$ and we shall often denote the value of this norm by $|\bdy \Omega|_{H^{s-0.5}}$.
\end{definition}

\begin{definition}
For $s > \novertwo + 1$, given a local chart $(\U, \vartheta)$ as defined in Proposition \ref{prop:Cr_domain_charts},
the induced metric in the local chart $(\U,\vartheta)$ is the $(0,2)$-tensor $g_{\alpha\beta}$ given by
$$
g_{\alpha\beta} = \dfrac{\p\vartheta}{\p y_\alpha} \cdot \dfrac{\p\vartheta}{\p y_\beta}\,,
$$
and the induced second-fundamental form in a local chart $(\U,\vartheta)$ is the $(0,2)$-tensor $b_{\alpha\beta}$ given by
$$
b_{\alpha\beta} = - \dfrac{\p^2 \vartheta}{\p y_\alpha\p y_\beta} \cdot (\bN \circ \vartheta^{-1})\,,
$$
where $\bN$ is the outward-pointing unit normal to $\bdy\Omega$.
\end{definition}

\begin{definition}[Tangential gradient and surface divergence operators]\label{defn:tangential_derivative}
For $s > \novertwo + 1$,\vspace{.1cm} let $\Omega \subseteq \Rn$ be a bounded $H^s$-domain.
We let $\bdygrad$ denote the \underline{tangential gradient of a function on $\bdy\Omega$}.
If $\varphi: \bdy\Omega \to \bbR$ is differentiable, then in local chart $(\U,\vartheta)$, $\bdygrad \varphi$ is given by
$$
(\bdygrad \varphi)\circ \vartheta = g^{\alpha\beta} \dfrac{\p (\varphi\circ \vartheta)}{\p y_\alpha} \dfrac{\p \vartheta}{\p y_\beta}\,,
$$
where $[g^{\alpha\beta}]$ is the inverse matrix of the induced metric $[g_{\alpha\beta}]$, and $\Big\{{\dfrac{\p \vartheta}{\p y_\beta}}\Big\}_{\beta=1}^2$ are tangent vectors to $\bdy \Omega$.

We define the surface divergence operator $\bdydiv $ to be the formal adjoint of $-\bdygrad$;
if $\u$ is a tangent vector field on $\bdy\Omega$ so that $\u \cdot \bN = 0$ on $\bdy\Omega$, then
$$
- \int_{\bdy\Omega} \u \cdot \bdygrad \varphi \,dS = \int_{\bdy\Omega} \varphi\, \bdydiv  \u \,dS \qquad\Forall \varphi\in H^1(\bdy\Omega)\,.
$$
In a local chart $(\U,\vartheta)$,
$$
(\bdydiv  \u)\circ \vartheta = \dfrac{1}{\sqrt{\rg}} \dfrac{\p}{\p y_\alpha} \Big[\sqrt{\rg} g^{\alpha\beta} \big((\u \circ \vartheta) \cdot \dfrac{\p \vartheta}{\p y_\beta}\big) \Big]\,,
$$
where $\rg = \det(g)$ is the determinant of the induced metric $[g_{\alpha\beta}]$.
\end{definition}

\begin{definition}[Tangential projection of a vector field onto $\bdy\Omega$] With $\bN$ denote the outward unit normal vector field to
$\bdy\Omega$ and $\v:\bdy\Omega \to \bbR^\n$, we define
$\rP_{\bN^\perp}:\Rn \to \Rn$ to be the tangential projection operator given by
\begin{equation}\label{defn:tan_proj}
\rP_{\bN^\perp}(\v) = \v - (\v\cdot \bN)\, \bN= \big(\operatorname{Id}  - \bN\otimes \bN\big) \v \,.
\end{equation}
We will also write $\underline \v$ for $\rP_{\bN^\perp}(\v) $.
\end{definition}

\begin{definition}[Various tangential derivatives]
We let $\u:\bdy\Omega \to \bbR^\n$ and $\bfw:\bdy\Omega \to \bbR^\n$ denote vector-valued functions, and let $\underline\bfw$ be given by {\rm(\ref{defn:tan_proj})}.
\begin{enumerate}
\item[\rm1.] $\bdygrad_{\underline\bfw} \u$ denotes the directional derivatives of $\u$ in
    the direction $\underline\bfw$. In a local chart $(\U,\vartheta)$,
    $$
    \big[\bdygrad_{\underline\bfw} \u\big] \circ \vartheta = g^{\alpha\beta} \dfrac{\p (\u \circ \vartheta)}{\p y_\alpha} \dfrac{\p \vartheta^j}{\p y_\beta} (\bfw^j\circ \vartheta) = g^{\alpha\beta} \Big[\dfrac{\p \vartheta}{\p y_\beta} \cdot (\bfw\circ \vartheta)\Big] \dfrac{\p (\u \circ \vartheta)}{\p y_\alpha}\,.
    $$

\item[\rm2.] $\bdygrad \u = ( \bdygrad \u^1\,, ...\,, \bdygrad \u^\n)$, and  $\bdygrad \u \cdot \bfw =\sum\limits_{i=1}^\n \bfw^i\, (\bdygrad \u^i )\, $, so that $\bdygrad \u \cdot \bfw $ is a vector in the tangent space of $\bdy\Omega$. In a local chart $(\U,\vartheta)$,
    $$
    (\bdygrad \u \cdot \bfw) \circ \vartheta = g^{\alpha\beta} \dfrac{\p (\u^j\circ \vartheta)}{\p y_\alpha} \dfrac{\p \vartheta^i}{\p y_\beta} (\bfw^j\circ \vartheta) = g^{\alpha\beta} \Big[\dfrac{\p (\u\circ \vartheta)}{\p y_\alpha} \cdot (\bfw\circ \vartheta)\Big] \dfrac{\p \vartheta^i}{\p y_\beta}\,.
    $$
    The product rule holds:
    $$
    \bdygrad \u \cdot \bfw = \bdygrad (\u \cdot \bfw) - \bdygrad \bfw \cdot \u\,.
    $$

\item[\rm3.] $\bdygrad \u \times \bfw$  is defined to be the linear map satisfying (at each point of $\bdy\Omega$)
    $$
    (\bdygrad \u \times \bfw) \v = ( \bdygrad _{\underline\v }\u) \times \bfw
    $$
    for all $\v\in \bbR^\n$. In a local chart $(\U,\vartheta)$,
    \begin{align*}
    (\bdygrad \u \times \bfw)_{ij} \circ \vartheta = \varepsilon_{irs} g^{\alpha\beta} \frac{\p \vartheta^j}{\p y_\beta} \frac{\p (\u^r\circ \vartheta)}{\p y_\alpha} (\bfw^s \circ \vartheta)\,.
    \end{align*}
\end{enumerate}
\end{definition}

\subsection{Basic inequalities}\label{sec:useful_ineq}
We now state some basic inequalities, that we use throughout the paper.
\begin{proposition}\label{prop:HkHl_product}
For $\rk > \novertwo$ and $0 \le \ell\le \rk$,
let $\Omega\subseteq \Rn$ be a bounded $\mC^\infty$-domain. Then for all $\sigma \in \big(0,\smallexp{$\displaystyle{}\frac{1}{4}$}\big)$,\vspace{.1cm} there exists a constant $C_\sigma$ depending on $\sigma$ such that for all $f\in H^\rk(\Omega)$ and $g\in H^\ell(\Omega)$,
\begin{equation}\label{commutator_estimate_elliptic_est_temp}
\sum_{j=1}^\ell \|\nabla^j f \nabla^{\ell - j} g\|_{L^2(\Omega)} \le C_\sigma \|f\|_{H^\rk(\Omega)} \|g\|_{H^{\ell-\sigma}(\Omega)} \,.
\end{equation}
Moreover, for some generic constant $C>0$,
\begin{equation}\label{HkHl_product}
\|fg\|_{H^\ell(\Omega)} \le C \|f\|_{H^\rk(\Omega)} \|g\|_{H^\ell(\Omega)} \qquad\Forall f\in H^\rk(\Omega), g\in H^\ell(\Omega)\,.
\end{equation}
\end{proposition}

\begin{remark}
Suppose that $s > \novertwo$\vspace{.1cm} and $0\le r\le s$ for some real numbers $r$ and $s$. Then there exists a generic constant $C_s > 0$ such that
\begin{equation}\label{product_Hr_est_Rn}
\|fg\|_{H^r(\Rn)} \le C_s \|f\|_{H^s(\Rn)} \|g\|_{H^r(\Rn)}\qquad\Forall f\in H^s(\Rn), g\in H^r(\Rn)\,.
\end{equation}
By the Sobolev extension argument, we also conclude that
\begin{equation}\label{product_Hr_est}
\|fg\|_{H^r(\Omega)} \le C_s \|f\|_{H^s(\Omega)} \|g\|_{H^r(\Omega)} \qquad\Forall f\in H^s(\Omega), g\in H^r(\Omega)
\end{equation}
if $\Omega$ is a bounded $\mC^\infty$-domain.
\end{remark}

The following corollary is a direct consequence of Proposition \ref{prop:HkHl_product} since by Leibniz's rule,
$$
\comm{\nabla^\ell}{f} g = \sum_{j=1}^\ell \footnoteexp{$\displaystyle{}{{\ell}\choose{j}}$} \nabla^j f \nabla^{\ell-j} g\,.
$$
\begin{corollary}\label{cor:useful_lemma_with_Sobolev_class_coeff}
Let $\Omega\subseteq \Rn$ be a bounded $\mC^\infty$-domain for some integer $\rk > \novertwo$\,.
\begin{enumerate}
\item[\rm 1.] Suppose that $\supp(g) \cptsubset \Omega$. Then for $0 < \sigma < \smallexp{$\displaystyle{}\frac{1}{4}$}$ and $1\le \ell \le \rk+1$,
  \begin{equation}\label{commutator_estimate_elliptic_est1}
   \big\|\comm{\nabla^\ell}{f} g\big\|_{L^2(\Omega)} \le C_\sigma \|f\|_{H^{\max\{\rk,\ell\}}(\Omega)} \|g\|_{H^{\ell-\sigma}(\Omega)}\,,
  \end{equation}
  where $\comm{\nabla^\ell}{f} g = \nabla^\ell (f g) - f \nabla^\ell g$.

\item[\rm 2.] Suppose that $\zeta$ is a smooth cut-off function such that
\begin{enumerate}
\item[\rm(a)] $\supp(\zeta) \subseteq \U$;
\item[\rm(b)] there exists a $\mC^\infty$-diffeomorphism $\vartheta: B(0,1) \to \U$ satisfying
\begin{enumerate}
\item[\rm(i)] $\vartheta: B^+(0,1) \equiv B(0,1)\cap \{y_\n > 0\} \to \U \cap \Omega$;
\item[\rm(ii)] $\vartheta: \{y_\n = 0\} \to \bdy\Omega$.
\end{enumerate}
\end{enumerate}
Define $F = (\zeta f)\circ \vartheta$ and $G = (\zeta g)\circ \vartheta$. Then for $0< \sigma < 1/4$, $1\le \ell \le \rk+1$,
\begin{equation}\label{commutator_estimate_elliptic_est2}
\big\|\comm{\bp^\ell}{F} G\big\|_{L^2(B^+(0,r))} \le C_\sigma \|f\|_{H^{\max\{\rk,\ell\}}(\Omega)} \|g\|_{H^{\ell-\sigma}(\Omega)}\,,
\end{equation}
where $\comm{\bp^\ell}{F} G = \bp^\ell(F G) - F \bp^\ell G$ and $\bp = (\p_{y_1}, \cdots, \p_{y_{\n-1}})$ denotes the tangential gradient.
\end{enumerate}
\end{corollary}

The following two corollaries are direct consequences of Proposition \ref{prop:HkHl_product}, and are the foundation of the study of inequalities on $H^s$-domains. The proof of these two corollaries can also be found in Appendix \ref{appendixA}.
\begin{corollary}\label{cor:JA_est}
Let $\rO \subseteq \Rn$ be a bounded $\mC^\infty$-domain, and $\psi:\cls{\rO} \to \cls{\Omega} \subseteq \Rn$ be a $H^{\rk+1}$-diffeomorphism for some integer $\rk > \novertwo$\,. If $J = \det(\nabla \psi)$ and $A = (\nabla \psi)^{-1}$, then
\begin{subequations}\label{JA_est}
\begin{align}
\|J\|_{H^\rk(\rO)} &\le C\big(\|\nabla \psi\|_{H^\rk(\rO)}\big)\,,\\
\|A\|_{H^\rk(\rO)} &\le C\big(\|1/J\|_{L^\infty(\rO)}, \|\nabla \psi\|_{H^\rk(\rO)}\big)\,.
\end{align}
\end{subequations}
\end{corollary}

\begin{corollary}\label{cor:f_comp_psi}
Let $\rO\subseteq \Rn$ be a bounded $\mC^\infty$-domain, and $\psi: \cls{\rO} \to \cls{\Omega}\subseteq \Rn$ be an $H^{\rk+1}$-diffeomorphism for some integer $\rk > \novertwo$\,. Then for all $\ell \le \rk + 1$,
\begin{subequations}\label{f_comp_psi_est}
\begin{alignat}{2}
\|f\|_{H^\ell(\Omega)} &\le C(\|\nabla \psi\|_{H^\rk(\rO)}) \|f\circ\psi\|_{H^\ell(\rO)} \qquad&&\Forall f\in H^\ell(\Omega)\,,\label{f_comp_psi_est1}\\
\|f\circ \psi\|_{H^\ell(\rO)} &\le C(\|\nabla \psi\|_{H^\rk(\rO)}) \|f\|_{H^\ell(\Omega)} \qquad&&\Forall f\in H^\ell(\Omega)\,.\label{f_comp_psi_est2}
\end{alignat}
\end{subequations}
\end{corollary}

\begin{remark}
Note that Corollary \ref{cor:f_comp_psi} implies that the interpolation inequalities on a Sobolev class domain are still valid if the domain is bounded and has $H^{\rk+1}$ regularity for some integer $\rk > \novertwo$. For example,
\begin{align*}
\|f\|_{H^{0.5}(\Omega)} &\le C(|\bdy\Omega|_{H^{\rk+0.5}}) \|f\circ \psi\|_{H^{0.5}(\rO)} \le C(|\bdy\Omega\|_{H^{\rk+0.5}}) \|f\circ\psi\|^\frac{1}{2}_{L^2(\rO)} \|f\circ \psi\|^\frac{1}{2}_{H^1(\rO)} \\
&\le C(|\bdy\Omega|_{H^{\rk+0.5}}) \|f\|^\frac{1}{2}_{L^2(\Omega)} \|f\|^\frac{1}{2}_{H^1(\Omega)}\,.
\end{align*}
\end{remark}
Similar arguments can be applied to prove the following theorem whose proof we omit.
\begin{theorem}\label{thm:HkHl_product2}
Let $\Omega\subseteq \Rn$ be a bounded domain\vspace{.1cm} of class $H^{\rk+1}$ for some integer $\rk > \novertwo$. Then for all $\sigma \in \big(0,\dfrac{1}{4}\big)$,\vspace{.1cm} there exists constant $C_\sigma$ depending on $|\bdy\Omega|_{H^{\rk+0.5}}$ and $\sigma$ such that for all $0 \le \ell\le \rk+1$, $f\in H^{\max\{\rk,\ell\}}(\Omega)$ and $g\in H^\ell(\Omega)$,
\begin{equation}\label{DjfDljg_est}
\sum_{j=1}^\ell \|\nabla^j f \nabla^{\ell - j} g\|_{L^2(\Omega)} \le C_\sigma \|f\|_{H^{\max\{\rk,\ell\}}(\Omega)} \|g\|_{H^{\ell-\sigma}(\Omega)}\,.
\end{equation}
Moreover, for some constant generic $C$ depending on $|\bdy\Omega|_{H^{\rk+0.5}}$,
\begin{equation}\label{HkHlproduct}
\|fg\|_{H^\ell(\Omega)} \le C \|f\|_{H^{\max\{\rk,\ell\}}(\Omega)} \|g\|_{H^\ell(\Omega)}\qquad\Forall f\in H^{\max\{\rk,\ell\}}(\Omega), g\in H^\ell(\Omega)\,.
\end{equation}
\end{theorem}

By using Theorem \ref{thm:HkHl_product2}, we can easily establish the following
\begin{theorem}\label{thm:useful_lemma_with_Sobolev_class_coeff2}
Let $\Omega\subseteq \Rn$ be a bounded $H^{\rk+1}$-domain for some integer $\rk > \novertwo$\,. Then for each integers $\ell \in \{0,1\}\cup \big(\novertwo, \rk+1\big]$, there exists a generic constant $C = C(|\bdy\Omega|_{H^{\rk+0.5}})$ such that
\begin{equation}\label{Hk_bdy_by_Linf_Hk}
\begin{array}{l}
\|fg\|_{H^\ell(\Omega)} \le C \Big[\|f\|_{L^\infty(\Omega)} \|g\|_{H^\ell(\Omega)} + \|f\|_{H^\ell(\Omega)} \|g\|_{L^\infty(\Omega)}\Big] \qquad \Forall f,g \in H^\ell(\Omega) \cap L^\infty(\Omega)\,.
\end{array}
\end{equation}
\end{theorem}

\subsection{\Poincare-type inequalities} We will make use of the following
 \Poincare-type inequalities, whose proofs are similar to the proof of the standard \Poincare\ inequality.
 \begin{lemma}
Let $\Omega\subseteq \bbR^3$ be a bounded smooth domain with outward-pointing unit normal $\bN$.  We set
\begin{align*}
H^1_\tau(\Omega) &\equiv \big\{ u:\Omega \to \bbR^3\hspace{1pt}\big|\hspace{1pt} u\in H^1(\Omega)\,, u\times \bN = 0 \text{ on }\bdy\Omega\big\}\,,\\
H^1_n(\Omega) &\equiv \big\{ u:\Omega \to \bbR^3\hspace{1pt}\big|\hspace{1pt} u\in H^1(\Omega)\,, u\cdot \bN = 0 \text{ on }\bdy\Omega\big\} \,.
\end{align*}
Then
\begin{equation}\label{vectorPoincareineq}
\|\u\|_{L^2(\Omega)} \le C \|\nabla \u\|_{L^2(\Omega)} \qquad\Forall \u \in H^1_\tau(\Omega)\,,
\end{equation}
and
\begin{equation}\label{vectorPoincareineq1}
\|\u\|_{L^2(\Omega)} \le C \|\nabla \u\|_{L^2(\Omega)} \qquad\Forall \u \in H^1_n(\Omega)\,.
\end{equation}
\end{lemma}

\subsection{Commutation with mollifiers}\label{sec:convolution}
Our proof of elliptic regularity relies on a mollification procedure (rather than the use of difference quotients).
\begin{definition}[Standard mollifiers]\label{defn:mollifier}
Let $\eta(x) = C \exp\big(\dfrac{1}{|x|^2-1}\big)$ for $|x|<1$ and $\eta$ vanishes outside the unit ball, where $C$ is chosen so that $\|\eta\|_{L^1(\bbR^\n)} = 1$. The standard mollifier $\eta_\epsilon$ is defined by
$$
\eta_\epsilon(x) = \dfrac{1}{\epsilon^\n} \eta\big(\dfrac{x}{\epsilon}\big)\,.
$$
\end{definition}

We will make use of the following
\begin{lemma}\label{lem:commutator_estimates1}
For $f \in W^{1, \infty}(\Omega)$ and $g\in L^2(\Omega)$ with compact support, there is a generic constant $C$ independent of $\epsilon$ such that
\begin{align}
\big\|\nabla \big( \comm{\eta_\epsilon\convolution }{f} g\big) \big\|_{L^2(\Omega)} &= \big\|\nabla \big[\eta_\epsilon\convolution (f g) - f \eta_\epsilon\convolution g\big]\big\|_{L^2(\Omega)} \nonumber\\
&\le C \|f\|_{W^{1,\infty}(\Omega)} \|g\|_{L^2(\Omega)} \label{commutator_estimates1}
\end{align}
for all $\text{\rm $0 < \epsilon < \min\big\{\text{dist}\big(\bdy\Omega,\supp(f)\big),\text{dist}\big(\bdy\Omega,\supp(g)\big)\big\}$}$.
\end{lemma}

Since we are dealing with problems on domains with boundaries, we make use of the {\it horizontal convolution-by-layers} operator, introduced in \cite{CoSh2007}. We define the horizontal convolution-by-layers operator $ \Lambda_\epsilon $ as follows:
$$
\Lambda_\epsilon f ( x_h, x_\n) = \int_{\bbR^{\n-1}} \rho_\epsilon (x_h -y_h) f(y_h, x_\n) dy_h \ \text{ for } f ( \cdot , x_\n) \in L^1( \bbR^{\n-1} ) \,,
$$
where $\rho_\epsilon(x_h) = \dfrac{1}{\epsilon^{\n-1}} \rho\big(\dfrac{x_h}{\epsilon}\big)$, and
$\rho \in C^\infty _0( \bbR^{\n-1})$ is given by $\rho(x) = C \exp \left( \dfrac{1}{|x|^2-1}\right)$ if $|x| <1$ and $\rho(x)=0$ if $|x_h|\ge 1$. The
constant $C$ is chosen so that $\smallexp{$\displaystyle{}\int_{ \bbR^{\n-1} }$} \rho \,dx =1$. It follows that
for $\epsilon >0$, $ 0 \le \rho_ {\epsilon } \in C^\infty_0( \bbR^{\n-1})$ with $\spt (\rho_{\epsilon }) \subset \overline{B(0, \epsilon )} $. (Here, $\spt$ stands for support.)

It should be clear that $ \Lambda_\epsilon $ smooths functions defined on $ \bbR^\n$ along all horizontal subspaces,
but does not smooth functions in the vertical $x_\n$-direction. On the other hand, we can restrict the operator $ \Lambda_\epsilon $
to act on functions $f: \bbR^{\n-1} \to \bbR $ as well, in which case $ \Lambda_\epsilon $ becomes the
usual mollification operator. Associated to $\Lambda_\epsilon$, we need the following

\begin{lemma}\label{lem:commutator_estimates2}
For $f \in W^{1, \infty}(\bbR^\n_+)$ and $g\in L^2(\bbR^\n_+)$, there is a generic constant $C$ independent of $\epsilon$ such that
\begin{equation}\label{commutator_estimates2}
\big\|\bp \big( \comm{\Lambda_\epsilon}{f} g\big)\big\|_{L^2(\bbR^\n_+)} = \big\|\bp\big[\Lambda_\epsilon (f g) - f \Lambda_\epsilon g\big]\big\|_{L^2(\bbR^\n_+)} \le C \big\|f\big\|_{W^{1,\infty}(\bbR^\n_+)} \big\|g\big\|_{L^2(\bbR^\n_+)}
\end{equation}
for all $\epsilon>0$.
\end{lemma}

\subsection{The Piola Identity}
\begin{lemma}[Piola identity]\label{thm:Piola_id}
Let $\psi:\Omega\subseteq \bbR^\n \to \bbR^\n$ be a diffeomorphism, and $[a_{ij}]_{\n\times \n}$ be the cofactor matrix of $\nabla \psi$. Then
\begin{align}
 \frac{\p}{\p x_j}\, a_{ji} = 0. \label{Piola_id}
\end{align}
\end{lemma}
The proof can be found in \cite{Evans_PDE}.

\section{Vector-Valued Elliptic Equations}\label{sec:vector-valued_elliptic_eq}
Let $\Omega\subseteq \Rn$ denote a bounded domain whose regularity will be specified below. In this section, we study a vector-valued elliptic equation
\begin{subequations}\label{vector-valued_elliptic_eq}
\begin{alignat}{2}
(L \u)^i = \u^i - \frac{\p}{\p x_j} \Big(a^{jk} \frac{\p \u^i}{\p x_k}\Big) &= \f^i \qquad&&\text{in}\quad\Omega\,,
\end{alignat}
\end{subequations}
with special types of boundary conditions, where $\u = (\u_1,\cdots,u_\n)$ and $f = (f_1,\cdots,f_\n)$ are vector-valued functions, and $a^{jk}$ is a two-tensor satisfying the positivity condition
\begin{equation}\label{positivity_of_a}
a^{jk} \xi_j \xi_k \ge \lambda |\xi|^2 \qquad\Forall \xi,\eta\in \Rn
\end{equation}
for some $\lambda > 0$. Since $\u\in \Rn$, $\n$ boundary conditions are needed to solve the system uniquely.
We consider a mixed-type boundary condition given by
$$
\begin{array}{crll}
& \displaystyle{} \u \cdot \bfw \hspace{-7pt}&= 0 \qquad\text{on}\quad\bdy\Omega\,, \hspace{70pt} \hfill{\rm(\ref{vector-valued_elliptic_eq}b)}\vspace{.1cm} \\
\hspace{115pt} & \displaystyle{} \rP_{\bfw^\perp}\Big(a^{jk} \frac{\p \u}{\p x_k} \bN_j - \g\Big) \hspace{-7pt}&= {\bf 0} \qquad\text{on}\quad\bdy\Omega\,, \hspace{108pt}\rm(\ref{vector-valued_elliptic_eq}c)
\end{array}
$$
where $\bfw$ is a uniformly continuous vector field defined in a neighborhood of $\bdy\Omega$ which vanishes nowhere on $\bdy\Omega$, $\bN$ is the outward-pointing unit normal to $\bdy\Omega$, $\g$ is a vector-valued function defined on $\bdy\Omega$, and
$\rP_{\bfw^\perp}:\Rn \to \Rn$ is the projection map given by
\begin{equation}\label{defn:P_tan}
\rP_{\bfw^\perp}(\v) = \v - \frac{(\v\cdot \bfw)}{|\bfw|^2}\, \bfw = \big(\id - \frac{\bfw\otimes \bfw}{|\bfw|^2}\big) \v \,.
\end{equation}
The condition (\ref{vector-valued_elliptic_eq}b) specifies the component of the vector $\u$ in the direction of $\bfw$, while the condition (\ref{vector-valued_elliptic_eq}c) specifies the $\n-1$ components of the  Neumann derivative
 \smallexp{$\displaystyle{}a^{jk} \frac{\p \u^i}{\p x_k} \bN_j$}. \vspace{.2cm}

Integration by parts with respect to $x_j$ leads to the following identity:
\begin{align*}
- \int_\Omega \frac{\p}{\p x_j} \Big(a^{jk} \frac{\p \u^i}{\p x_k}\Big) \Varphi^i dx &= \int_\Omega a^{jk} \frac{\p \u^i}{\p x_k} \frac{\p \Varphi^i}{\p x_j} dx - \int_{\bdy\Omega} a^{jk} \frac{\p \u^i}{\p x_k} \bN_j \Varphi^i dx \\
&= \int_\Omega a^{jk} \frac{\p \u^i}{\p x_k} \frac{\p \Varphi^i}{\p x_j} dx - \int_{\bdy\Omega} \Big[\g^i + a^{jk} \frac{\p \u^r}{\p x_k} \bN_j \frac{\bfw_r \bfw_i}{|\bfw|^2}\Big] \Varphi^i dx \,,
\end{align*}
which, in turn,  motivates the following
\begin{definition}
Let $\V = \big\{\v \in H^1(\Omega)\,\big|\, \v\cdot \bfw = 0 \text{ on } \bdy\Omega\big\}$. A function $\u\in \V$ is called a weak solution to {\rm(\ref{vector-valued_elliptic_eq})} if
\begin{equation}\label{vector-valued_elliptic_weak_form}
(\u,\Varphi)_{L^2(\Omega)} + \int_\Omega \hspace{-1pt}a^{jk} \frac{\p \u^i}{\p x_k} \frac{\p \Varphi^i}{\p x_j}\, dx = \hspace{-1pt} (\f,\Varphi)_{L^2(\Omega)} + \langle \g, \Varphi\rangle_{\bdy\Omega} \quad\Forall \Varphi \in \V\,,
\end{equation}
where $\langle \cdot,\cdot\rangle_{\bdy\Omega}$ denotes the duality pairing between distributions in $H^{-\frac{1}{2}}(\bdy\Omega)$ and functions in $H^\frac{1}{2}(\bdy\Omega)$.
\end{definition}

With the help of the Lax-Milgram theorem it is easy to conclude the following
\begin{theorem}[Weak solutions]\label{thm:weak_existence_vector-valued_elliptic_prob}
Suppose that $a^{jk} \in L^\infty(\Omega)$ satisfies the positivity condition {\rm(\ref{positivity_of_a})}, and $\bfw$ is a uniformly continuous vector field defined in a neighborhood of $\bdy\Omega$ which vanishes nowhere on $\bdy\Omega$. Then for all $\f\in L^2(\Omega)$ and $\g\in H^{-0.5}(\bdy\Omega)$, there exists a unique weak solution to {\rm(\ref{vector-valued_elliptic_eq})} in $\V$, and the weak solution $\u$ satisfies
\begin{equation}\label{vector-valued_elliptic_weak_est}
\|\u\|_{H^1(\Omega)} \le C \Big[\|\f\|_{L^2(\Omega)} + \|\g\|_{H^{-0.5}(\bdy\Omega)} \Big]\,.
\end{equation}
\end{theorem}

\begin{remark}
Let $\u \in H^2(\Omega)\cap \V$ be a weak solution to {\rm(\ref{vector-valued_elliptic_eq})}. Integrating by parts with respect to $x_j$, we find that
\begin{align*}
\smallexp{$\displaystyle{}\int_{\Omega} \Big(\u^i - \dfrac{\p}{\p x_j} \big(a^{jk}\dfrac{\p \u^i}{\p x_k}\big) - \f^i \Big) \Varphi^i dx + \int_{\bdy\Omega} \big(a^{jk} \frac{\p \u^i}{\p x_k} \bN_j - \g\big)\Varphi^i dS = 0\quad\Forall \Varphi \in \V\,.$}
\end{align*}
Since $\Varphi \hspace{-.5pt}\cdot\hspace{-.5pt} \bfw = 0$ on $\bdy\Omega$, the Neumann-type boundary condition {\rm(\ref{vector-valued_elliptic_eq}c)} is thus shown to hold.
\end{remark}

We next establish the regularity theory for weak solutions satisfying (\ref{vector-valued_elliptic_weak_form}).
\begin{definition}[Partition-of-unity subordinate to an open cover]\label{partition_of_unity}
For a given collection of open sets $\{\U_m\}_{m=1}^K \subseteq \bbR^\n$, there exists a partition of unity
$\{\zeta_m\}_{m=1}^K$ subordinate to $\{\U_m\}_{m=1}^K$ such that $\sqrt{\zeta_m} \in \mC^\infty_\cptspt(\bbR^\n)$ for all
$1\le m\le K$. In fact, if $\{\xi_m\}_{m=1}^K$ is a smooth partition-of-unity subordinate to $\{\U_m\}_{m=1}^K$, by defining
$\{\zeta_m\}_{m=1}^K$ by
$$
\zeta_m = \frac{\xi^2_m}{\sum_{j=1}^K \xi_j^2}\,,
$$
then $0\le \zeta_m \le 1$, $\supp(\zeta_m) \subseteq \U_m$,  $\sqrt{\zeta_m} \in \mC^\infty_\cptspt(\bbR^\n)$ for all $1\le m\le K$,  and that $\sum\limits_{m=1}^K \zeta_m = 1$.
\end{definition}

\subsection{The case that the coefficients \texorpdfstring{$a^{jk}$}{aʲᵏ} are of class \texorpdfstring{$\mC^k$}{Cᵏ}}
Now we study the regularity of the weak solution $\u$ to (\ref{vector-valued_elliptic_eq}) when the coefficients $a^{jk}$ are of Sobolev class $H^\rk$, $\rk\in \bbN$, and the domain $\Omega$ is $\mC^{\rk+1}$. To do so, we shall first establish this regularity result under the more restrictive assumption that  the coefficients $a^{jk}$ are in $\mC^\rk(\cls{\Omega})$.

\begin{theorem}[Regularity for the case that $a^{jk} \in\mC^\rk(\cls{\Omega})$ and $\Omega\in \mC^{\rk+1}$] \label{thm:vector-valued_elliptic_regularity}
Suppose that for some $\rk\in \bbN$, $\Omega\subseteq \Rn$ is a bounded $\mC^{\rk+1}$-domain, $a^{jk} \in \mC^\rk(\cls{\Omega})$
satisfies the positivity condition {\rm(\ref{positivity_of_a})},
$\bfw$ is $\mC^{\rk+1}$ in an open neighborhood of $\bdy\Omega$, and $|\bfw| > 0$
on $\bdy\Omega$. Then for all $\f\in H^{\rk-1}(\Omega)$ and $\g \in H^{\rk-0.5}(\bdy\Omega)$, the weak solution $\u$ to {\rm(\ref{vector-valued_elliptic_eq})} in fact belongs to $H^{\rk+1}(\Omega)$, and satisfies
\begin{equation}\label{vector-valued_elliptic_regularity}
\|\u\|_{H^{\rk+1}(\Omega)} \le C \Big[\|\f\|_{H^{\rk-1}(\Omega)} + \|\g\|_{H^{\rk-0.5}(\bdy\Omega)} \Big]
\end{equation}
for some constant $C$ depending on $\|a\|_{\mC^\rk(\Omega)}$, $\|\bfw\|_{\mC^{\rk+1}(\Omega)}$ and $|\bdy\Omega|_{\mC^{\rk+1}}$.
\end{theorem}
\begin{proof}
Our goal is to establish the regularity theory for weak solutions $\u\in \V$ to (\ref{vector-valued_elliptic_regularity}). We prove this by induction and divide the proof into several steps as follows:

\vspace{.1 in}
\noindent {\bf Step 1: (Interior regularity)} Suppose that $\u\in H^\ell(\Omega)$ for some $1\le \ell\le \rk$.
Let $\chi$ be a smooth function with $\supp(\chi) \cptsubset \Omega$, and $0< \epsilon < \text{dist}(\supp(\chi), \bdy \Omega)$.
We define
$$
\Varphi = (-1)^\ell \chi \big[\eta_\epsilon \cvl \nabla^{2\ell} \big(\eta_\epsilon \cvl (\chi \u)\big)\big]
$$
with no summation over the index $\ell$, and we let $\{\eta_\epsilon\}_{\epsilon>0}$ denote a sequence of standard mollifiers given in Definition \ref{defn:mollifier}. We note that this choice of test function $\Varphi$ is in $\V$, and can hence be used in the variational formulation
(\ref{vector-valued_elliptic_weak_form}). First, we see that
\begin{equation}\label{uphi_inner_product}
(\u,\Varphi)_{L^2(\Omega)} = \big\|\nabla^\ell \eta_\epsilon \cvl (\chi \u)\big\|^2_{L^2(\Omega)}\,.
\end{equation}
Since convolution is self-adoint,  the product rule shows that
\begin{align*}
\int_\Omega a^{jk} \frac{\p \u^i}{\p x_k} \frac{\p \Varphi^i}{\p x_j}\, dx &= \int_\Omega \nabla \big[\eta_\epsilon \cvl \big(a^{jk} \nabla^{\ell - 1} (\chi \u^i),_k \big) \big] \nabla^\ell \big[\eta_\epsilon \cvl (\chi \u^i),_j \big] dx \nonumber\\
&\quad + \sum_{r=0}^{\ell-2} \smallexp{$\displaystyle{}{{\ell\hspace{-1pt}-\hspace{-1pt}1}\choose{r}}$} \int_\Omega \nabla \big[\eta_\epsilon \cvl \big((\nabla^{\ell - 1 - r} a^{jk}) \nabla^r (\chi \u^i),_k\big)\big] \nabla^\ell \big[\eta_\epsilon \cvl (\chi \u^i),_j \big] dx \nonumber\\
&\quad - \int_\Omega \nabla^{\ell} \big[\eta_\epsilon \cvl \big( a^{jk} \u^i \chi,_k \big)\big] \nabla^{\ell} \big[\eta_\epsilon \cvl (\chi \u^i),_j \big] dx \nonumber\\
&\quad - \int_\Omega \nabla^{\ell-1} \big[\eta_\epsilon \cvl \big(a^{jk} \chi,_j \u^i_{,k}\big) \big] \nabla^{\ell+1} \big[\eta_\epsilon \cvl (\chi \u^i) \big] dx \,. 
\end{align*}
Using the commutator notation $\comm{A}{B} f = A (Bf) - B(Af)$, the first term on the right-hand side of the identity above can be rewritten as
\begin{align*}
& \int_\Omega \nabla \big[\eta_\epsilon \cvl \big(a^{jk} \nabla^{\ell - 1} (\chi \u^i),_k \big)\big] \nabla^\ell \big[\eta_\epsilon \cvl (\chi \u^i),_j \big] dx \\
&\qquad\quad = \int_\Omega \nabla \big[a^{jk} \eta_\epsilon \cvl \nabla^{\ell-1} (\chi \u^i),_k\big] \nabla^\ell \big[\eta_\epsilon \cvl (\chi \u^i),_j \big] dx \\
&\qquad\qquad + \int_\Omega \big[\nabla \bigcomm{\eta_\epsilon\cvl}{a^{jk}} \nabla^{\ell-1} (\chi \u^i),_k\big] \nabla^\ell \big[\eta_\epsilon \cvl (\chi \u^i),_j \big] dx \\
&\qquad\quad = \int_\Omega a^{jk} \nabla^\ell \big[\eta_\epsilon\cvl (\chi \u^i),_k\big] \nabla^\ell \big[\eta_\epsilon \cvl (\chi \u^i),_j \big] dx \\
&\qquad\qquad + \int_\Omega (\nabla a^{jk}) \nabla^{\ell-1} \big[\eta_\epsilon\cvl (\chi \u^i),_k\big] \nabla^\ell \big[\eta_\epsilon \cvl (\chi \u^i),_j \big] dx \\
&\qquad\qquad + \int_\Omega \big[\nabla \bigcomm{\eta_\epsilon\cvl}{a^{jk}} \nabla^{\ell-1} (\chi \u^i),_k\big] \nabla^\ell \big[\eta_\epsilon \cvl (\chi \u^i),_j \big] dx \,;
\end{align*}
thus, after rearranging terms, the positivity condition (\ref{positivity_of_a}) implies that
\begin{align}
& \lambda \big\|\nabla^{\ell + 1} \big(\eta_\epsilon \cvl (\chi \u)\big)\big\|^2_{L^2(\Omega)} \le \int_\Omega a^{jk} \frac{\p \u^i}{\p x_k} \frac{\p \Varphi^i}{\p x_j}\, dx - \int_\Omega (\nabla a^{jk}) \nabla^{\ell-1} \big[\eta_\epsilon\cvl (\chi \u^i),_k\big] \nabla^\ell \big[\eta_\epsilon \cvl (\chi \u^i),_j \big] dx \nonumber\\
&\qquad - \int_\Omega \big[\nabla \bigcomm{\eta_\epsilon\cvl}{a^{jk}} \nabla^{\ell-1} (\chi \u^i),_k\big] \nabla^\ell \big[\eta_\epsilon \cvl (\chi \u^i),_j \big] dx \nonumber\\
&\qquad - \sum_{r=0}^{\ell-2} \smallexp{$\displaystyle{}{{\ell\hspace{-1pt}-\hspace{-1pt}1}\choose{r}}$} \int_\Omega \nabla \big[\eta_\epsilon \cvl \big((\nabla^{\ell - 1 - r} a^{jk}) \nabla^r (\chi \u^i),_k\big)\big] \nabla^\ell \big[\eta_\epsilon \cvl (\chi \u^i),_j \big] dx \nonumber\\
&\qquad + \int_\Omega \nabla^{\ell} \big[\eta_\epsilon \cvl \big( a^{jk} \u^i \chi,_k \big)\big] \nabla^{\ell} \big[\eta_\epsilon \cvl (\chi \u^i),_j \big] dx \nonumber\\
&\qquad + \int_\Omega \nabla^{\ell-1} \big[\eta_\epsilon \cvl \big(a^{jk} \chi,_j \u^i_{,k}\big) \big] \nabla^{\ell+1} \big[\eta_\epsilon \cvl (\chi \u^i) \big] dx\,. \nonumber
\end{align}
The last five integrals on the right-hand side of the inequality above can be estimated using \Holder's inequality and the commutation estimate (\ref{commutator_estimates1}), and we obtain that
\begin{equation}\label{elliptic_interior_estimate_temp1}
\lambda \big\|\nabla^{\ell + 1} \big(\eta_\epsilon \cvl (\chi \u)\big)\big\|^2_{L^2(\Omega)} \le \int_\Omega a^{jk} \frac{\p \u^i}{\p x_k} \frac{\p \Varphi^i}{\p x_j}\, dx + C \|a\|_{\mC^\ell(\cls{\Omega})} \|\u\|_{H^\ell(\Omega)} \big\|\nabla^{\ell+1} \big(\eta_\epsilon \cvl (\chi \u^i)\big) \big\|_{L^2(\Omega)}\,.
\end{equation}
On the other hand, it is easy to see that
\begin{equation}\label{elliptic_interior_estimate_temp2}
\begin{array}{ll}
& \displaystyle{} \int_\Omega \f \hspace{-.5pt}\cdot\hspace{-.5pt} \Varphi \, dx = - \int_\Omega \nabla^{\ell-1} \big[\eta_\epsilon\cvl (\chi \f^i)\big] \nabla^{\ell + 1} \big[\eta_\epsilon \cvl (\chi \u^i)\big] dx \vspace{.2cm}\\
&\displaystyle{}\hspace{56pt} \le C \|\f\|_{H^{\ell-1}(\Omega)} \big\|\nabla^{\ell+1} \big(\eta_\epsilon\cvl (\chi \u) \big)\big\|_{L^2(\Omega)}\,.
\end{array}
\end{equation}
Summing (\ref{uphi_inner_product}), (\ref{elliptic_interior_estimate_temp1}) and (\ref{elliptic_interior_estimate_temp2}),  we find that
\begin{align*}
& \big\|\nabla^\ell (\eta_\epsilon\cvl (\chi \u)\big\|^2_{L^2(\Omega)} + \lambda \big\|\nabla^{\ell + 1} \big(\eta_\epsilon \cvl (\chi \u)\big)\big\|^2_{L^2(\Omega)} \\
&\qquad\qquad \le C \Big[\|\f\|_{H^{\ell-1}(\Omega)} + \|a\|_{\mC^\ell(\cls{\Omega})} \|\u\|_{H^\ell(\Omega)} \Big] \big\|\nabla^{\ell+1} \big(\eta_\epsilon\cvl (\chi \u)\big)\big\|_{L^2(\Omega)}\,;
\end{align*}
therefore, by Young's inequality,
\begin{align*}
& \big\|\nabla^\ell (\eta_\epsilon\cvl (\chi \u)\big\|^2_{L^2(\Omega)} + \lambda \big\|\nabla^{\ell + 1} \big(\eta_\epsilon \cvl (\chi \u)\big)\big\|^2_{L^2(\Omega)} \\
&\qquad\quad \le \frac{C}{\lambda}\, \Big[\|\f\|^2_{H^{\ell-1}(\Omega)} + \|a\|^2_{\mC^\ell(\cls{\Omega})} \|\u\|^2_{H^\ell(\Omega)} \Big] + \frac{\lambda}{2}\, \big\|\nabla^{\ell+1} \big[\eta_\epsilon \cvl (\chi \u^i) \big]\big\|_{L^2(\Omega)}
\end{align*}
which further implies that
\begin{equation}\label{elliptic_interior_estimate_temp}
\big\|\nabla^{\ell + 1} \big(\eta_\epsilon \cvl (\chi \u)\big)\big\|_{L^2(\Omega)} \le \frac{C}{\lambda}\, \Big[\|\f\|_{H^{\ell-1}(\Omega)} + \|a\|_{\mC^\ell(\cls{\Omega})} \|\u\|_{H^\ell(\Omega)} \Big]\,.
\end{equation}
Since $\f\in H^{\rk-1}(\Omega)$ and $a\in \mC^\rk(\cls{\Omega})$, the assumption that $\u\in H^\ell(\Omega)$ implies that the right-hand side of (\ref{elliptic_interior_estimate_temp}) is bounded independent of the smoothing parameter $\epsilon$. Therefore, we can pass to the limit as $\epsilon \to 0$ in (\ref{elliptic_interior_estimate_temp}) and obtain that
$$
\|\nabla^{\ell + 1} (\chi \u)\|_{L^2(\Omega)} \le \frac{C}{\lambda}\, \Big[\|\f\|_{H^{\ell-1}(\Omega)} + \|a\|_{\mC^\ell(\cls{\Omega})} \|\u\|_{H^\ell(\Omega)} \Big]
$$
or
\begin{equation}\label{vector-valued_interior_estimate}
\|\chi \nabla^{\ell + 1} \u\|_{L^2(\Omega)} \le \frac{C}{\lambda}\, \Big[\|\f\|_{H^{\ell-1}(\Omega)} + \big(\|a\|_{\mC^\ell(\cls{\Omega})} + \lambda)\|\u\|_{H^\ell(\Omega)} \Big]\,.
\end{equation}
This implies that $\u\in H^{\ell+1}_{\text{loc}}(\Omega)$.   So, we have shown that if  $\u\in H^{\ell}(\Omega)$  for $1\le \ell\le k$ then, in fact, $\u\in H^{\ell+1}_{\text{loc}}(\Omega)$ and $\u$ satisfies (\ref{vector-valued_interior_estimate}).

Now, we note that by Theorem \ref{thm:weak_existence_vector-valued_elliptic_prob}, $\u\in H^{1}(\Omega)$ and hence  $\u\in H^{2}_{\text{loc}}(\Omega)$. This allows us to integrate-by-parts in the variational formulation (\ref{vector-valued_elliptic_weak_form}) and find that
$$
\int_\Omega \Big[\u - \frac{\p}{\p x_j}\Big(a^{jk} \frac{\p \u}{\p x_k}\Big) - \f\Big] \hspace{-.5pt}\cdot\hspace{-.5pt} \Varphi\, dx = 0 \qquad\Forall \Varphi\in \mC^\infty_\cptspt(\Omega)\,.
$$
The above identity implies that
\begin{equation}\label{a.e._equality}
\u^i - \frac{\p}{\p x_j}\Big(a^{jk} \frac{\p \u^i}{\p x_k}\Big) = \f^i \quad \text{a.e. in $\Omega$}.
\end{equation}

\vspace{.1 in}
\noindent {\bf Step 2: (Regularity for tangential derivatives of $\u$ in the $\bfw$-direction  near $\bdy\Omega$)} We now initiate the induction procedure. We assume that $\u \in H^\ell(\Omega)$ for some $1\le \ell \le \rk-1$, and prove then that $\u \in H^{\ell+1}(\Omega)$. Let $\{\U_m\}_{m=0}^K$, $\{\vartheta_m\}_{m=1}^K$, and $\{r_m\}_{m=1}^K$ denote the system of local charts  given in Proposition \ref{prop:Cr_domain_charts}
with $\varepsilon \ll 1$, and let $0 \le \zeta_m \le 1$ in $\mC^\infty_\cptspt(\U_m)$ denote a partition-of-unity subordinate to the open covering $\U_m$ as given in Definition \ref{partition_of_unity}. We fixed $m\in \{1,\cdots,K\}$, and work with the chart $\vartheta_m: B(0,r_m) \to \U_m$. On $B(0,r_m)$, we define the new functions
$$
\widetilde{\zeta} = \zeta_m \circ \vartheta_m,\ \widetilde{\u} = \u \circ \vartheta_m,\ \widetilde{\bfw} = \bfw_\epsilon \circ \vartheta_m,\ \widetilde{\f} = \f \circ \vartheta_m,\ \widetilde{\g} = \g \circ \vartheta_m \text{ \ and \ } \widetilde{\Varphi} = \Varphi \circ \vartheta_m\,.
$$
With $A = (\nabla \vartheta_m)^{-1}$, we define $b^{rs} = (a^{jk}\circ \vartheta) A^k_s A^j_r$. Then the matrix $b$ is positive-define. In fact, since $\|\nabla \vartheta_m - \id\|_{L^\infty(B^+_m)} \ll 1$,
\begin{equation}\label{brs_elliptic}
b^{rs} \xi_r \xi_s = (a^{jk}\circ\vartheta_m) A^s_k A^r_j \xi_r \xi_s \ge \lambda |A^\rT \xi|^2 \ge \frac{\lambda}{2} |\xi|^2\,.
\end{equation}
Setting $x=\vartheta_m(y)$,  the change-of-variables formula shows  that the variational formulation (\ref{vector-valued_elliptic_weak_form}) takes the form
\begin{equation}\label{transformed_weak_form_on_planar_domain}
\int_{B^+_m} \widetilde{\u}\cdot \widetilde{\Varphi} \,dy + \int_{B^+_m} b^{rs} \frac{\p \widetilde{\u}^i}{\p y_s} \frac{\p \widetilde{\Varphi}^i}{\p y_r} \,dy = \int_{B^+_m} \widetilde{\f} \cdot \widetilde{\Varphi} \,dy + \int_{B_m \cap \{y_\n = 0\}} \widetilde{\g} \cdot \widetilde{\Varphi} \,dS \quad \Forall \widetilde{\Varphi} \in \widetilde{\V}_m,
\end{equation}
where $\widetilde{\V}_m = \big\{ \widetilde{\Varphi} \in H^1(B^+_m)\,\big|\, \widetilde{\Varphi}\cdot \widetilde{\bfw} = 0 \text{ on } B_m \cap \{y_\n = 0\}, \widetilde{\Varphi} = 0 \text{ on } \bbR^\n_+ \cap \bdy B_m \big\}$.

With $\Delta_0 = \sum\limits_{\alpha = 1}^{\n-1} \dfrac{\p^2}{\p y_j^2}$ denoting the horizontal Laplace operator, we define
$$
\Varphi^i = (-1)^\ell \big[\widetilde{\zeta} \widetilde{\bfw}^i \Lambda_\epsilon \Delta_0^\ell \Lambda_\epsilon (\widetilde{\zeta} \widetilde{\u} \hspace{-.5pt}\cdot\hspace{-.5pt} \widetilde{\bfw})\big]\circ \vartheta_m^{-1}\,,
$$
where $\Lambda_\epsilon$ is the horizontal convolution-by-layers operator given by
$$
\Lambda_\epsilon \phi (y_h, y_\n) = \int_{\bbR^{\n-1}} \rho_\epsilon (y_h - z_h) \phi(z_h, y_\n) dz_h \qquad \text{for}\quad \phi ( \cdot , y_\n) \in L^1(\bbR^{\n-1} ) \,,
$$
where $y_h=(y_1,..., y_{\n-1})$.
Recalling that $\bp = \big(\dfrac{\p}{\p y_1},\cdots,\dfrac{\p}{y_{\n-1}}\big)$ denotes the horizontal gradient, we note that
\begin{align*}
\bp^\ell \v \cdot \bp^\ell \w &= \sum_{\alpha_1=1}^{\n-1} \cdots \sum_{\alpha_\ell=1}^{\n-1} \frac{\p^\ell \v}{\p y_{\alpha_1} \cdots \p y_{\alpha_\ell}} \cdot \frac{\p^\ell \w}{\p y_{\alpha_1} \cdots \p y_{\alpha_\ell}}\,,\\
\bp^{\ell-1} \v \cdot \bp^{\ell+1} \w &= \sum_{\alpha_1=1}^{\n-1} \cdots \sum_{\alpha_{\ell-1}=1}^{\n-1} \frac{\p^{\ell-1} \v}{\p y_{\alpha_1} \cdots \p y_{\alpha_{\ell-1}}} \cdot \frac{\p^{\ell-1} \Delta_0 \w}{\p y_{\alpha_1} \cdots \p y_{\alpha_{\ell-1}}}\,,
\end{align*}
and so forth.

Since $\Varphi \hspace{-.5pt}\cdot\hspace{-.5pt} \bfw = 0$ on $\bdy\Omega$, $\Varphi \in \V$ and can be used as a test function. The use of $\Varphi$ as a test function in (\ref{vector-valued_elliptic_weak_form}) implies that
\begin{equation}\label{pf_of_thm10.4_inequality1}
(\widetilde{\u},\widetilde{\Varphi})_{L^2(\Omega)} + \int_\Omega b^{rs} \frac{\p \widetilde{\u}^i}{\p y_s} \frac{\p \widetilde{\Varphi}^i}{\p y_r}\,dy \le C \Big[\|\f\|_{H^\ell(\Omega)} + \|\g\|_{H^{\ell-0.5}(\bdy\Omega)} \Big] \big\|\bp^\ell \Lambda_\epsilon (\widetilde{\zeta} \widetilde{\u}\hspace{-1pt}\cdot\widetilde{\rw}) \big\|_{H^1(B^+_m)}\,.
\end{equation}
Similar to (\ref{uphi_inner_product}), integrating by parts with respect to $y_h$ implies that
$$
(\u,\Varphi)_{L^2(\Omega)} = \big\|\bp^\ell \Lambda_\epsilon (\widetilde{\zeta} \widetilde{\u} \hspace{-.5pt}\cdot\hspace{-.5pt} \widetilde{\bfw})\big\|^2_{L^2(B^+_m)}\,.
$$
Now we focus on the second term on the left-hand side of (\ref{pf_of_thm10.4_inequality1}). Integrating by parts in the horizontal direction (using $\bp$) yields
\begin{align}
\int_{B^+_m} b^{rs} \frac{\p \widetilde{\u}^i}{\p y_r} \frac{\p \Varphi^i}{\p y_s}\,dy &= \int_{B^+_m} \bp^\ell \Lambda_\epsilon \big(b^{rs} \widetilde{\zeta} \widetilde{\u}_{m,s}\hspace{-1pt}\cdot \widetilde{\bfw} \big) \bp^\ell \Lambda_\epsilon (\widetilde{\zeta} \widetilde{\u}\hspace{-1pt}\cdot \widetilde{\bfw}),_r dy \nonumber\\
&= \int_{B^+_m} \bp \Lambda_\epsilon \big(b^{rs} \bp^{\ell-1} (\widetilde{\zeta} \widetilde{\u}\hspace{-1pt}\cdot\widetilde{\bfw}),_s\hspace{-2pt}\big) \bp^\ell \Lambda_\epsilon (\widetilde{\zeta} \widetilde{\u}\hspace{-1pt}\cdot\widetilde{\bfw}),_r dy \nonumber\\
&\quad + \sum_{k=0}^{\ell-2} \smallexp{$\displaystyle{}{{\ell\hspace{-1pt}-\hspace{-1pt}1}\choose{k}}$} \hspace{-2pt}\int_{B^+_m} \hspace{-3pt}\bp \Lambda_\epsilon \big(\bp^{\ell-1-k} b^{rs} \bp^k (\widetilde{\zeta} \widetilde{\u}\hspace{-1pt}\cdot\widetilde{\bfw}),_s\hspace{-2pt}\big) \bp^\ell \Lambda_\epsilon (\widetilde{\zeta} \widetilde{\u}\hspace{-1pt}\cdot\widetilde{\bfw}),_r dy \nonumber\\
&\quad - \int_{B^+_m} \bp^\ell \Lambda_\epsilon \big(b^{rs} \widetilde{\u}^i (\widetilde{\zeta} \widetilde{\bfw}^i),_s\hspace{-2pt}\big) \bp^\ell \Lambda_\epsilon (\widetilde{\zeta} \widetilde{\u}\hspace{-1pt}\cdot\widetilde{\bfw}),_r dy \,.\label{vector-valued_elliptic_est_id1}
\end{align}
For the first term on the right-hand side of (\ref{vector-valued_elliptic_est_id1}), as in Step 1, we find that
\begin{align}
& \int_{B^+_m} \bp \Lambda_\epsilon \big(b^{rs} \bp^{\ell-1} (\widetilde{\zeta} \widetilde{\u}\hspace{-1pt}\cdot\widetilde{\bfw}),_s\hspace{-2pt}\big) \bp^\ell \Lambda_\epsilon (\widetilde{\zeta} \widetilde{\u}\hspace{-1pt}\cdot\widetilde{\bfw}),_r dy \nonumber\\
&\qquad = \int_{B^+_m} \bp \big[b^{rs} \Lambda_\epsilon \bp^{\ell-1} (\widetilde{\zeta} \widetilde{\u}\hspace{-1pt}\cdot\widetilde{\bfw}),_s\big] \bp^\ell \Lambda_\epsilon (\widetilde{\zeta} \widetilde{\u}\hspace{-1pt}\cdot\widetilde{\bfw}),_r dy \nonumber\\
&\qquad\quad + \int_{B^+_m} \big(\bp \comm{\Lambda_\epsilon}{b^{rs}} \bp^{\ell-1} (\widetilde{\zeta} \widetilde{\u}\hspace{-1pt}\cdot\widetilde{\bfw}),_s\big) \bp^\ell \Lambda_\epsilon (\widetilde{\zeta} \widetilde{\u}\hspace{-1pt}\cdot\widetilde{\bfw}),_r dy \nonumber\\
&\qquad = \int_{B^+_m} b^{rs} \bp^\ell \Lambda_\epsilon (\widetilde{\zeta} \widetilde{\u}\hspace{-1pt}\cdot\widetilde{\bfw}),_s \bp^\ell \Lambda_\epsilon (\widetilde{\zeta} \widetilde{\u}\hspace{-1pt}\cdot\widetilde{\bfw}),_r dy \nonumber\\
&\qquad\quad + \int_{B^+_m} (\bp b^{rs}) \Lambda_\epsilon \bp^{\ell-1} (\widetilde{\zeta} \widetilde{\u}\hspace{-1pt}\cdot\widetilde{\bfw}),_s \bp^\ell \Lambda_\epsilon (\widetilde{\zeta} \widetilde{\u}\hspace{-1pt}\cdot\widetilde{\bfw}),_r dy \nonumber\\
&\qquad\quad + \int_{B^+_m} \big(\bp \comm{\Lambda_\epsilon}{b^{rs}} \bp^{\ell-1} (\widetilde{\zeta} \widetilde{\u}\hspace{-1pt}\cdot\widetilde{\bfw}),_s\hspace{-2pt}\big) \bp^\ell \Lambda_\epsilon (\widetilde{\zeta} \widetilde{\u}\hspace{-1pt}\cdot\widetilde{\bfw}),_r dy \,.\label{pf_of_thm10.4_inequality2}
\end{align}
The positivity condition (\ref{brs_elliptic}) and the commutation estimate (\ref{commutator_estimates2}) imply that
\begin{align}
& \int_{B^+_m} \bp \Lambda_\epsilon \big(b^{rs} \bp^{\ell-1} (\widetilde{\zeta} \widetilde{\u}\hspace{-1pt}\cdot\widetilde{\bfw}),_s\hspace{-2pt}\big) \bp^\ell \Lambda_\epsilon (\widetilde{\zeta} \widetilde{\u}\hspace{-1pt}\cdot\widetilde{\bfw}),_r dy \label{pf_of_thm10.4_inequality3}\\
&\qquad \ge \frac{\lambda}{2} \big\|\bp^\ell \Lambda_\epsilon \nabla (\widetilde{\zeta} \widetilde{\u}\hspace{-1pt}\cdot\widetilde{\bfw}) \big\|^2_{L^2(B^+_m)} \hspace{-2pt}-\hspace{-1pt} C \|\u\|_{H^\ell(\Omega)} \Big[\big\|\bp^\ell \Lambda_\epsilon (\widetilde{\zeta} \widetilde{\u}\hspace{-1pt}\cdot\widetilde{\bfw}) \big\|_{H^1(B^+_m)} \hspace{-2pt}+\hspace{-1pt} \|\u\|_{H^\ell(\Omega)} \Big]. \nonumber
\end{align}
For the remaining terms on the right-hand side of (\ref{vector-valued_elliptic_est_id1}), we apply \Holder's inequality and find that
\begin{align}
& \Big|\sum_{k=0}^{\ell-2} \smallexp{$\displaystyle{}{{\ell\hspace{-1pt}-\hspace{-1pt}1}\choose{k}}$} \int_{B^+_m} \bp \Lambda_\epsilon \big(\bp^{\ell-1-k} b^{rs} \bp^k (\widetilde{\zeta} \widetilde{\u}\hspace{-1pt}\cdot\widetilde{\bfw}),_s\hspace{-2pt}\big) \bp^\ell \Lambda_\epsilon (\widetilde{\zeta} \widetilde{\u}\hspace{-1pt}\cdot\widetilde{\bfw}),_r dy\Big| \nonumber\\
&\qquad + \Big|\int_{B^+_m} \bp^\ell \Lambda_\epsilon \big(b^{rs} \widetilde{\u}^i (\widetilde{\zeta} \widetilde{\bfw}^i),_s\hspace{-2pt}\big) \bp^\ell \Lambda_\epsilon (\widetilde{\zeta} \widetilde{\u}\hspace{-1pt}\cdot\widetilde{\bfw}),_r dy\Big| \nonumber\\
&\qquad\qquad \le C \|\u\|_{H^\ell(\Omega)} \Big[\big\|\bp^\ell \Lambda_\epsilon \nabla (\widetilde{\zeta} \widetilde{\u}\hspace{-1pt}\cdot\widetilde{\bfw}) \big\|_{L^2(B^+_m)} + \|\u\|_{H^\ell(\Omega)} \Big]\,, \label{vector-valued_elliptic_proof_ineq1}
\end{align}
where $C$ depends on $\|a\|_{\mC^\ell(\Omega)}$, $\|\bfw\|_{\mC^{\ell+1}(\Omega)}$ and $|\bdy\Omega|_{\mC^{\ell+1}(\Omega)}$.
As a consequence,  Young's inequality together with (\ref{pf_of_thm10.4_inequality1})--(\ref{pf_of_thm10.4_inequality3}) implies that
\begin{align*}
& \big\|\bp^\ell \Lambda_\epsilon (\widetilde{\zeta} \widetilde{\u}\hspace{-1pt}\cdot\widetilde{\bfw}) \big\|^2_{L^2(B^+_m)} \hspace{-1.5pt}+ \lambda \big\|\bp^\ell \Lambda_\epsilon \nabla (\widetilde{\zeta} \widetilde{\u}\hspace{-1pt}\cdot\widetilde{\bfw}) \big\|^2_{L^2(B^+_m)} \\
&\qquad \le C_\delta \Big[\|\u\|^2_{H^\ell(\Omega)} \hspace{-1.5pt}+ \|\f\|^2_{H^{\ell-1}(\Omega)} \hspace{-1.5pt}+ \|\g\|^2_{H^{\ell-1.5}(\bdy\Omega)} \Big] \hspace{-1.5pt}+ \delta \big\|\bp^\ell \Lambda_\epsilon \nabla (\widetilde{\zeta} \widetilde{\u}\hspace{-1pt}\cdot\widetilde{\bfw}) \big\|^2_{L^2(B^+_m)}\,,
\end{align*}
which, by choosing $\delta > 0$ sufficiently small, shows that
$$
\big\|\bp^\ell\, \Lambda_\epsilon (\widetilde{\zeta} \widetilde{\u}\hspace{-1pt}\cdot\widetilde{\bfw}) \big\|_{H^1(B^+_m)} \le C \Big[\|\f\|_{H^{\ell-1}(\Omega)} + \|\g\|_{H^{\ell-0.5}(\bdy\Omega)} + \|\u\|_{H^\ell(\Omega)} \Big]
$$
for some constant $C = C(\|a\|_{\mC^\ell(\Omega)}, \|\bfw\|_{\mC^{\ell+1}(\Omega)}, |\bdy\Omega|_{\mC^{\ell+1}})$.

Since the estimate above is independent of the smoothing parameter $\epsilon$, by passing to the limit as $\epsilon \to 0$ we conclude that
$$
\big\|\widetilde{\zeta} \bp^\ell (\widetilde{\u}\hspace{-1pt}\cdot\widetilde{\bfw}) \big\|_{H^1(B^+_m)} \le C \Big[\|\u\|_{H^\ell(\Omega)} + \|\f\|_{H^{\ell-1}(\Omega)} + \|\g\|_{H^{\ell-0.5}(\bdy\Omega)} \Big] \,.
$$
Since both $\bfw$ and $\vartheta_m$ are $\mC^{k+1}$ in the support of $\widetilde{\zeta}$, it follows that
\begin{equation}\label{vector-valued_bdy_est_temp}
\big\|\widetilde{\zeta} \widetilde{\bfw} \hspace{-1pt}\cdot \bp^\ell \widetilde{\u}\big\|_{H^1(B^+_m)} \le C \Big[\|\f\|_{H^{\ell-1}(\Omega)} + \|\g\|_{H^{\ell-0.5}(\bdy\Omega)} + \|\u\|_{H^\ell(\Omega)}\Big] \,.
\end{equation}

\vspace{.1 in}
\noindent {\bf Step 3: (Regularity for tangential derivatives of $\u$ in the $\bfw^\perp$-directions  near $\bdy\Omega$)} Estimate (\ref{vector-valued_bdy_est_temp})  provides  regularity for the vector $\widetilde{\zeta}\, \bp^\ell \nabla \widetilde{\u} \cdot \bfw$.
Next, we establish the regularity of $\widetilde{\zeta}\, \bp^\ell \nabla \widetilde{\u} \times  \bfw $. We define
\begin{align*}
\widetilde{\Varphi}^i &= (-1)^\ell \Big[\widetilde{\zeta} \Lambda_\epsilon \Delta_0^\ell \Lambda_\epsilon (\widetilde{\zeta} \widetilde{\u}^i) - \big(\widetilde{\zeta} \widetilde{\bfw} \hspace{-1pt}\cdot\Lambda_\epsilon \Delta_0^\ell \Lambda_\epsilon (\widetilde{\zeta} \widetilde{\u})\big) \frac{\widetilde{\bfw}^i}{|\widetilde{\bfw}|^2} \Big] \\
&= (-1)^\ell \Big[\widetilde{\zeta} \Lambda_\epsilon \Delta_0^\ell \Lambda_\epsilon (\widetilde{\zeta} \widetilde{\u}^i) - 
\big(\widetilde{\zeta} \Lambda_\epsilon \Delta_0^\ell \Lambda_\epsilon (\widetilde{\zeta} \widetilde{\u}^j)\big) \frac{\widetilde{\bfw}^j \widetilde{\bfw}^i}{|\widetilde{\bfw}|^2} \Big] \,.
\end{align*}
Note that $\Varphi$ is the projection of the vector $\widetilde{\zeta} \Lambda_\epsilon \Delta_0^\ell \Lambda_\epsilon (\widetilde{\zeta} \widetilde{\u})$ onto the affine space with normal $\bfw$, so $\Varphi \in \V$ and can be used as a test function in (\ref{transformed_weak_form_on_planar_domain}).  Following the similar computation in Step 2 above, we have that
\begin{align*}
& \big\|\bp^\ell \Lambda_\epsilon (\widetilde{\zeta} \widetilde{\u})\big\|^2_{L^2(B^+_m)} + \frac{\lambda}{4} \big\|\bp^\ell \Lambda_\epsilon \nabla (\widetilde{\zeta} \widetilde{\u})\big\|^2_{L^2(B^+_m)} \\
&\qquad \le C_\delta \Big[\|\f\|^2_{H^{\ell-1}(\Omega)} + \|\g\|^2_{H^{\ell-1.5}(\bdy\Omega)} + \|\u\|^2_{H^\ell(\Omega)}\Big] + \delta \big\|\bp^\ell \Lambda_\epsilon \nabla (\widetilde{\zeta} \widetilde{\u})\big\|^2_{L^2(B^+_m)} \\
&\qquad\quad + (-1)^{\ell+1} \int_{B^+_m} b^{rs} \widetilde{\u}^i,_s \Big[\big(\widetilde{\zeta} \Lambda_\epsilon \Delta_0^\ell \Lambda_\epsilon (\widetilde{\zeta} \widetilde{\u}^j_m)\big) \frac{\widetilde{\bfw}^j \widetilde{\bfw}^i}{|\widetilde{\bfw}|^2}\Big]_{,r} dy \\
&\qquad \le C_\delta \Big[\|\f\|^2_{H^{\ell-1}(\Omega)} + \|\g\|^2_{H^{\ell-1.5}(\bdy\Omega)} + \|\u\|^2_{H^\ell(\Omega)}\Big] + 2 \delta \big\|\bp^\ell \Lambda_\epsilon \nabla (\widetilde{\zeta} \widetilde{\u})\big\|^2_{L^2(B^+_m)} \\
&\qquad\quad + (-1)^{\ell+1} \int_{B^+_m} \frac{b^{rs}}{|\widetilde{\bfw}|^2}\, \widetilde{\zeta} \widetilde{\bfw}\hspace{-1pt}\cdot \widetilde{\u},_s \big( \Lambda_\epsilon \Delta_0^\ell \Lambda_\epsilon (\widetilde{\zeta} \widetilde{\u}^j,_r)\big) \widetilde{\bfw}^j \,dy \\
&\qquad \le C_\delta \Big[\|\f\|^2_{H^{\ell-1}(\Omega)} + \|\g\|^2_{H^{\ell-1.5}(\bdy\Omega)} + \|\u\|^2_{H^\ell(\Omega)}\Big] + 3 \delta \big\|\bp^\ell \Lambda_\epsilon \nabla (\widetilde{\zeta} \widetilde{\u})\big\|^2_{L^2(B^+_m)} \\
&\qquad\quad - \int_{B^+_m} \frac{b^{rs}}{|\widetilde{\bfw}|^2}\, \bp^\ell (\widetilde{\zeta} \widetilde{\u}_{m,s} \hspace{-.5pt}\cdot\hspace{-.5pt} \widetilde{\bfw}) \big(\Lambda_\epsilon \bp^\ell \Lambda_\epsilon (\widetilde{\zeta} \widetilde{\u}^j,_r)\big) \widetilde{\bfw}^j \, dy\,.
\end{align*}
Applying estimate (\ref{vector-valued_bdy_est_temp}) and Young's inequality,
\begin{align*}
& - \int_{B^+_m} \frac{b^{rs}}{|\widetilde{\bfw}|^2}\, \bp^\ell (\widetilde{\zeta} \widetilde{\u}_{m,s} \hspace{-.5pt}\cdot\hspace{-.5pt} \widetilde{\bfw}) \big(\Lambda_\epsilon \bp^\ell \Lambda_\epsilon (\widetilde{\zeta} \widetilde{\u}^j_{m,r})\big) \widetilde{\bfw}^j_m \, dy \\
&\qquad\quad \le C \big\|\bp^\ell (\widetilde{\zeta} \widetilde{\bfw}\cdot \nabla \widetilde{\u})\big\|_{L^2(\Omega)} \big\|\bp^\ell \Lambda_\epsilon (\widetilde{\zeta} \nabla \widetilde{\u})\big\|_{L^2(\Omega)} \\
&\qquad\quad \le C_\delta \Big[\|\f\|^2_{H^{\ell-1}(\Omega)} + \|\g\|^2_{H^{\ell-1.5}(\bdy\Omega)} + \|\u\|^2_{H^\ell(\Omega)}\Big] + \delta \big\|\bp^\ell \Lambda_\epsilon (\widetilde{\zeta} \nabla \widetilde{\u})\big\|^2_{L^2(\Omega)}\,;
\end{align*}
thus by choosing $\delta>0$ sufficiently small, we conclude that
\begin{align*}
\|\bp^\ell \Lambda_\epsilon (\widetilde{\zeta} \widetilde{\u})\|^2_{H^1(B^+_m)} \le C \Big[\|\u\|^2_{H^\ell(\Omega)} + \|\f\|^2_{H^{\ell-1}(\Omega)} + \|\g\|^2_{H^{\ell-1.5}(\bdy\Omega)}\Big]\,.
\end{align*}
Again, due to the $\epsilon$-independence of the right-hand side, we conclude that
\begin{equation}\label{vector-valued_bdy_est}
\|\widetilde{\zeta}\, \bp^\ell \widetilde{\u}\|_{H^1(B^+_m)} \le C \Big[\|\u\|_{H^\ell(\Omega)} + \|\f\|_{H^{\ell-1}(\Omega)} + \|\g\|_{H^{\ell-0.5}(\bdy\Omega)}\Big]
\end{equation}
for some constant $C = C(\|a\|_{\mC^\ell(\Omega)}, \|\bfw\|_{\mC^{\ell+1}(\Omega)}, |\bdy\Omega|_{\mC^{\ell+1}})$.

\vspace{.1 in}
\noindent {\bf Step 4: (Regularity for normal derivatives of $\u$ near $\bdy\Omega$)}
Multiplying (\ref{a.e._equality}) by $\zeta_m$ and then composing with $\vartheta_m$, by the Piola identity (\ref{Piola_id}) we obtain that
\begin{align*}
\widetilde{\zeta} \widetilde{\u} - \widetilde{\zeta} \big(b^{rs} \widetilde{\u},_s \big),_r = \widetilde{\zeta} (\f\circ \vartheta_m) \quad\text{a.e. in $B^+_m$}\,.
\end{align*}
Letting $\bp^{\ell - 1 - j} \nabla^j$ act on the equation above, we find that
\begin{equation}
\widetilde{\zeta} b^{rs} \bp^{\ell-1-j} \nabla^j \widetilde{\u},_{rs} = \text{\bf\emph{F}}_{(\ell,j)} \quad\text{a.e. in $B^+_m$}\label{vector-valued_elliptic_id_temp1}
\end{equation}
for some $\text{\bf\emph{F}}_{(\ell,j)} \in L^2(\Omega)$ satisfying
$$
\|\text{\bf\emph{F}}_{(\ell,j)}\|_{L^2(\Omega)} \le C \Big[\|\f\|_{H^{\ell-1}(\Omega)} + \|\u\|_{H^\ell(\Omega)} \Big]\,,
$$
where the constant $C$ depends on $\|a\|_{\mC^\ell(\Omega)}$. Using (\ref{brs_elliptic}), $b^{\n\n} > 0$; thus (\ref{vector-valued_elliptic_id_temp1}) further implies that
$$
\widetilde{\zeta} \bp^{\ell-1-j} \nabla^j \widetilde{\u},_{\n\n} = \frac{1}{b^{\n\n}} \Big[\text{\bf\emph{F}}_{(\ell,j)} - \widetilde{\zeta} \sum_{(r,s)\ne (\n,\n)} b^{rs} \bp^{\ell-1} \nabla^j \widetilde{\u},_{rs} \Big]\,.
$$
Now we argue by induction on $0\le j\le \ell-1$. When $j=0$, (\ref{vector-valued_bdy_est}) shows that
$$
\big\|\widetilde{\zeta} \bp^{\ell-1} \widetilde{\u},_{\n\n}\big\|_{L^2(B^+_m)} \le C \Big[\|\u\|_{H^\ell(\Omega)} + \|\f\|_{H^{\ell-1}(\Omega)} + \|\g\|_{H^{\ell-0.5}(\bdy\Omega)} \Big]
$$
which, combined with (\ref{vector-valued_bdy_est}), provides the estimate
$$
\big\|\widetilde{\zeta} \bp^{\ell-1} \nabla^2 \widetilde{\u}\big\|_{L^2(B^+_m)} \le C \Big[\|\u\|_{H^\ell(\Omega)} + \|\f\|_{H^{\ell-1}(\Omega)} + \|\g\|_{H^{\ell-0.5}(\bdy\Omega)} \Big]\,.
$$
Repeating this process for $j=1,\cdots,\ell-1$, we conclude that
\begin{equation}\label{vector-valued_boundary_estimate}
\|\widetilde{\zeta} \nabla^{\ell+1} \widetilde{\u}\|_{L^2(B^+_m)} \le C \Big[\|\u\|_{H^\ell(\Omega)} + \|\f\|_{H^{\ell-1}(\Omega)} + \|\g\|_{H^{\ell-0.5}(\bdy\Omega)}\Big]\,.
\end{equation}
The combination of (\ref{vector-valued_interior_estimate}) and (\ref{vector-valued_boundary_estimate}), as well as the induction process, proves the theorem.
\end{proof}

\subsection{The case that the coefficients \texorpdfstring{$a^{jk}$}{aʲᵏ} are of Sobolev class}
We are now in the position of studying the regularity of solution $\u$ to (\ref{vector-valued_elliptic_eq}) when the coefficient $a^{jk}$ and the domain $\Omega$ is of Sobolev class. We first prove the following rather technical
\begin{theorem}[Regularity for the case that $a^{jk} \in H^\rk(\Omega)$ and $\Omega\in \mC^ \infty $]\label{thm:vector-valued_elliptic_eq_Sobolev_coeff}
Let $\Omega \subseteq \Rn$ be a bounded $\mC^\infty$-domain. Suppose that for some integer $\rk > \novertwo$ and $1\le \ell \le \rk$\,, $a^{jk} \in H^\rk(\Omega)$ satisfies the positivity condition
$$
a^{jk} \xi_j \xi_k \ge \lambda |\xi|^2\qquad\Forall \xi \in \Rn\,,
$$
and $\bfw \in H^{\max\{\rk,\ell+1\}}(\Omega)$ $($or $\bfw \in H^{\max\{\rk-\frac{1}{2},\ell+\frac{1}{2}\}}(\bdy\Omega))$ such that $\bfw$ vanishes nowhere on $\bdy\Omega$. Then for all $f\in H^{\ell-1}(\Omega)$ and $g \in H^{\ell-0.5}(\bdy\Omega)$, the weak solution $\u$ to {\rm(\ref{vector-valued_elliptic_eq})} belongs to $H^{\ell+1}(\Omega)$, and satisfies
\begin{equation}\label{vector-valued_elliptic_regularity_sobolev1}
\displaystyle{} \|\u\|_{H^{\ell+1}(\Omega)} \le C \Big[\|\f\|_{H^{\ell-1}(\Omega)} + \|\g\|_{H^{\ell-0.5}(\bdy\Omega)} + \P\big(\|a\|_{H^\rk(\Omega)}\big) \Big(\|\f\|_{L^2(\Omega)} + \|\g\|_{H^{-0.5}(\bdy\Omega)}\Big) \Big]
\end{equation}
for some constant $C = C\big(\|\bfw\|_{H^{\max\{\rk,\ell+1\}}(\Omega)}\big)$ and some polynomial $\P$.
\end{theorem}
\begin{proof}
Let $\{\U_m,\vartheta_m\}_{m=1}^K$ be a collection of charts of $\bdy\Omega$ given in Proposition \ref{prop:Cr_domain_charts}, $\{\zeta_m\}_{m=1}^K$ a partition-of-unity subordinate to $\{\U_m\}_{m=1}^K$
given in Definition \ref{partition_of_unity},  and let  $\rE: H^{\rk+1}(\Omega) \to H^{\rk+1}(\Rn)$ denote a Sobolev extension operator. We define $a_\epsilon = \eta_\epsilon \cvl (\rE a)$, $\f_{\hspace{-1.5pt}\epsilon} = \eta_\epsilon \cvl (\rE f)$, $\bfw_\epsilon = \eta_\epsilon \cvl (\rE \bfw)$. Finally, let $\g_\epsilon$ denote a smooth regularization of $\g$ defined by
$$
\g_\epsilon = \sum_{m=1}^K \sqrt{\zeta_m} \, \big[\Lambda_\epsilon \big((\sqrt{\zeta_m}\, \g)\circ\vartheta_m\big)\big]\circ \vartheta_m^{-1} \,.
$$
It follows that for $\epsilon \ll 1$ sufficiently small,
\begin{equation}\label{uniform_ellipticity_of_tilde_a}
a^{jk}_\epsilon(x) \xi_j \xi_k \ge \frac{\lambda}{2}\, |\xi|^2 \qquad\Forall\xi \in \Rn, x\in \Omega\,.
\end{equation}
Hence by Theorem \ref{thm:vector-valued_elliptic_regularity}, the solution $\u^\epsilon$ to the variational problem
$$
\int_\Omega \u^\epsilon \cdot \Varphi \,dx + \int_\Omega a^{jk}_\epsilon \frac{\p \u^\varepsilon}{\p x_k}\cdot \frac{\p \Varphi}{\p x_j} \,dx = \int_\Omega \f_{\hspace{-1.5pt}\epsilon} \cdot \Varphi \,dx + \int_{\bdy\Omega} \g_{\hspace{-1.5pt}\epsilon} \cdot \Varphi \,dS \qquad\Forall \Varphi \in \V
$$
satisfies $\u^\epsilon \in H^k(\Omega)$ for all $k\ge 1$. We next establish an $\epsilon$-independent upper bound for $\|\u^\epsilon\|_{H^{\ell+1}(\Omega)}$.

\noindent {\bf Step 1: (Regularity for tangential derivatives of $\u$ in the $\bfw$-direction  near $\bdy\Omega$)} We fix $m\in \{1,\cdots,K\}$ and set
$$
\widetilde{\zeta} = \zeta_m \circ \vartheta_m,\ \widetilde{\u} = \u^\epsilon \circ \vartheta,\ \ \widetilde{\bfw} = \bfw_\epsilon \circ \vartheta,\ \ \widetilde{\f} = \f_{\hspace{-1.5pt}\epsilon}\circ \vartheta,\ \ \widetilde{\g} = \g_{\hspace{-1.5pt}\epsilon} \circ \vartheta \text{ \ and \ } \widetilde{\Varphi} = \Varphi \circ \vartheta\,.
$$
With $A = (\nabla \vartheta)^{-1}$, we define $b^{rs}_\epsilon = (a^{jk}_\epsilon \circ \vartheta) A^s_k A^r_j$. Then, since $\|\nabla \vartheta - \id\|_{L^\infty(B^+_m)} \ll 1$, the matrix $b_\epsilon$ is positive-definite since using (\ref{uniform_ellipticity_of_tilde_a}),
\begin{equation}\label{be_elliptic}
b^{rs}_\epsilon  \xi_r \xi_s = (a^{jk}_\epsilon \circ \vartheta) A^s_k A^r_j \xi_r \xi_s \ge \frac{\lambda}{2} |A^\rT \xi|^2 \ge \frac{\lambda}{8} |\xi|^2 \qquad\Forall \xi\in \bbR^\n\,.
\end{equation}
Setting $x=\vartheta_m(y)$, and using the change-of-variables formula,  we find that the variational formulation
(\ref{vector-valued_elliptic_weak_form}) can
be written as
\begin{equation}\label{transformed_weak_form_vector-valued_elliptic}
\int_{B^+_m} \widetilde{\u}\cdot \widetilde{\Varphi} \,dy + \int_{B^+_m} b^{rs}_\epsilon \dfrac{\p \widetilde{\u}^i}{\p y_s} \dfrac{\p \widetilde{\Varphi}^i}{\p y_r} \,dy = \int_{B^+_m} \widetilde{\f} \cdot \widetilde{\Varphi} \,dy + \int_{B_m \cap \{y_\n = 0\}} \widetilde{\g} \cdot \widetilde{\Varphi} \,dS \quad \Forall \widetilde{\Varphi} \in \widetilde{\V}_m,
\end{equation}
where $\widetilde{\V}_m = \big\{ \widetilde{\Varphi} \in H^1(B^+_m)\,\big|\, \widetilde{\Varphi}\cdot \widetilde{\bfw} = 0 \text{ on } B_m \cap \{y_\n = 0\}, \widetilde{\Varphi} = 0 \text{ on } \bbR^\n_+ \cap \bdy B_m \big\}$. With $\Delta_0$ denoting the horizontal Laplace operator and $\bp$ denoting the horizontal gradient defined in Step 2 in the proof of Theorem \ref{thm:vector-valued_elliptic_regularity}, we define
$$
\widetilde{\Varphi}^i = (-1)^\ell \widetilde{\zeta} \widetilde{\bfw}^i \Delta_0^\ell (\widetilde{\zeta} \widetilde{\u} \hspace{-.5pt}\cdot\hspace{-.5pt} \widetilde{\bfw}) \,,
$$
so that
\begin{align}
& (\widetilde{\u}, \widetilde{\Varphi})_{L^2(B^+_m)} + \int_{B^+_m} b^{rs}_\epsilon \frac{\p \widetilde{\u}^i}{\p y_s} \frac{\p \widetilde{\Varphi}^i}{\p y_r} \,dy \nonumber\\
&\qquad\quad \le C \Big[\|\widetilde{\f}\|_{H^{\ell-1}(B^+_m)} + \|\widetilde{\g}\|_{H^{\ell-0.5}(B_m \cap \{y_\n = 0\})} \Big] \big\|\bp^\ell (\widetilde{\zeta} \widetilde{\u} \hspace{-.5pt}\cdot\hspace{-.5pt} \widetilde{\bfw})\big\|_{H^1(B^+_m)} \nonumber\\
&\qquad\quad \le C \Big[\|\f\|_{H^{\ell-1}(\Omega)} + \|\g\|_{H^{\ell-0.5}(\bdy\Omega)} \Big] \big\|\bp^\ell (\widetilde{\zeta} \widetilde{\u} \hspace{-.5pt}\cdot\hspace{-.5pt} \widetilde{\bfw})\big\|_{H^1(B^+_m)} \,, \label{pf_of_thm10.7_inequality1}
\end{align}
where we have used Young's inequality for convolution to conclude the last inequality. We focus now on the left-hand side of (\ref{pf_of_thm10.7_inequality1}). As in the proof of Theorem \ref{thm:vector-valued_elliptic_regularity}, we have that
$$
(\widetilde{\u}, \widetilde{\Varphi})_{L^2(B^+_m)} = \big\|\bp^\ell (\widetilde{\zeta} \widetilde{\u} \hspace{-.5pt}\cdot\hspace{-.5pt} \widetilde{\bfw})\big\|^2_{L^2(B^+_m)}\,.
$$
Moreover,
\begin{align}
\int_\Omega b^{rs}_\epsilon \frac{\p \widetilde{\u}^i}{\p y_s} \frac{\p \widetilde{\Varphi}^i}{\p y_r} \,dy &= (-1)^\ell \int_{B^+_m} b^{rs}_\epsilon \widetilde{\u}^i,_s \big[ \widetilde{\zeta} \widetilde{\bfw}^i \Delta_0^\ell (\widetilde{\zeta} \widetilde{\u} \hspace{-1pt}\cdot \widetilde{\bfw})\big],_r dy \nonumber\\
&= \int_{B^+_m} \bp^\ell \big[b^{rs}_\epsilon (\widetilde{\zeta} \widetilde{\u} \hspace{-.5pt}\cdot\hspace{-.5pt} \widetilde{\bfw}),_s\big] \bp^\ell (\widetilde{\zeta} \widetilde{\u} \hspace{-1pt}\cdot \widetilde{\bfw}),_r dy \nonumber\\
&\quad - \int_{B^+_m} \bp^\ell \big[b^{rs}_\epsilon \widetilde{\u}^i (\widetilde{\zeta} \widetilde{\bfw}^i),_s\big] \bp^\ell (\widetilde{\zeta} \widetilde{\u} \hspace{-1pt}\cdot \widetilde{\bfw}),_r dy \nonumber\\
&\quad - \int_{B^+_m} \bp^{\ell-1} \big[b^{rs}_\epsilon \widetilde{\u}^i,_s (\widetilde{\zeta} \widetilde{\bfw}^i),_r\big] \bp^{\ell+1} (\widetilde{\zeta} \widetilde{\u} \hspace{-1pt}\cdot \widetilde{\bfw})\big] dy \,. \label{vector-valued_elliptic_proof_id1}
\end{align}
For the first term on the right-hand side of (\ref{vector-valued_elliptic_proof_id1}), we make use of the positivity condition (\ref{be_elliptic}) and Young's inequality to conclude that
\begin{align}
& \int_{B^+_m} \bp^\ell \big[b^{rs}_\epsilon (\widetilde{\zeta} \widetilde{\u} \hspace{-.5pt}\cdot\hspace{-.5pt} \widetilde{\bfw}),_s\big] \bp^\ell (\widetilde{\zeta} \widetilde{\u} \hspace{-.5pt}\cdot\hspace{-.5pt} \widetilde{\bfw}),_r dy \nonumber\\
&\qquad = \int_{B^+_m} b^{rs}_\epsilon \bp^\ell (\widetilde{\zeta} \widetilde{\u} \hspace{-.5pt}\cdot\hspace{-.5pt} \widetilde{\bfw}),_s \bp^\ell (\widetilde{\zeta} \widetilde{\u} \hspace{-.5pt}\cdot\hspace{-.5pt} \widetilde{\bfw}),_r dy + \int_{B^+_m} \big[\comm{\bp^\ell}{b^{rs}_\epsilon}(\widetilde{\zeta} \widetilde{\u} \hspace{-.5pt}\cdot\hspace{-.5pt} \widetilde{\bfw}),_s\big] \bp^\ell (\widetilde{\zeta} \widetilde{\u} \hspace{-.5pt}\cdot\hspace{-.5pt} \widetilde{\bfw}),_r dy \nonumber\\
&\qquad \ge \big(\frac{\lambda}{8} - \delta\big) \big\|\bp^\ell \nabla (\widetilde{\zeta} \widetilde{\u} \hspace{-.5pt}\cdot\hspace{-.5pt} \widetilde{\bfw})\big\|^2_{L^2(B^+_m)} - C_\delta \big\|\comm{\bp^\ell}{b_\epsilon}\nabla (\widetilde{\zeta} \widetilde{\u} \hspace{-.5pt}\cdot\hspace{-.5pt} \widetilde{\bfw})\big\|^2_{L^2(B^+_m)}\,. \label{vector-valued_elliptic_proof_inequality1}
\end{align}
Then, Corollary \ref{cor:useful_lemma_with_Sobolev_class_coeff} with $\sigma = \dfrac{1}{8}$ shows that
\begin{align*}
& \int_{B^+_m} \bp^\ell \big[b^{rs}_\epsilon (\widetilde{\zeta} \widetilde{\u} \hspace{-.5pt}\cdot\hspace{-.5pt} \widetilde{\bfw}),_s\big] \bp^\ell (\widetilde{\zeta} \widetilde{\u} \hspace{-.5pt}\cdot\hspace{-.5pt} \widetilde{\bfw}),_r dy \ge \big(\frac{\lambda}{8} - \delta\big) \big\|\bp^\ell \nabla (\widetilde{\zeta} \widetilde{\u} \hspace{-.5pt}\cdot\hspace{-.5pt} \widetilde{\bfw})\big\|^2_{L^2(B^+_m)} - C_\delta \|a\|^2_{H^\rk(\Omega)}\|\u^\epsilon\|^2_{H^{\ell+\frac{7}{8}}(\Omega)} \,.
\end{align*}
For the second and the third terms on the right-hand side of (\ref{vector-valued_elliptic_proof_id1}), we use the inequality (\ref{HkHl_product}), and find that
\begin{align}
& \Big|\int_{B^+_m} \bp^\ell \big[b^{rs}_\epsilon \widetilde{\u}^i (\widetilde{\zeta} \widetilde{\bfw}^i),_s\big] \bp^\ell (\widetilde{\zeta} \widetilde{\u} \hspace{-.5pt}\cdot\hspace{-.5pt} \widetilde{\bfw}),_r dy\Big| + \Big|\int_{B^+_m} \bp^{\ell-1} \big[b^{rs}_\epsilon \widetilde{\u}^i,_s (\widetilde{\zeta} \widetilde{\bfw}^i),_r\big] \bp^{\ell+1} (\widetilde{\zeta} \widetilde{\u} \hspace{-.5pt}\cdot\hspace{-.5pt} \widetilde{\bfw})\,dy \Big| \nonumber\\
&\qquad\qquad\quad \le C_\delta \|a\|^2_{H^\rk(\Omega)} \|\u^\epsilon\|^2_{H^\ell(\Omega)} + \delta \big\|\bp^\ell \nabla (\widetilde{\zeta} \widetilde{\u}^i \widetilde{\bfw}^i)\big\|^2_{L^2(B^+_m)} \label{vector-valued_elliptic_proof_inequality2}
\end{align}
for some constant $C_\delta$ depending on $\|\bfw\|_{H^{\max\{\rk,\ell+1\}}(\Omega)}$.

Choosing $\delta > 0$ sufficiently small in (\ref{vector-valued_elliptic_proof_inequality1}) and (\ref{vector-valued_elliptic_proof_inequality2}), we conclude that
\begin{equation}\label{vector-valued_elliptic_est_temp2}
\begin{array}{l}
\displaystyle{} \big\|\bp^\ell (\widetilde{\zeta} \widetilde{\u} \hspace{-.5pt}\cdot\hspace{-.5pt} \widetilde{\bfw})\big\|_{L^2(B^+_m)} + \big\|\bp^\ell \nabla (\widetilde{\zeta} \widetilde{\u} \hspace{-.5pt}\cdot\hspace{-.5pt} \widetilde{\bfw})\big\|_{L^2(B^+_m)} \vspace{.2cm}\\
\displaystyle{} \qquad\quad \le C \Big[\|\f\|_{H^{\ell-1}(\Omega)} + \|\g\|_{H^{\ell-0.5}(\bdy\Omega)} + \|a\|_{H^\rk(\Omega)} \|\u^\epsilon\|_{H^{\ell+\frac{7}{8}}(\Omega)} \Big]
\end{array}
\end{equation}
for some constant $C = C\big(\|\bfw\|_{H^{\max\{\rk,\ell+1\}}(\Omega)}\big)$.

\vspace{.1 in}
\noindent {\bf Step 2: (Regularity for tangential derivatives of $\u$ in the $\bfw^\perp$-directions  near $\bdy\Omega$)}
Now we estimate $\u^\epsilon$ in the directions perpendicular to $\bfw$. Similar to Step 3 in the proof of Theorem \ref{thm:vector-valued_elliptic_regularity}, we use
$$
\widetilde{\Varphi}^i = 
\widetilde{\zeta} \Delta_0^\ell (\widetilde{\zeta} \widetilde{\u}) - \big(\widetilde{\zeta} \widetilde{\bfw} \hspace{-.5pt}\cdot\hspace{-.5pt} \Delta_0^\ell (\widetilde{\zeta} \widetilde{\u}) \big)\smallexp{$\displaystyle{}\frac{\widetilde{\bfw}}{|\widetilde{\bfw}|^2}$}
$$
as a test function in (\ref{transformed_weak_form_vector-valued_elliptic}) and find that
\begin{align*}
& \big\|\bp^\ell (\widetilde{\zeta} \widetilde{\u})\big\|^2_{L^2(B^+_m)} + \big(\frac{\lambda}{8} - \delta\big) \big\|\bp^\ell \nabla (\widetilde{\zeta} \widetilde{\u})\big\|^2_{L^2(B^+_m)} \\
&\qquad\quad \le C \Big[\|\f\|^2_{H^{\ell-1}(\Omega)} + \|\g\|^2_{H^{\ell-0.5}(\bdy\Omega)}\Big]+ C_\delta \|a\|^2_{H^\rk (\Omega)} \|\u^\epsilon\|^2_{H^{\ell+\frac{7}{8}}(\Omega)} \\
&\qquad\qquad + (-1)^{\ell} \smallexp{$\displaystyle{}\int_{B^+_m}$}\, \frac{\widetilde{\bfw}^j}{|\widetilde{\bfw}|^2} (\widetilde{\zeta} \widetilde{\u}\cdot \widetilde{\bfw}) \Delta_0^\ell (\widetilde{\zeta} \widetilde{\u}^j)\, dy \\
&\qquad\qquad + (-1)^{\ell+1} \smallexp{$\displaystyle{}\int_{B^+_m}$} b^{rs}_\epsilon \widetilde{\u}^i,_s \Big[\widetilde{\zeta} \Delta_0^\ell (\widetilde{\zeta} \widetilde{\u}^j) \dfrac{\widetilde{\bfw}^j \widetilde{\bfw}^i}{|\widetilde{\bfw}|^2} \Big],_r dy.
\end{align*}
Integrating-by-parts  in the horizontal direction (using $\bp$), by Corollary \ref{cor:useful_lemma_with_Sobolev_class_coeff} and (\ref{vector-valued_elliptic_est_temp2}) we obtain that
\begin{align}
& (-1)^{\ell} \smallexp{$\displaystyle{}\int_{B^+_m}$}\, \frac{\widetilde{\bfw}^j}{|\widetilde{\bfw}|^2} (\widetilde{\zeta} \widetilde{\u}\cdot \widetilde{\bfw}) \p^{2\ell} (\widetilde{\zeta} \widetilde{\u}^j)\, dy \nonumber\\
&\qquad\quad \le \Big\|\bp^{\ell-1} \Big(\dfrac{\widetilde{\bfw}^j}{|\widetilde{\bfw}|^2} (\widetilde{\zeta} \widetilde{\u}\cdot \widetilde{\bfw})\Big)\Big\|_{L^2(B^+_m)}\big\|\bp^{\ell+1} (\widetilde{\zeta} \widetilde{u}^j)\big\|^2_{L^2(B^+_m)} \nonumber\\
&\qquad\quad \le C_\delta \|\u^\epsilon\|^2_{H^{\ell-1}(\Omega)} + \delta \|\bp^\ell \nabla (\widetilde{\zeta} \widetilde{u}^j)\big\|^2_{L^2(B^+_m)} \label{vector-valued_elliptic_proof_inequality3}
\end{align}
and
\begin{align}
& (-1)^{\ell+1} \int_{B^+_m} b^{rs}_\epsilon \widetilde{\u}^i,_s \Big[\widetilde{\zeta} \p^{2\ell} (\widetilde{\zeta} \widetilde{\u}^j) \smallexp{$\displaystyle{}\frac{\widetilde{\bfw}^j \widetilde{\bfw}^i}{|\widetilde{\bfw}|^2}$}\Big],_r dy \nonumber\\
&\qquad\quad \le C \Big[\big\|\bp^\ell (b_\epsilon \nabla \widetilde{\u})\big\|_{L^2(B^+_m)} + \big\|b_\epsilon \bp^\ell \nabla (\widetilde{\zeta} \widetilde{\u} \hspace{-.5pt}\cdot\hspace{-.5pt} \widetilde{\bfw})\big\|_{L^2(B^+_m)} \nonumber\\
&\qquad\qquad\quad + \big\|\comm{\bp^\ell}{b_\epsilon} \nabla (\widetilde{\zeta} \widetilde{\u} \hspace{-.5pt}\cdot\hspace{-.5pt} \widetilde{\bfw})\big\|_{L^2(B^+_m} \Big] \big\|\bp^\ell \nabla (\widetilde{\zeta} \widetilde{\u})\big\|_{L^2(B^+_m)} \nonumber\\
&\qquad\quad \le C_\delta \|a\|^2_{H^\rk(\Omega)} \|\u^\epsilon\|^2_{H^{\ell+\frac{7}{8}}(\Omega)} + \delta \big\|\bp^\ell \nabla (\widetilde{\zeta} \widetilde{\u})\big\|^2_{L^2(B^+_m)}\,, \label{vector-valued_elliptic_proof_inequality4}
\end{align}
in which the constant $C_\delta$ also depends on $\|\bfw\|_{H^{\max\{\rk,\ell+1\}}(\Omega)}$. Therefore, choosing $\delta > 0$ sufficiently small in (\ref{vector-valued_elliptic_proof_inequality3}) and (\ref{vector-valued_elliptic_proof_inequality4}), we conclude that
\begin{equation}\label{vector-valued_ellptic_bdy_est_with_Sobolev_coeff_temp}
\begin{array}{l}
\displaystyle{} \big\|\widetilde{\zeta}\, \bp^\ell \widetilde{\u}^i\big\|_{L^2(B^+_m)} + \big\|\widetilde{\zeta}\, \bp^\ell \nabla \widetilde{\u}^i\big\|_{L^2(B^+_m)} \vspace{.1in}\\
\displaystyle{} \qquad \le C \Big[\|\u^\epsilon\|_{H^{\ell-1}(\Omega)} + \|\f\|_{H^{\ell-1}(\Omega)} + \|\g\|_{H^{\ell-0.5}(\bdy\Omega)} + \|a\|_{H^\rk(\Omega)} \|\u^\epsilon\|_{H^{\ell+\frac{7}{8}}(\Omega)} \Big]
\end{array}
\end{equation}
for some constant $C = C\big(\|\bfw\|_{H^{\max\{\rk,\ell+1\}}(\Omega)}\big)$.

\vspace{.1 in}
\noindent {\bf Step 3: (Regularity for normal derivatives of $\u$ near $\bdy\Omega$)}
In this step, we follow the procedure of Step 4 in the proof of Theorem \ref{thm:vector-valued_elliptic_regularity}. Since $\u^\epsilon$ is a strong solution, it follows that
$$
\u^\epsilon - \frac{\p}{\p x_j} \Big(a^{jk}_\epsilon \frac{\p \u^\epsilon}{\p x_k}\Big) = \f_{\hspace{-1.5pt}\epsilon} \quad \text{in}\quad \Omega\,;
$$
thus the Piola identity (\ref{Piola_id}) implies that
$$
\widetilde{\zeta} \big(b^{rs}_\epsilon \widetilde{\u},_s \big),_r = \widetilde{\zeta} \big(\widetilde{\u} - (\f_{\hspace{-1.5pt}\epsilon} \circ \vartheta)\big) \quad\text{in}\quad B^+_m\,.
$$
With $\widetilde{\u},_n$ and $\widetilde{\u},_{\n\n}$ denoting $\dfrac{\p \widetilde{\u}}{\p y_\n}$ and $\dfrac{\p^2 \widetilde{\u}}{\p y_\n^2}$, respectively, we have that
\begin{equation}\label{pf_of_thm12.12_id2}
\widetilde{\zeta} b^{\n\n}_\epsilon \widetilde{\u},_{\n\n} = \widetilde{\zeta}\, \Big[ \widetilde{\u} - (\f_{\hspace{-1.5pt}\epsilon} \circ \vartheta) - b^{\n\n}_{\epsilon,\n} \widetilde{\u},_\n - \sum_{(r,s)\ne (\n,\n)} b^{rs}_{\epsilon,r} \widetilde{\u},_s - \sum_{(r,s)\ne (\n,\n)} b^{rs}_\epsilon \widetilde{\u},_{sr} \Big] \qquad\text{in}\quad B^+_m\,.
\end{equation}
Let $\text{\bf\emph{G}} = \widetilde{\zeta}\, \Big[ \widetilde{\u} - (\f_{\hspace{-1.5pt}\epsilon} \circ \vartheta) - b^{\n\n}_{\epsilon,\n} \widetilde{\u},_\n - \sum\limits_{(r,s)\ne (\n,\n)} b^{rs}_{\epsilon,r} \widetilde{\u},_s\Big]$, and for $0\le j \le \ell-1$ we define
$$
\text{\bf\emph{G}}_{(\ell,j)} = \bp^{\ell-1-j} \nabla^j \text{\bf\emph{G}} - \comm{\bp^{\ell-1-j}\nabla^j}{b^{rs}_\epsilon} \widetilde{\u},_{sr}\,.
$$
Letting $\bp^{\ell-1-j}\nabla^j$ act on (\ref{pf_of_thm12.12_id2}), we obtain that
\begin{equation}\label{pf_of_thm12.12_id3}
\widetilde{\zeta} b^{\n\n}_\epsilon \bp^{\ell-1-j}\nabla^j \widetilde{\u},_{\n\n} = \text{\bf\emph{G}}_{(\ell,j)} - \sum_{(r,s) \ne (\n,\n)} \widetilde{\zeta} b^{rs}_\epsilon \bp^{\ell-1-j}\nabla^j \widetilde{\u},_{rs}
\end{equation}
Now we estimate $\text{\bf\emph{G}}_{(\ell,j)}$ in $L^2(B^+_m)$. First we note that
$$
\big\|\bp^{\ell-1-j} \nabla^j \big[\widetilde{\zeta} (\widetilde{\u} - \f_{\hspace{-1.5pt}\epsilon} \circ \vartheta)\big]\big\|_{L^2(B^+_m)} \le C \Big[ \|\u^\epsilon\|_{H^{\ell-1}(\Omega)} + \|\f\|_{H^{\ell-1}(\Omega)} \Big]\,.
$$
Moreover, since $\ell \le \rk$, by Proposition \ref{prop:HkHl_product} with $\sigma = \smallexp{$\displaystyle{}\frac{1}{8}$}$ we find that
\begin{align*}
& \big\|\bp^{\ell-1-j} \nabla^j (\widetilde{\zeta} b^{\n\n}_{,\n} \widetilde{\u},_\n)\big\|_{L^2(B^+_m)} + \sum_{(r,s)\ne (\n,\n)} \big\|\bp^{\ell-1-j} \nabla^j (b^{rs}_{\epsilon,r} \widetilde{\u},_s)\big\|_{L^2(B^+_m)} \\
&\qquad \le C \sum_{j=0}^{\ell-1} \|\nabla^{j+1} a \nabla^{\ell-1-j} \u^\epsilon\|_{L^2(\Omega)} = C \sum_{j=1}^\ell \|\nabla^j a \nabla^{\ell-j} \u^\epsilon \|_{L^2(\Omega)} \\
&\qquad \le C \|a\|_{H^\rk(\Omega)} \|\u\|_{H^{\ell+\frac{7}{8}}(\Omega)}\,.
\end{align*}
Finally, by Corollary \ref{cor:useful_lemma_with_Sobolev_class_coeff} with $\sigma = \dfrac{1}{8}$,
\begin{align*}
\big\|\comm{\bp^{\ell-1-j}\nabla^j}{\widetilde{\zeta} b^{\n\n}_\epsilon} \widetilde{\u},_{\n\n}\big\|_{L^2(B^+_m)} &+ \sum_{(r,s) \ne (\n,\n)} \big\| \comm{\bp^{\ell-1-j}\nabla^j}{\widetilde{\zeta} b^{rs}_\epsilon} \widetilde{\u},_{rs}\big\|_{L^2(B^+_m)} \\
&\le C_\epsilon \|a\|_{H^\rk(\Omega)} \|\u^\epsilon\|_{H^{\ell+\frac{7}{8}}(\Omega)}\,.
\end{align*}
Therefore, $\text{\bf\emph{G}}_{(\ell,j)}$ satisfies
$$
\|\text{\bf\emph{G}}_{(\ell,j)}\|_{L^2(B^+_m)} \le C \Big[\|\u\|_{H^{\ell-1}(\Omega)} + \|\f\|_{H^{\ell-1}(\Omega)} + \|a\|_{H^\rk(\Omega)} \|\u^\epsilon\|_{H^{\ell+\frac{7}{8}}(\Omega)} \Big]\,.
$$
Now we argue by induction on $0\le j\le \ell-1$. By the positivity condition (\ref{uniform_ellipticity_of_tilde_a}), $b^{\n\n}_\epsilon \ge \dfrac{\lambda}{4}$ so that when $j=0$, the inequalities (\ref{vector-valued_ellptic_bdy_est_with_Sobolev_coeff_temp}) and (\ref{pf_of_thm12.12_id3}) show that
\begin{align*}
\smallexp{$\displaystyle{}\big\|\widetilde{\zeta}\, \bp^{\ell-1} \widetilde{\u},_{\n\n}\big\|_{L^2(B^+_m)}$} \hspace{-1pt} &\smallexp{$\displaystyle{}\hspace{-1pt}\le \|\text{\bf\emph{G}}_{(\ell,j)}\|_{L^2(B^+_m)} \hspace{-2pt}+ \hspace{-2pt} \sum_{(r,s) \ne (\n,\n)} \hspace{-2pt}\|b^{rs}_\epsilon\|_{L^\infty(B^+_m)} \big\|\widetilde{\zeta}\, \bp^{\ell-1} \widetilde{\u},_{rs}\big\|_{L^2(B^+_m)}$} \\
&\smallexp{$\displaystyle{}\hspace{-1pt}\le C \Big[\|\u\|_{H^{\ell-1}(\Omega)} + \|\f\|_{H^{\ell-1}(\Omega)} + \|\g\|_{H^{\ell-0.5}(\bdy\Omega)} + \|a\|_{H^\rk(\Omega)} \|\u^\epsilon\|_{H^{\ell+\frac{7}{8}}(\Omega)} \Big]$}
\end{align*}
which, combined with (\ref{vector-valued_ellptic_bdy_est_with_Sobolev_coeff_temp}), provides the estimate
\begin{align*}
&\big\|\widetilde{\zeta}\, \bp^{\ell-1} \nabla^2 \widetilde{\u}\big\|_{L^2(B^+_m)} \\
&\qquad\quad \le C \Big[\|\u\|_{H^{\ell-1}(\Omega)} + \|\f\|_{H^{\ell-1}(\Omega)} + \|\g\|_{H^{\ell-0.5}(\bdy\Omega)} + \|a\|_{H^\rk(\Omega)} \|\u^\epsilon\|_{H^{\ell+\frac{7}{8}}(\Omega)} \Big]\,.
\end{align*}
Repeating this process for $j=1,\cdots,\ell-1$, we conclude that
\begin{equation}\label{vector-valued_ellptic_bdy_est_with_Sobolev_coeff}
\begin{array}{l}
\displaystyle{} \big\|\widetilde{\zeta} \nabla^\ell \widetilde{\u}^i\big\|_{L^2(B^+_m)} + \big\|\widetilde{\zeta} \nabla^{\ell+1} \widetilde{\u}^i\big\|_{L^2(B^+_m)} \vspace{.2cm}\\
\displaystyle{} \qquad \le C \Big[\|\u^\epsilon\|_{H^{\ell-1}(\Omega)} + \f\|_{H^{\ell-1}(\Omega)} + \|\g\|_{H^{\ell-0.5}(\bdy\Omega)} + \|a\|_{H^\rk(\Omega)} \|\u^\epsilon\|_{H^{\ell+\frac{7}{8}}(\Omega)} \Big]
\end{array}
\end{equation}
for some constant $C = C\big(\|\bfw\|_{H^{\max\{\rk,\ell+1\}}(\Omega)}\big)$.

\vspace{.1 in}
\noindent {\bf Step 4: (Completing the regularity theory)}
Let $\chi\ge 0$ be a smooth cut-off function so that $\supp(\chi) \cptsubset \Omega$.  Arguing as in Step 1 of the proof of Theorem \ref{thm:vector-valued_elliptic_regularity}, we find that
\begin{equation}\label{vector-valued_ellptic_int_est_with_Sobolev_coeff}
\|\chi \nabla^\ell \u^\epsilon\|_{L^2(\Omega)} + \|\chi \nabla^{\ell+1} \u^\epsilon\|_{L^2(\Omega)} \le C \Big[\|\f\|_{H^{\ell-1}(\Omega)} + \|a\|_{H^\rk(\Omega)} \|\u^\epsilon\|_{H^{\ell+\frac{7}{8}}(\Omega)} \Big] \,.
\end{equation}
Combining (\ref{vector-valued_ellptic_bdy_est_with_Sobolev_coeff}) and (\ref{vector-valued_ellptic_int_est_with_Sobolev_coeff}) establishes the inequality
\begin{equation}\label{vector-valued_ellptic_est_with_Sobolev_coeff}
\|\u^\epsilon\|_{H^{\ell+1}(\Omega)} \le C \Big[\|\f\|_{H^{\ell-1}(\Omega)} + \|\g\|_{H^{\ell-0.5}(\bdy\Omega)} + \big(1+\|a\|_{H^\rk(\Omega)}\big) \|\u^\epsilon\|_{H^{\ell+\frac{7}{8}}(\Omega)} \Big]
\end{equation}
for some constant $C = C\big(\|\bfw\|_{H^{\max\{\rk,\ell+1\}}(\Omega)}\big)$. Since the interpolation inequality provides
\begin{align*}
\|\u^\epsilon\|_{H^{\ell+\frac{7}{8}}(\Omega)} &\le C \|\u^\epsilon\|^{1-\frac{1}{8\ell}}_{H^{\ell+1}(\Omega)} \|\u^\epsilon\|^{\frac{1}{8\ell}}_{H^1(\Omega)} \,,
\end{align*}
Young's inequality further shows that
\begin{equation}\label{cs1001}
\|\u^\epsilon\|_{H^{\ell+1}(\Omega)} \le C_\delta \Big[\|\f\|_{H^{\ell-1}(\Omega)} + \|\g\|_{H^{\ell-0.5}(\bdy\Omega)} + \P\big(\|a\|_{H^\rk(\Omega)}\big) \|\u^\epsilon\|_{H^1(\Omega)} \Big] + \delta \|\u^\epsilon\|_{H^{\ell+1}(\Omega)}
\end{equation}
for some polynomial function $\P$. Finally, choosing $\delta > 0$ sufficiently small and then passing to the limit as $\epsilon \to 0$, by the fact that
\begin{alignat*}{2}
a^{jk}_\epsilon &\to a^{jk} \qquad&&\text{in}\quad H^\rk(\Omega) \,,\\
\bfw_\epsilon &\to \bfw &&\text{in}\quad H^{\max\{k,\ell+1\}}(\Omega)\,,\\
\f_{\hspace{-1.5pt}\epsilon} &\to \f &&\text{in}\quad H^{\ell-1}(\Omega)\,,\\
\g_\epsilon &\to \g &&\text{in}\quad H^{\ell-0.5}(\bdy\Omega)\,,
\end{alignat*}
we find that $\u^\epsilon$ converges to the unique weak solution $\u$ to (\ref{vector-valued_elliptic_eq}), and the inequality  (\ref{vector-valued_elliptic_regularity}) is established by substitution of the $H^1$-estimate (\ref{vector-valued_elliptic_weak_est}) in the inequality
(\ref{cs1001}).
\end{proof}

Having established the regularity theory for the case that $a^{jk} \in H^\rk(\Omega)$ and $\Omega\in \mC^ \infty $, we can now prove the
following
\begin{corollary}[Regularity for the case that $a^{jk} \in H^\rk(\Omega)$ and $\Omega \in H^{\rk+1} $]\label{cor:vector-valued_elliptic_eq_Sobolev_coeff}
Let $\Omega \subseteq \Rn$ be a bounded $H^{\rk+1}$-domain for some integer $\rk > \novertwo$\,. Suppose that $a^{jk} \in H^\rk(\Omega)$ satisfies the positivity condition
$$
a^{jk} \xi_j \xi_k \ge \lambda |\xi|^2\qquad\Forall \xi \in \Rn\,,
$$
and for some $1\le \ell\le \rk$, $\bfw \in H^{\max\{\rk,\ell+1\}}(\Omega)$ $($or $\bfw \in H^{\max\{\rk-\frac{1}{2},\ell+\frac{1}{2}\}}(\bdy\Omega))$ such that $\bfw$ vanishes nowhere on $\bdy\Omega$. Then for all $\f\in H^{\ell-1}(\Omega)$ and $\g \in H^{\ell-0.5}(\bdy\Omega)$, the weak solution $\u$ to {\rm(\ref{vector-valued_elliptic_eq})} belongs to $H^{\ell+1}(\Omega)$, and satisfies
\begin{equation}\label{vector-valued_elliptic_regularity_sobolev}
\begin{array}{l}
\displaystyle{} \|\u\|_{H^{\ell+1}(\Omega)} \le C \Big[\|\f\|_{H^{\ell-1}(\Omega)} + \|\g\|_{H^{\ell-0.5}(\bdy\Omega)} \vspace{.2cm}\\
\displaystyle{} \hspace{75pt} + \P\big(\|a\|_{H^\rk(\Omega)}\big) \Big(\|\f\|_{L^2(\Omega)} + \|\g\|_{H^{-0.5}(\bdy\Omega)}\Big) \Big]
\end{array}
\end{equation}
for some constant $C = C\big(\|\bfw\|_{H^{\max\{\rk,\ell+1\}}(\Omega)}, |\bdy\Omega|_{H^{k+0.5}}\big)$ and some polynomial $\P$.
\end{corollary}
\begin{proof}  Using Definition \ref{defn:Hs_domain}, we
let $\psi:\cls{\rO} \to \cls{\Omega}$ be an $H^{\rk+1}$-diffeomorphism, where $\rO$ is a bounded $\mC^\infty$-domain.
Making the change-of-variables $x = \psi(y)$,
with $\rA$ denoting $(\nabla \psi)^{-1}$ we can rewrite (\ref{vector-valued_elliptic_eq}) as
\begin{alignat*}{2}
\widebar{\u} - \frac{\p}{\p y_r} \Big(\widebar{a}^{jk} \rA^r_j \rA^s_k \frac{\p \widebar{\u}}{\p y_s} \Big) &= \widebar{\f} + \widebar{a}^{jk} \rA^s_k \frac{\p \rA^r_j}{\p y_r} \frac{\p \widebar{\u}}{\p y_s} \qquad &&\text{in}\quad \rO\,,\\
\widebar{\u} \hspace{-.5pt}\cdot\hspace{-.5pt} \widebar{\bfw} &= 0 &&\text{on}\quad \bdy\rO\,,\\
\rP_{\widebar{\bfw}^\perp} \Big(\widebar{a}^{jk} \rA^r_j \rA^s_k \frac{\p \widebar{\u}}{\p y_s} \widebar{\bN}_r - \widebar{\g}\Big) &= {\bf 0} &&\text{on}\quad\bdy\rO\,,
\end{alignat*}
where we use the bar notation to denote the variable defined on $\rO$ through the composition with $\psi$:
\begin{align*}
\widebar{a} &= a \circ \psi\,,\quad \widebar{\u} = \u\circ\psi\,,\quad \widebar{\bfw} = \bfw \circ\psi\,,\quad \widebar{\f} = \f\circ\psi\,, \quad \widebar{\g} = \g \circ \psi\,,
\end{align*}
and $\widebar{\bN}$ is the outward-pointing unit normal to $\rO$.
By Proposition \ref{prop:HkHl_product}, Corollary \ref{cor:JA_est}, and Corollary \ref{cor:f_comp_psi}, we find that
\begin{align*}
\|\widebar{a}^{jk} \rA^s_k \rA^r_j\|_{H^\rk(\rO)} &\le C(|\bdy\Omega|_{H^{k+0.5}}) \|a\|_{H^\rk(\Omega)} \,, \\
\|\widebar{\bfw}\|_{H^{\max\{\rk,\ell+1\}}(\Omega)} &\le C(|\bdy\Omega|_{H^{k+0.5}}) \|\bfw\|_{H^{\max\{\rk,\ell+1\}}(\Omega)}\,,\\
\|\widebar{\f}\|_{H^{\ell-1}(\rO)} + \|\widebar{\g}\|_{H^{\ell-0.5}(\bdy\rO)} &\le C(|\bdy\Omega|_{H^{k+0.5}}) \Big[ \|\f\|_{H^{\ell-1}(\Omega)} + \|\g\|_{H^{\ell-0.5}(\bdy\Omega)} \Big]\,.
\end{align*}
Theorem \ref{thm:vector-valued_elliptic_eq_Sobolev_coeff} then implies that
\begin{align*}
\|\widebar{\u}\|_{H^{\ell+1}(\rO)} &\le C \Big[\|\widebar{\f}\|_{H^{\ell-1}(\rO)} + \|\widebar{\g}\|_{H^{\ell-0.5}(\bdy\rO)} \\
&\qquad + \P\big(\|\rA \widebar{a} \rA^\rT\|_{H^\rk(\rO)}\big) \Big(\|\widebar{\f}\|_{L^2(\rO)} + \|\widebar{\g}\|_{H^{-0.5}(\bdy\rO)} \Big)\Big] \\
&\le C \Big[\|\f\|_{H^{\ell-1}(\rO)} + \|\g\|_{H^{\ell-0.5}(\bdy\Omega)} \\
&\qquad + \P\big(\|a\|_{H^\rk(\Omega)}, |\bdy\Omega|_{H^{\rk+0.5}}\big) \Big(\|\f\|_{L^2(\Omega)} + \|\g\|_{H^{-0.5}(\bdy\Omega)} \Big)\Big]
\end{align*}
for some constant $C = C\big(\|\bfw\|_{H^{\max\{\rk,\ell+1\}}(\Omega)}, |\bdy\Omega|_{H^{k+0.5}}\big)$. Estimate (\ref{vector-valued_elliptic_regularity_sobolev}) then follows from Corollary \ref{cor:f_comp_psi}.
\end{proof}

\begin{corollary}[Regularity for the classical Dirichlet and Neumann problems]\label{cor:scalar_elliptic_eq_Sobolev_coeff}
Let $\Omega \subseteq \Rn$ be a bounded $H^{\rk+1}$-domain for some integer $\rk > \novertwo$, and $a^{jk} \in H^\rk(\Omega)$ satisfies the positivity condition
$$
a^{jk} \xi_j \xi_k \ge \lambda |\xi|^2\qquad\Forall \xi \in \Rn\,.
$$
Let $\ell$ be an integer such that $1\le \ell\le \rk$. Then
\begin{enumerate}
\item[\rm1.] For any $f\in H^{\ell-1}(\Omega)$, the weak solution $u\in H^1_0(\Omega)$ to the Dirichlet problem
\begin{alignat*}{2}
- \frac{\p}{\p x_j} \Big(a^{jk} \frac{\p u}{\p x_k}\Big) &= f \qquad &&\text{in}\quad \Omega\,,\\
u &= 0 &&\text{on}\quad \bdy\Omega\,,
\end{alignat*}
belongs to $H^{\ell+1}(\Omega)$, and satisfies
\begin{equation}\label{scalar_elliptic_regularity_sobolev_Dirichlet_estimate}
\|u\|_{H^{\ell+1}(\Omega)} \le C \|f\|_{H^{\ell-1}(\Omega)}
\end{equation}
for some constant $C = C\big(\|a\|_{H^\rk(\Omega)}, |\bdy\Omega|_{H^{\rk+0.5}}\big)$.
\item[\rm2.] For any $f\in H^{\ell-1}(\Omega)$ and $g\in H^{\ell-0.5}(\bdy\Omega)$, the weak solution $v\in H^1(\Omega)$ to the Neumann problem
  \begin{alignat*}{2}
  v - \frac{\p}{\p x_j} \Big(a^{jk} \frac{\p v}{\p x_k}\Big) &= f \qquad &&\text{in}\quad \Omega\,,\\
  a^{jk} \frac{\p u}{\p x_k} \bN_j &= g &&\text{on}\quad \bdy\Omega\,,
  \end{alignat*}
  belongs to $H^{\ell+1}(\Omega)$, and satisfies
  \begin{equation}\label{scalar_elliptic_regularity_sobolev_Neumann_estimate}
  \|v\|_{H^{\ell+1}(\Omega)} \le C \Big[\|f\|_{H^{\ell-1}(\Omega)} + \|g\|_{H^{\ell-0.5}(\bdy\Omega)}\Big]
  \end{equation}
for some constant $C = C\big(\|a\|_{H^\rk(\Omega)}, |\bdy\Omega|_{H^{\rk+0.5}}\big)$.
\end{enumerate}
\end{corollary}
\begin{proof}
It suffices to prove the case that $u$ and $v$ are both scalar functions.
\begin{enumerate}
\item Let $\bfw = (1,0,\cdots,0)$, and $\u $ be the solution to
  \begin{alignat}{2}
   \u  - \frac{\p}{\p x_j} \Big(a^{jk} \frac{\p \u }{\p x_k}\Big) &= (f+u,0,\cdots,0) \qquad&&\text{in}\quad\Omega\,,\\
  \u  \hspace{-.5pt}\cdot\hspace{-.5pt} \bfw &= 0 &&\text{on}\quad\bdy\Omega\,,\\
  \rP_{\bfw^\perp}\Big(a^{jk} \frac{\p \u }{\p x_k} \bN_j\Big) &= {\bf 0} &&\text{on}\quad\bdy\Omega\,.
  \end{alignat}
   Then $u = \u _1$ (in fact, $\u  = (u,0,\cdots, 0)$); thus (\ref{vector-valued_elliptic_regularity_sobolev}) implies that
  $$
   \|u\|_{H^{\ell+1}(\Omega)} \le C \|f + u\|_{H^{\ell-1}(\Omega)} \le C \Big[\|f\|_{H^{\ell-1}(\Omega)} + \|u\|_{H^{\ell-1}(\Omega)} \Big]
  $$
   for some constant $C = C\big(\|a\|_{H^\rk(\Omega)}, |\bdy\Omega|_{H^{k+0.5}}\big)$. By interpolation and Young's inequality,
  $$
   \|u\|_{H^{\ell+1}(\Omega)} \le C \|f\|_{H^{\ell-1}(\Omega)} + C_\delta \|u\|_{H^1(\Omega)} + \delta \|u\|_{H^{\ell+1}(\Omega)}\,;
  $$
   thus (\ref{scalar_elliptic_regularity_sobolev_Dirichlet_estimate}) follows from choosing $\delta > 0$ small enough and the estimate for the weak solution.
\item Let $\bfw = (0,1,0,\cdots,0)$, and $\v $ be the solution to
  \begin{alignat}{2}
   \v  - \frac{\p}{\p x_j} \Big(a^{jk} \frac{\p \v }{\p x_k}\Big) &= (0, f,0,\cdots,0) \qquad&&\text{in}\quad\Omega\,,\\
  \v  \hspace{-.5pt}\cdot\hspace{-.5pt} \bfw &= 0 &&\text{on}\quad\bdy\Omega\,,\\
   \rP_{\bfw^\perp}\Big(a^{jk} \frac{\p \v }{\p x_k} \bN_j\Big) &= (0,g,0,\cdots,0) &&\text{on}\quad\bdy\Omega\,.
  \end{alignat}
\end{enumerate}
\noindent\hspace{29pt}Then $v = \v _2$ (in fact, $\v  = (0,v,0,\cdots, 0)$); thus (\ref{scalar_elliptic_regularity_sobolev_Neumann_estimate}) follows from (\ref{vector-valued_elliptic_regularity_sobolev}).
\end{proof}

In general, elliptic estimates with Sobolev class coefficients $a^{jk}$ have a nonlinear dependence on the Sobolev norm of $a^{jk}$.  There are, however, situations
when the estimate becomes linear with respect to  the Sobolev norm of $a^{jk}$.

\begin{theorem}[Regularity estimate which is linear in the coefficient matrix $a^{jk}$] \label{thm_linear_est}
Suppose that  the assumptions of
Theorem {\rm\ref{thm:vector-valued_elliptic_eq_Sobolev_coeff}} are satisfied with $\ell = \rk$, and  that furthermore
$$
\|a - \text{\rm Id}\|_{L^\infty(\Omega)} \le \epsilon \ll 1\,.
$$
Then the solution $\u \in H^{\rk+1}(\Omega)$ to {\rm(\ref{vector-valued_elliptic_eq})} satisfies
\begin{equation}\label{vector-valued_elliptic_est2}
\|\u\|_{H^{\rk+1}(\Omega)} \le C \Big[\|\f\|_{H^{\rk-1}(\Omega)} + \|\g\|_{H^{\rk-0.5}(\bdy\Omega)} + \big(1+\|a\|_{H^\rk(\Omega)}\big) \|\nabla \u\|_{L^\infty(\Omega)} \Big]
\end{equation}
for some constant $C = C\big(\|\rw\|_{H^{\rk+1}(\Omega)}
\big)$. $($Recall that $\bfw$ is an $H^{\rk+1}(\Omega)$ vector field defined in a neighborhood of $\bdy\Omega$ which vanishes nowhere on $\bdy\Omega$.$)$
\end{theorem}

\begin{proof}
By Theorem \ref{thm:vector-valued_elliptic_eq_Sobolev_coeff} we know that $\u \in H^{\rk+1}(\Omega)$ so equation (\ref{vector-valued_elliptic_eq}) holds in the pointwise sense. We rewrite (\ref{vector-valued_elliptic_eq}) as
\begin{alignat*}{2}
\u - \Delta \u &= {\bf f} \equiv\frac{\p}{\p x_j} \Big(\big(a^{jk} - \delta^{jk}\big) \frac{\p \u}{\p x_k}\Big) + \f \qquad&&\text{in}\quad \Omega\,,\\
\u\cdot \bfw &= 0 &&\text{on}\quad\bdy\Omega\,,\\
\rP_{\bfw^\perp} \Big(\frac{\p \u}{\p \bN}\Big)&= {\bf g} \equiv \rP_{\bfw^\perp} \Big(\big(\delta^{jk} - a^{jk}\big) \frac{\p \u}{\p x_k} \bN_j + \g \Big) \qquad &&\text{on}\quad\bdy\Omega\,.
\end{alignat*}
We then conclude from Theorem \ref{thm:vector-valued_elliptic_eq_Sobolev_coeff} that
\begin{align*}
\|\u\|_{H^{\rk+1}(\Omega)} &\le C \Big[\|{\bf f}\|_{H^{\rk-1}(\Omega)} + \|{\bf g}\|_{H^{\rk-0.5}(\bdy\Omega)}\Big] \\
&\le C \Big[\|\f\|_{H^{\rk-1}(\Omega)} + \|\g\|_{H^{\rk-0.5}(\bdy\Omega)} + \big\|\smallexp{$\displaystyle{}\frac{\p}{\p x_j} \Big(\big(\delta^{jk} - a^{jk}\big) \frac{\p \u}{\p x_k}\Big)$}\big\|_{H^{\rk-1}(\Omega)} \\
&\qquad + \big\|\rP_{\bfw^\perp} \smallexp{$\displaystyle{}\Big(\big(\delta^{jk} - a^{jk}\big) \frac{\p \u}{\p x_k} \bN_j\Big)$}\big\|_{H^{\rk-0.5}(\bdy\Omega)} \Big]
\end{align*}
for some constant $C = C\big(\|\bfw\|_{H^{\rk+1}(\Omega)}\big)$. By Theorem \ref{thm:useful_lemma_with_Sobolev_class_coeff2},
\begin{align*}
& \Big\|\smallexp{$\displaystyle{}\frac{\p}{\p x_j} \Big(\big(\delta^{jk} - a^{jk}\big) \frac{\p \u}{\p x_k}\Big)$}\Big\|_{H^{\rk-1}(\Omega)} \le \Big\|\smallexp{$\displaystyle{}\big(\delta^{jk} - a^{jk}\big) \frac{\p \u}{\p x_k}$}\Big\|_{H^\rk(\Omega)} \\
&\qquad\quad \le C \Big[\|\delta - a\|_{L^\infty(\Omega)} \|\nabla \u\|_{H^\rk(\Omega)} + \|\delta - a\|_{H^\rk(\Omega)} \|\nabla \u\|_{L^\infty(\Omega)} \Big] \\
&\qquad\quad \le C \epsilon \|\u\|_{H^{\rk+1}(\Omega)} + C \big(1+\|a\|_{H^\rk(\Omega)}\big) \|\nabla \u\|_{L^\infty(\Omega)} \,.
\end{align*}
Similarly, by the trace estimate (and the fact that $\rk-0.5 > \smallexp{$\displaystyle{}\frac{\n-1}{2}$}$, where $\n-1$ is the dimension of $\bdy\Omega$),
\begin{align*}
& \Big\|\rP_{\bfw^\perp} \smallexp{$\displaystyle{}\Big(\big(\delta^{jk} - a^{jk}\big) \frac{\p \u}{\p x_k} \bN_j\Big)$}\Big\|_{H^{\rk-0.5}(\bdy\Omega)} \le C \Big\|\big(\delta^{jk} - a^{jk}\big) \frac{\p \u}{\p x_k} \bN_j\Big\|_{H^{\rk-0.5}(\bdy\Omega)} \\
&\qquad\quad \le C \Big[\|\delta - a\|_{L^\infty(\Omega)} \|\nabla \u\|_{H^{\rk-0.5}(\bdy\Omega)} + \|\delta - a\|_{H^{\rk-0.5}(\bdy\Omega)} \|\nabla \u\|_{L^\infty(\bdy\Omega)}\\
&\qquad\quad \le C \epsilon \|\u\|_{H^{\rk+1}(\Omega)} + C \big(1+\|a\|_{H^\rk(\Omega)}\big) \|\nabla \u\|_{L^\infty(\bdy\Omega)}
\end{align*}
for some constant $C = C\big(\|\bfw\|_{H^{\rk+1}(\Omega)}\big)$. The embedding $H^{\frac{\n}{2}+\delta}(\Omega) \contsubset \mC^{0,\alpha}(\Omega)$ for some $\alpha > 0$ further suggests that $\nabla \u$ is uniformly \Holder\ continuous; thus $\|\nabla \u\|_{L^\infty(\bdy\Omega)} \le \|\nabla \u\|_{L^\infty(\Omega)}$. (\ref{vector-valued_elliptic_est2}) then follows from the assumption that $\epsilon \ll 1$.
\end{proof}

\begin{remark} As we noted, inequality {\rm(\ref{vector-valued_elliptic_est2})} is linear with respect to the highest-order norms.
This permits the use of linear
interpolation to extend this inequality to fractional-order Sobolev spaces.
\end{remark}

In the same way that we proved Theorem \ref{thm:vector-valued_elliptic_eq_Sobolev_coeff}, we can prove the following complimentary result:
\begin{theorem}\label{thm:vector-valued_elliptic_eq_Sobolev_coeff_v2}
Let $\Omega \subseteq \Rn$ be a bounded $H^{\rk+1}$-domain for some integer $\rk > \novertwo$\,. Suppose that $a^{jk} \in H^\rk(\Omega)$ satisfies the positivity condition
$$
a^{jk} \xi_j \xi_k \ge \lambda |\xi|^2\qquad\Forall \xi \in \Rn\,,
$$
and for some $1\le \ell\le \rk$, $\bfw \in H^{\max\{\rk,\ell+1\}}(\Omega)$ $($or $\bfw \in H^{\max\{\rk-\frac{1}{2},\ell+\frac{1}{2}\}}(\bdy\Omega))$ such that $\bfw$ vanishes nowhere on $\bdy\Omega$. Then for all $\f\in H^{\ell-1}(\Omega)$ and $g \in H^{\ell-0.5}(\bdy\Omega)$, there exists a solution $\u$ to
\begin{subequations}\label{vector-valued_elliptic_eq2}
\begin{alignat}{2}
\u^i - \frac{\p}{\p x_j} \Big(a^{jk} \frac{\p \u^i}{\p x_k}\Big) &= \f^i \qquad&&\text{in}\quad\Omega\,,\\
\u \hspace{-.5pt}\times\hspace{-.5pt} \bfw &= {\bf 0} &&\text{on}\quad \bdy\Omega\,,\\
a^{jk} \frac{\p \u^i}{\p x_k} \bN_j \bfw^i &= g &&\text{on}\quad \bdy\Omega\,,
\end{alignat}
\end{subequations}
and satisfies
\begin{equation}\label{vector-valued_elliptic_regularity_sobolev_v2}
\|\u\|_{H^{\ell+1}(\Omega)} \le C \Big[\|\f\|_{H^{\ell-1}(\Omega)} + \|\g\|_{H^{\ell-0.5}(\bdy\Omega)} + \P\big(\|a\|_{H^\rk(\Omega)}\big) \Big(\|\f\|_{L^2(\Omega)} + \|\g\|_{H^{-0.5}(\bdy\Omega)}\Big) \Big]
\end{equation}
for some constant $C = C\big(\|\bfw\|_{H^{\max\{\rk,\ell+1\}}(\Omega)}, |\bdy\Omega|_{H^{k+0.5}}\big)$ and some polynomial $\P$.
\end{theorem}

\section{The Proof of Theorem \ref{thm:main_thm2}}\label{sec:Hodge_elliptic_estimate}
In this section, we prove our main regularity result given by Theorem \ref{thm:main_thm2}. We first establish the following lemma which is also fundamental to the proof of Theorem \ref{thm:main_thm1}.
\begin{lemma}\label{lem:tangential_component_of_dwdN}
Let $\Omega \subseteq \bbR^3$ be a bounded $H^{\rk+1}$-domain with outward-pointing unit normal $\bN$.
Then for every differentiable vector field $\text{\bf\emph{w}}:\Omega \to \bbR^3$,  the following identities hold:
\begin{subequations}\label{wid}
\begin{alignat}{2}
\rP_{\bN^\perp} \Big(\frac{\p \text{\bf\emph{w}}}{\p \bN}\Big) &= (\curl \text{\bf\emph{w}} \hspace{-.5pt}\times\hspace{-.5pt} \bN) + \bdygrad \text{\bf\emph{w}} \cdot \bN &&\text{on}\quad\bdy\Omega\,, \label{PNdwdN_id}\\
\div \w &= \frac{\p \w}{\p \bN} \cdot \bN + 2 \rH (\w \cdot \bN) + \bdydiv  (\rP_{\bN^\perp} \w) \qquad&&\text{on}\quad\bdy\Omega\,,\label{wid2}\\
\curl \w \cdot \bN &= \bdydiv  (\w \times \bN) &&\text{on}\quad\bdy\Omega\,, \label{wid1}
\end{alignat}
\end{subequations}
where $\rH$ is the mean curvature of $\bdy\Omega$ {\rm(}in local chart $(\U,\vartheta)$, $\rH$ is given by $\rH = \dfrac{1}{2} g^{\alpha\beta} b_{\alpha\beta}${\rm)}.
\end{lemma}
\begin{proof}
We define
$$
\Theta(y) = \vartheta(y_1,y_2,0) + y_3 (\bN \circ \vartheta)(y_1,y_2,0)\,,
$$
and $\G_{ij} = \Theta,_i \cdot \Theta,_j$ with inverse $\G^{ij}$. Let $\widetilde{\bN} \equiv (\bN\circ \vartheta)\big|_{y_3 = 0}$, and $\widetilde{\f} \equiv \f \circ \Theta$ if $\f\ne \bN$. Since $\Theta,_1$, $\Theta,_2 \perp \widetilde{\bN}$, for every vector $v\in \bbR^3$, $\widetilde{\v}$ can be expressed as the linear combination of $\Theta,_1$, $\Theta,_2$ and $\widetilde{\bN}$. In particular, we have
\begin{subequations}\label{decomposition_of_vector}
\begin{align}
\widetilde{\v}^i = (\widetilde{\v} \cdot \widetilde{\bN}) \widetilde{\bN}^i + (\G^{\alpha\beta} \Theta^j,_\beta \widetilde{\v}^j) \Theta^i,_\alpha \equiv \widetilde{\v}_3 \widetilde{\bN}^i + \widetilde{\v}_\alpha \Theta^i,_\alpha
\end{align}
\end{subequations}
and
$$
\f,_k \circ\, \Theta = \widetilde{\f},_3 \widetilde{\bN}^3 + \G^{\alpha\beta} \widetilde{\f},_\beta \Theta^k,_\alpha\,. \eqno{\rm(\ref{decomposition_of_vector}b)}
$$
To see (\ref{PNdwdN_id}), we first note that
\begin{equation}\label{pNdwdN_id_temp}
\frac{\p \text{\bf\emph{w}}^i}{\p \bN}\circ\vartheta \hspace{-1.5pt}=\hspace{-1.5pt} \big[\widetilde{\text{\bf\emph{w}}}_3 \widetilde{\bN}^i \hspace{-1.5pt}+\hspace{-0.5pt} \widetilde{\text{\bf\emph{w}}}_\alpha \Theta^i,_\alpha\hspace{-1.5pt}\big],_3\Big|_{y_3=0} \hspace{-1pt}=\hspace{-1pt} \big[\widetilde{\text{\bf\emph{w}}}_{3,3} \widetilde{\bN}^i \hspace{-1.5pt}+\hspace{-0.5pt} \widetilde{\text{\bf\emph{w}}}_{\alpha,3} \Theta^i,_\alpha \hspace{-1.5pt}+\hspace{-0.5pt} \widetilde{\text{\bf\emph{w}}}_\alpha \widetilde{\bN}^i,_\alpha\hspace{-1.5pt}\big]\Big|_{y_3=0};
\end{equation}
thus, since $\widetilde{\bN} \cdot \Theta,_\alpha = \widetilde{\bN} \cdot \widetilde{\bN},_\alpha = 0$,
\begin{align}
& \rP_{\bN^\perp}\Big(\frac{\p\text{\bf\emph{w}}}{\p \bN}\Big)\circ \vartheta \nonumber\\
&\qquad = \big[\widetilde{\text{\bf\emph{w}}}_{3,3} \widetilde{\bN}^i + \widetilde{\text{\bf\emph{w}}}_{\alpha,3} \Theta^i,_\alpha + \widetilde{\text{\bf\emph{w}}}_\alpha \widetilde{\bN}^i,_\alpha\big]\Big|_{y_3=0} - \big[\widetilde{\text{\bf\emph{w}}}_{3,3} \widetilde{\bN}^k + \widetilde{\text{\bf\emph{w}}}_{\alpha,3} \Theta^k,_\alpha + \widetilde{\text{\bf\emph{w}}}_\alpha \widetilde{\bN}^k,_\alpha\big]\Big|_{y_3=0} \widetilde{\bN}^k \widetilde{\bN}^i \nonumber\\
&\qquad = \widetilde{\text{\bf\emph{w}}}_{\alpha,3} \vartheta^i,_\alpha + \widetilde{\text{\bf\emph{w}}}_\alpha \widetilde{\bN}^i,_\alpha\,. \label{PNdwdN_id_temp}
\end{align}
Moreover, by the identity
$$
(\curl\text{\bf\emph{w}} \times \bN)^i = \varepsilon_{ijk} \varepsilon_{jrs} \text{\bf\emph{w}}^s,_r \bN^k = (\delta_{is} \delta_{kr} - \delta_{ir} \delta_{ks}) \text{\bf\emph{w}}^s,_r \bN^k = (\text{\bf\emph{w}}^i,_k - \text{\bf\emph{w}}^k,_i) \bN^k\,,
$$
we find that
\begin{align}
(\curl\text{\bf\emph{w}} \hspace{-1pt}\times\hspace{-1pt} \bN)^i\hspace{-1pt}\circ\hspace{-1pt}\vartheta &=\hspace{-1pt} \big[\hspace{-1pt}(\widetilde{\text{\bf\emph{w}}}^i,_3 \hspace{-1pt}\widetilde{\bN}^k \hspace{-1pt}+ \G^{\alpha\beta} \widetilde{\text{\bf\emph{w}}}^i,_\beta \hspace{-1pt}\Theta^k,_\alpha \hspace{-1pt}- \widetilde{\text{\bf\emph{w}}}^k,_3 \hspace{-1pt}\widetilde{\bN}^i \hspace{-1pt}- \G^{\alpha\beta} \widetilde{\text{\bf\emph{w}}}^k,_\beta \hspace{-1pt}\Theta^i,_\alpha\hspace{-2pt}) \widetilde{\bN}^k\big]\Big|_{y_3 = 0} \nonumber\\
&= \big[\widetilde{\text{\bf\emph{w}}}^i,_3 - \widetilde{\bN}^i \widetilde{\text{\bf\emph{w}}}^k,_3 \widetilde{\bN}^k - \G^{\alpha\beta} \Theta^i,_\alpha \widetilde{\text{\bf\emph{w}}}^k,_\beta \widetilde{\bN}^k\big]\Big|_{y_3=0} \nonumber\\
&= \big[(\widetilde{\text{\bf\emph{w}}}_3 \widetilde{\bN}^i + \widetilde{\text{\bf\emph{w}}}_\alpha \Theta^i,_\alpha),_3 - \widetilde{\bN}^i (\widetilde{\text{\bf\emph{w}}}_3 \widetilde{\bN}^k + \widetilde{\text{\bf\emph{w}}}_\alpha \Theta^k,_\alpha),_3 \widetilde{\bN}^k \nonumber\\
&\quad - \G^{\alpha\beta} \Theta^i,_\alpha (\widetilde{\text{\bf\emph{w}}}_3 \widetilde{\bN}^k + \widetilde{\text{\bf\emph{w}}}_\gamma \Theta^k,_\gamma),_\beta \widetilde{\bN}^k\big] \Big|_{y_3 = 0} \nonumber\\
&= \widetilde{\text{\bf\emph{w}}}_\alpha \widetilde{\bN}^i,_\alpha + \widetilde{\text{\bf\emph{w}}}_{\alpha,3} \vartheta^i,_\alpha - g^{\alpha\beta} \vartheta^i,_\alpha (\widetilde{\text{\bf\emph{w}}}_{3,\beta} - \widetilde{\text{\bf\emph{w}}}_\gamma b_{\gamma\beta}) \nonumber\\
&= \widetilde{\text{\bf\emph{w}}}_\alpha \widetilde{\bN}^i,_\alpha + \widetilde{\text{\bf\emph{w}}}_{\alpha,3} \vartheta^i,_\alpha - g^{\alpha\beta} \vartheta^i,_\alpha (\widetilde{\text{\bf\emph{w}}}_{,\beta} \cdot \widetilde{\bN})\,.
\label{wid3}
\end{align}
Combining (\ref{PNdwdN_id_temp}) and (\ref{wid3}), we conclude (\ref{PNdwdN_id}).

Now we proceed to the proof of identity (\ref{wid2}). Using (\ref{decomposition_of_vector}) we obtain that
\begin{align*}
\div \w \big|_{\bdy\Omega}\circ\vartheta 
&= \big[\wtrN^k (\widetilde{\w}_3 \wtrN^k + \widetilde{\w}_\alpha \Theta^k,_\alpha),_3 + \G^{\alpha\beta} \Theta^k,_\alpha (\widetilde{\w}_3 \wtrN^k \hspace{-1pt} + \widetilde{\w}_\gamma \Theta^k,_\gamma\hspace{-2pt}),_\beta\big] \Big|_{y_3=0} \nonumber\\
& = \widetilde{\w}_{3,3} + g^{\alpha\beta} \vartheta^k,_\alpha (\widetilde{\w}_\gamma,_\beta \vartheta^k,_\gamma + \widetilde{\w}_\gamma \vartheta^k,_{\beta\gamma} + \wtrN^k,_\beta \widetilde{\w}_3 + \wtrN^k \widetilde{\w}_3,_\beta) \nonumber\\
& = \widetilde{\w}_{3,3} + \widetilde{\w}_{\gamma,\gamma} + \Gamma^\beta_{\beta\gamma} \widetilde{\w}_\gamma + 2 \rH \widetilde{\w}_3 \,, 
\end{align*}
where $\Gamma_{\alpha\beta}^\gamma$ is the Christoffel symbol defined by
$$
\Gamma_{\alpha\beta}^\gamma = \frac{1}{2}\, g^{\gamma\delta} (g_{\alpha\delta,\beta} + g_{\beta\delta,\alpha} - g_{\alpha\beta,\delta}) = g^{\gamma\delta} \vartheta_{,\alpha\beta}\cdot \vartheta,_\delta\,.
$$
Since $\rg,_\alpha = \rg g^{\gamma\delta} g_{\gamma\delta,\alpha} = \rg \Gamma^\beta_{\alpha\beta}$, we find that the surface divergence operator $\bdydiv $ in Definition \ref{defn:tangential_derivative} can also be given by
$$
(\bdydiv  \v)\circ \vartheta = \widetilde{\v}_{\gamma,\gamma} + \Gamma^\beta_{\beta\gamma} \widetilde{\v}_\gamma\quad\ \Forall \v\in \rT(\bdy\Omega) \text{ \ (or equivalently, $\widetilde{\v}^i = \widetilde{\v}_\gamma \vartheta^i,_\gamma$)}\,.
$$
As a consequence,
$$
\div \w = \big[\widetilde{\w}_3,_3 + \bdydiv  (\rP_{\wtrN^\perp} \widetilde{\w}) + 2 \widetilde{\rH} \widetilde{\w}_3\big]\circ\vartheta^{-1} \qquad\text{on}\quad \bdy\Omega\,,
$$
where we recall that $\rP_{\bN^\perp}$ denotes the projection of a vector onto the tangent plane of $\bdy\Omega$. With the help of (\ref{pNdwdN_id_temp}), in local chart $(\U,\vartheta)$ we have
\begin{equation}\label{w3,3_id}
\Big[\frac{\p \w}{\p \bN} \cdot \bN \Big]\circ\vartheta = \big[\widetilde{\w}_3,_3 \wtrN^i + \widetilde{\w}_{\alpha,3} \vartheta^i,_\alpha + \widetilde{\w}_\alpha g^{\alpha\beta} \wtrN^i,_\beta\big] \wtrN^i = \widetilde{\w}_3,_3\,;
\end{equation}
thus (\ref{wid2}) is valid.

Finally, by the divergence theorem we obtain that
$$
\int_{\bdy\Omega} (\curl \w \cdot \bN) \varphi \,dS = \int_\Omega \curl \w \cdot \nabla \varphi \,dx = \int_{\bdy\Omega} (\bN \times \w) \cdot \nabla \varphi \,dS\qquad\Forall \varphi\in H^1(\Omega)\,.
$$
Since $(\nabla \varphi)^i = \varphi,_3 \bN^i + g^{\alpha\beta} \varphi,_\alpha \vartheta^i,_\beta = \dfrac{\p \varphi}{\p \bN} \bN^i + (\bdygrad \varphi)^i$ and $(\bN \times \w) \perp \bN$, we conclude that
$$
\int_{\bdy\Omega} (\curl \w \cdot \bN) \varphi \,dS = \int_{\bdy\Omega} (\bN \times \w) \cdot \bdygrad \varphi \,dS = \int_{\bdy\Omega} \bdydiv  (\w \times \bN) \varphi \,dS\qquad\Forall \varphi\in H^1(\Omega)\,;
$$
thus we verify (\ref{wid1}).
\end{proof}

With Lemma \ref{lem:tangential_component_of_dwdN}, we can now prove Theorem \ref{thm:main_thm2} with $\n=3$; we note that the case $\n=2$ follows from
the more general case by considering vector fields of the type $\u = (\u^1(x_1,x_2), \u^2(x_1,x_2),0)$.
\begin{proof}[Proof of Theorem {\rm\ref{thm:main_thm2}}]
Let $u\in H^{\rk+1}(\Omega)$, and $\curl \u =f$, $\div \u = g$, $\bdygrad \u\cdot \bN = h$. By the well-known identity
\begin{equation}\label{curlcurlu_identity}
- \Delta \u = \curl \curl \u - \nabla \div \u \qquad\text{in}\quad\Omega\,,
\end{equation}
we find that if $\chi$ is a smooth cut-off function with $\supp(\chi)\cptsubset \Omega$, then $\chi \u$ satisfies
\begin{alignat*}{2}
- \Delta (\chi \u) &= - \u\Delta \chi - 2 \nabla \chi \cdot \nabla \u + \chi (\curl\,\f - \nabla g) \qquad&&\text{in}\quad\rO \,,\\
\chi \u &= 0 \qquad&&\text{on}\quad\bdy\rO\,,
\end{alignat*}
for some smooth domain $\rO \cptsubset \Omega$ (choose $\rO$ to be some smooth domain so that $\supp(\chi) \cptsubset \rO \cptsubset \Omega$). Standard
interior elliptic estimates then show that
\begin{equation}
\|\chi \u\|_{H^{\rk+1}(\rO)} \le C \Big[\|\u\|_{H^\rk(\rO)} + \|\f\|_{H^\rk(\Omega)} + \|g\|_{H^\rk(\Omega)} \Big]\,. \label{div_curl_est_temp1}
\end{equation}
Now we proceed to the estimates near the boundary. Let $\{\zeta_m\}_{m=1}^K$ and $\{\vartheta_m\}_{m=1}^K$ be a partition of unity (subordinate to $\U_m$) and charts satisfying
\begin{enumerate}
\item $\vartheta_m: B(0,r_m) \to \U_m$ belongs to $H^{\rk+1}(B(0,r_m))$;
\item $\vartheta_m: B_+(0,r_m) \to \Omega \cap \U_m$;
\item $\vartheta_m: B(0,r_m)\cap\{y_3=0\} \to \bdy\Omega \cap \U_m$,
\end{enumerate}
and $g_m$ and $b_m$ denote the induced metric tensor and second fundamental form, respectively. Then
\begin{alignat*}{2}
- \Delta (\zeta_m \u) \hspace{-1pt}&= \zeta_m (\curl\,\f - \hspace{-1pt}\nabla g) - \u\Delta \zeta_m - 2 \nabla \zeta_m \hspace{-1pt} \cdot \hspace{-1pt} \nabla \u \qquad&&\text{in}\quad\U_m\,,\\
(\zeta_m \bdygrad \u)\cdot \bN \hspace{-1pt}&= \zeta_m h &&\text{on}\quad\U_m\,.
\end{alignat*}
In each local chart, we define the functions
\begin{align*}
& \widetilde{\u}_m = \u \circ \vartheta_m\,, \ \ \widetilde{\zeta}_m = \zeta_m \circ \vartheta_m \,, \ \ \wtrN = \bN \circ \vartheta_m \,, \\
&A = (\nabla \vartheta_m)^{-1}\,, \ \ J = \det(\nabla \vartheta_m) \,, \ \ \rg_m = \det(g_m) \,.
\end{align*}

Taking the composition of the equations above with map $\vartheta_m$, by the Piola identity (\ref{Piola_id}), we find that
\begin{subequations}\label{xi_v_eq}
\begin{alignat}{2}
- \big[J A^j_\ell A^k_\ell (\widetilde{\zeta}_m \widetilde{\u}_m),_k\hspace{-1pt}\big],_j &= J \Big[\zeta_m (\curl\,\f - \nabla g) - \u\Delta \zeta_m - 2 \nabla \zeta_m \cdot \nabla \u\Big]\circ\vartheta_m \qquad &&\text{in}\quad \rU\,,\\
(\widetilde{\zeta}_m \widetilde{\u}^i_m),_\sigma \hspace{-1pt}\wtrN^i &= \widetilde{\zeta}_m,_\sigma \hspace{-1pt}\widetilde{\u}^i_m \wtrN^i \hspace{-1pt}+\hspace{-1pt} \xi_m (h\circ\vartheta_m)\qquad &&\text{on}\quad \bdy\rU\,,
\end{alignat}
\end{subequations}
for some smooth domain $\rU$ satisfying that $\supp(\widetilde{\zeta}_m) \subseteq \cls{\rU}$ and $\supp(\xi_m) \cap \bdy\rU \subseteq \{y_3 = 0\}$.

The function $(\widetilde{\zeta}_m \widetilde{\u}_m),_\sigma$, where $\sigma = 1,\cdots,\n-1$, will be the fundamental (dependent) variable that we are going to estimate; however, in order to apply Theorem \ref{thm:vector-valued_elliptic_eq_Sobolev_coeff} we need to transform the boundary condition (\ref{xi_v_eq}b) to a homogeneous one. This is done by introducing the function ${\boldsymbol\phi}_\sigma$ which is the solution to the elliptic equation
\begin{alignat*}{2}
{\boldsymbol\phi}_\sigma - (J A^j_\ell A^k_\ell {\boldsymbol\phi}_\sigma,_k),_j &= {\bf 0} &&\text{in}\quad\rU\,,\\
{\boldsymbol\phi}_\sigma,_k A^k_\ell J A^j_\ell \nn_j &= \widetilde{\zeta}_m,_\sigma \hspace{-1pt}\widetilde{\u}^i_m J A^j_\ell \nn_j \hspace{-1pt}+\hspace{-1pt} \sqrt{\rg_m}\, \widetilde{\zeta}_m (h\circ\vartheta_m) \qquad&&\text{on}\quad\bdy\rU\,,
\end{alignat*}
in which $\nn$ is the outward-pointing unit normal to $\bdy \rU$, and then defining $\text{\bf\emph{w}}^i_\sigma = (\widetilde{\zeta}_m \widetilde{\u}^i_m),_\sigma - A^r_i{\boldsymbol\phi}_\sigma,_r$ as the new dependent variable of interest. Since $\sqrt{\rg_m}\,\wtrN = J A^\rT \nn$ on $B(0,r_m) \cap \{y_3 = 0\}$,
$$
\text{\bf\emph{w}}_\sigma \cdot \wtrN = (\widetilde{\zeta}_m \widetilde{\u}_m),_\sigma \cdot \wtrN - \frac{{\boldsymbol\phi}_\sigma,_k A^k_\ell J A^j_\ell \nn_j}{\sqrt{\rg_m}} = 0 \qquad\text{on}\quad \bdy\rU\,;
$$
thus $\text{\bf\emph{w}}_\sigma$ satisfies a homogeneous boundary condition.

Differentiating (\ref{xi_v_eq}a) with respect to $y_\sigma$, with $a^{jk}$ denoting $J A^j_\ell A^k_\ell$
we find that $\text{\bf\emph{w}}_\sigma$ satisfies
\begin{subequations}\label{wsigma_eq}
\begin{alignat}{2}
\text{\bf\emph{w}}_\sigma - \frac{\p}{\p y_j} \Big(a^{jk} \frac{\p \text{\bf\emph{w}}_\sigma}{\p y_k} \Big) &= \text{\bf\emph{F}}_\sigma \qquad&&\text{in}\quad \rU\,, \\
\text{\bf\emph{w}}_\sigma \cdot \wtrN &= 0 &&\text{on}\quad \bdy\rU\,,
\end{alignat}
\end{subequations}
where $\text{\bf\emph{F}}_\sigma$ is given by
\begin{align*}
\text{\bf\emph{F}}^i_\sigma &= \Big[J \big(\zeta_m (\curl\,\f - \nabla g) - \u\Delta \zeta_m - 2 \nabla \zeta_m \cdot \nabla \u\big)\circ\vartheta_m \Big]^i,_\sigma \\
&\quad + \text{\bf\emph{w}}^i_\sigma + \big[(J A^j_\ell A^k_\ell),_\sigma (\widetilde{\zeta}_m \widetilde{\u}^i_m),_k\big],_j + \big[J A^j_\ell A^k_\ell (A^r_i {\boldsymbol\phi}_\sigma,_r),_k\big],_j
\end{align*}
Moreover, by Lemma \ref{lem:tangential_component_of_dwdN},
\begin{align*}
& \Big[\rP_{\bN^\perp}\Big(\frac{\p (\text{\bf\emph{w}}_\sigma\circ \vartheta^{-1}_m)}{\p \bN}\Big)\Big] \circ \vartheta_m = \big[\curl (\text{\bf\emph{w}}_\sigma\circ \vartheta^{-1}_m) \times \bN\big]\circ \vartheta_m + g^{\alpha\beta}_m \vartheta_m,_\alpha (\text{\bf\emph{w}}_\sigma,_\beta \cdot \widetilde{\bN}) \\
&\qquad\quad = \big[\curl (\text{\bf\emph{w}}_\sigma\circ \vartheta^{-1}_m) \times \bN\big]\circ \vartheta_m + g^{\alpha\beta}_m \vartheta_m,_\alpha (\text{\bf\emph{w}}_\sigma \cdot \widetilde{\bN}),_\beta - g^{\alpha\beta}_m g^{\gamma\delta}_m (\text{\bf\emph{w}}_\sigma \cdot \vartheta,_\delta) b_{\gamma\beta} \, \vartheta,_\sigma
\,;
\end{align*}
thus using (\ref{wsigma_eq}b) in the second term of the right-hand side, we obtain that
\begin{equation}\label{wsigma_eq_c}
\rP_{\wtrN^\perp}\Big(a^{jk} \frac{\p \text{\bf\emph{w}}_\sigma}{\p x_k} \nn_j\Big) = \sqrt{\rg_m}\, \big[\curl (\text{\bf\emph{w}}_\sigma\circ \vartheta^{-1}_m) \times \bN\big]\circ \vartheta_m - \sqrt{\rg_m}\, g^{\alpha\beta}_m g^{\gamma\delta}_m (\text{\bf\emph{w}}_\sigma \cdot \vartheta,_\delta) b_{\gamma\beta} \, \vartheta,_\sigma
\text{ \ on \ }\bdy\rU\,.
\end{equation}
Since
\begin{align*}
& \big[\curl (\text{\bf\emph{w}}_\sigma\circ \vartheta^{-1}_m) \times \bN\big]^i\circ \vartheta_m = \varepsilon_{ijk} \varepsilon_{jrs} A^\ell_r w^s_\sigma,_\ell \wtrN^k \\
&\qquad = A^\ell_k \big[(\widetilde{\zeta}_m \widetilde{\u}^i_m),_\sigma \hspace{-1pt}-\hspace{-1pt} A^r_i{\boldsymbol\phi}_\sigma,_r\big],_\ell \wtrN^k \hspace{-1pt}-\hspace{-1pt} A^\ell_i \big[(\widetilde{\zeta}_m \widetilde{\u}^k_m),_\sigma \hspace{-1pt}-\hspace{-1pt} A^r_k{\boldsymbol\phi}_\sigma,_r\big],_\ell \wtrN^k \\
&\qquad = A^\ell_k (\widetilde{\zeta}_m \widetilde{\u}^i_m),_{\alpha\ell} \wtrN^k \hspace{-1pt}-\hspace{-1pt} A^\ell_i (\xi_m \widetilde{\u}^k_m),_{\alpha\ell} \wtrN^k \hspace{-1pt}-\hspace{-1pt} A^\ell_k (A^r_i{\boldsymbol\phi}_\sigma,_r),_\ell \wtrN^k \hspace{-1pt}+\hspace{-1pt} A^\ell_i (A^r_k{\boldsymbol\phi}_\sigma,_r),_\ell \wtrN^k
\end{align*}
and
\begin{align*}
& \big[\widetilde{\zeta}_m \big(\curl \u \times \bN\big) \circ \vartheta_m\big],_\sigma = \varepsilon_{ijk} \varepsilon_{jrs} \big(\widetilde{\zeta}_m A^\ell_r \widetilde{\u}^s_m,_\ell \wtrN^k\big),_\sigma \\
&\qquad = \big[A^\ell_k (\widetilde{\zeta}_m \widetilde{\u}^i_m),_\ell \wtrN^k \hspace{-1pt}-\hspace{-1pt} A^\ell_i (\widetilde{\zeta}_m \widetilde{\u}^k_m),_\ell \wtrN^k\big],_\sigma \hspace{-1pt}-\hspace{-1pt} \big[A^\ell_k \xi_m,_\ell \widetilde{\u}^i_m \wtrN^k \hspace{-1pt}-\hspace{-1pt} A^\ell_i \xi_m,_\ell \widetilde{\u}^k_m \wtrN^k\big],_\sigma \\
&\qquad = A^\ell_k (\widetilde{\zeta}_m \widetilde{\u}^i_m),_{\alpha\ell} \wtrN^k \hspace{-1pt}-\hspace{-1pt} A^\ell_i (\xi_m \widetilde{\u}^k_m),_{\alpha\ell} \wtrN^k \hspace{-1pt}-\hspace{-1pt} (A^\ell_k \wtrN^k),_\sigma (\xi_m \widetilde{\u}^i_m),_\ell \\
&\qquad\quad \hspace{-1pt}-\hspace{-1pt} (A^\ell_i \wtrN^k),_\sigma (\widetilde{\zeta}_m \widetilde{\u}^i_m),_\ell \hspace{-1pt}-\hspace{-1pt} \big[A^\ell_k \xi_m,_\ell \widetilde{\zeta}_m,_\ell \widetilde{\u}^i_m \wtrN^k \hspace{-1pt}-\hspace{-1pt} A^\ell_i \widetilde{\zeta}_m,_\ell \widetilde{\u}^k_m \wtrN^k\big],_\sigma \,,
\end{align*}
we find that
\begin{align*}
& \big[\curl (\text{\bf\emph{w}}_\sigma\circ \vartheta^{-1}_m) \times \bN\big]^i\circ \vartheta_m - \big[\widetilde{\zeta}_m \big(\curl \u \times \bN\big) \circ \vartheta_m\big],_\sigma \\
&\qquad = - A^\ell_k (A^r_i{\boldsymbol\phi}_\sigma,_r),_\ell \wtrN^k + A^\ell_i (A^r_k{\boldsymbol\phi}_\sigma,_r),_\ell \wtrN^k + (A^\ell_k \wtrN^k),_\sigma (\widetilde{\zeta}_m \widetilde{\u}^i_m),_\ell \\
&\qquad\quad + (A^\ell_i \wtrN^k),_\sigma (\widetilde{\zeta}_m \widetilde{\u}^i_m),_\ell + (A^\ell_k \xi_m,_\ell \widetilde{\zeta}_m,_\ell \widetilde{\u}^i_m \wtrN^k),_\sigma + (A^\ell_i \widetilde{\zeta}_m,_\ell \widetilde{\u}^k_m \wtrN^k),_\sigma \,;
\end{align*}
thus (\ref{wsigma_eq_c}) implies that
$$
\rP_{\wtrN^\perp}\Big(a^{jk} \frac{\p \text{\bf\emph{w}}_\sigma}{\p x_k} \nn_j\Big) = \sqrt{\rg_m}\, \text{\bf\emph{G}}_\sigma \qquad\text{on}\quad\bdy\rU\,, \eqno{\rm(\ref{wsigma_eq}c)}
$$
where $\text{\bf\emph{G}}_\sigma$ is given by
\begin{align*}
\text{\bf\emph{G}}_\sigma &= \big[\widetilde{\zeta}_m (f \times \bN)\circ \vartheta_m\big],_\sigma - A^\ell_k (A^r_i{\boldsymbol\phi}_\sigma,_r),_\ell \wtrN^k + A^\ell_i (A^r_k{\boldsymbol\phi}_\sigma,_r),_\ell \wtrN^k \\
&\quad + (A^\ell_k \wtrN^k),_\sigma (\widetilde{\zeta}_m \widetilde{\u}^i_m),_\ell + (A^\ell_i \wtrN^k),_\sigma (\widetilde{\zeta}_m \widetilde{\u}^i_m),_\ell + (A^\ell_k \widetilde{\zeta}_m,_\ell \widetilde{\u}^i_m \wtrN^k),_\sigma \\
&\quad + (A^\ell_i \widetilde{\zeta}_m,_\ell \widetilde{\u}^k_m \wtrN^k),_\sigma - g^{\alpha\beta}_m g^{\gamma\delta}_m \big[(\xi_m \widetilde{\u}^i_m),_\sigma \vartheta^i_m,_\delta - A^r_i{\boldsymbol\phi}_\sigma,_r \vartheta^i_m,_\delta\big] b_{\gamma\beta} \, \vartheta,_\sigma \,.
\end{align*}
As a consequence, $\text{\bf\emph{w}}_\sigma$ is the solution to equation (\ref{wsigma_eq}), and Theorem \ref{thm:vector-valued_elliptic_eq_Sobolev_coeff} (with $\ell = \rk-1$ and $\rw = n$) then implies that $(\widetilde{\zeta}_m\widetilde{\u}_m),_\sigma$ satisfies
\begin{equation}\label{wsigma_est}
\|\text{\bf\emph{w}}_\sigma\|_{H^\rk(\rU)} \le C \Big[\|\text{\bf\emph{F}}_\sigma\|_{H^{\rk-2}(\rU)} + \|\text{\bf\emph{G}}_\sigma\|_{H^{\rk-1.5}(\bdy\rU)}\Big]
\end{equation}
for some constant $C = C\big(\|a\|_{H^\rk(B(0,r_m))},\|A\|_{H^\rk(B(0,r_m))}, \|\wtrN\|_{H^{\rk-0.5}(\bdy\rU)}\big)$.

We focus on the estimate of $\text{\bf\emph{F}}_\sigma$ first. By Corollary \ref{cor:JA_est},
\begin{equation}\label{JAgm_est}
\displaystyle{}\|J\|_{H^\rk(B(0,r_m))} + \|A\|_{H^\rk(B(0,r_m))} + \|\rg_m\|_{H^{\rk-0.5}(\bdy\rU)} \le C(|\bdy\Omega|_{H^{\rk+0.5}})\,;
\end{equation}
thus Corollary \ref{cor:scalar_elliptic_eq_Sobolev_coeff} shows that
\begin{align}
\|{\boldsymbol\phi}_\sigma\|_{H^{\rk+1}(\rU)} &\le C(\|a\|_{H^\rk(\rU)},|\bdy\Omega|_{H^{k+0.5}}) \big\|\widetilde{\zeta}_m,_\sigma \hspace{-1pt}\widetilde{\u}^j_m J A^\ell_j \nn_\ell \hspace{-1pt}+\hspace{-1pt} \sqrt{\rg_m}\, \widetilde{\zeta}_m (h\circ\vartheta_m)\big\|_{H^{\rk-0.5}(\bdy\rU)} \nonumber\\
&\le C(|\bdy\Omega|_{H^{k+0.5}}) \Big[\|\u\|_{H^\rk(\Omega)} + \|h\|_{H^{k-0.5}(\bdy\Omega)}\Big] \,. \label{phi_sigma_est}
\end{align}
Moreover, by Corollary \ref{cor:f_comp_psi}, we also have that
\begin{equation}\label{an_est}
\|a\|_{H^\rk(\rU)} + \|\wtrN\|_{H^{\rk-0.5}(\bdy\rU)} \le C(|\bdy\Omega|_{H^{\rk+0.5}})\,.
\end{equation}
As a consequence,
\begin{equation}\label{Fsigma_est}
\|\text{\bf\emph{F}}_\sigma\|_{H^{\rk-2}(\rU)} \le\hspace{-1pt} C\hspace{-1pt}(|\bdy\Omega|_{H^{\rk+0.5}}) \Big[\hspace{-1pt}\|\f\|_{H^\rk(\Omega)} \hspace{-1.5pt}+\hspace{-1.5pt} \|g\|_{H^\rk(\Omega)} \hspace{-1.5pt}+\hspace{-1.5pt} \|h\|_{H^{\rk-0.5}(\bdy\Omega)} + \|\u\|_{H^\rk(\Omega)} \hspace{-1pt}\Big].
\end{equation}

As for the estimate of $\text{\bf\emph{G}}_\sigma$, the highest order terms are $(A^\ell_k \wtrN^k),_\sigma (\widetilde{\zeta}_m \widetilde{\u}^i_m),_\ell$, $(A^\ell_i \wtrN^k),_\sigma (\widetilde{\zeta}_m \widetilde{\u}^i_m),_\ell$ and $g^{\alpha\beta}_m g^{\gamma\delta}_m (\xi_m \widetilde{\u}^i_m),_\sigma \vartheta^i_m,_\delta b_{\gamma\beta} \, \vartheta,_\sigma$, and we apply (\ref{product_Hr_est}) to obtain, for example, that
\begin{align*}
& \big\|(A^\ell_k \wtrN^k),_\sigma (\widetilde{\zeta}_m \widetilde{\u}^i_m),_\ell\|_{H^{\rk-1.5}(\bdy\rU)} \le C \big\|(A^\ell_k \wtrN^k),_\sigma (\widetilde{\zeta}_m \widetilde{\u}^i_m),_\ell\big\|_{H^{\rk-1}(\rU)} \\
&\qquad\quad \le C \|\p (A \wtrN)\|_{H^{k-1}(\rU)} \|\nabla (\widetilde{\zeta}_m \widetilde{\u}_m)\|_{H^s(\rU)} \le C(|\bdy|\Omega|_{2.5}) \|\u\|_{H^{s+1}(\Omega)}\,,
\end{align*}
where $s = \max\big\{\rk-1,\smallexp{$\displaystyle{}\frac{\rk}{2} + \frac{\n}{4}$}\big\}$ is chosen so that (\ref{product_Hr_est}) can be applied (since $s>\novertwo$). Therefore,
\begin{equation}\label{Gsigma_est}
\|\text{\bf\emph{G}}_\sigma\|_{H^{\rk-1.5}(\bdy\rU)} \le C(|\bdy\Omega|_{H^{\rk+0.5}}) \Big[\|\f\|_{H^{\rk-1}(\Omega)} + \|h\|_{H^{\rk-0.5}(\bdy\Omega)} + \|\u\|_{H^{s+1}(\Omega)}\Big].
\end{equation}
Combining estimates (\ref{wsigma_est}), (\ref{JAgm_est}), (\ref{phi_sigma_est}), (\ref{an_est}), (\ref{Fsigma_est}) and (\ref{Gsigma_est}), we find that
\begin{align}
\|(\widetilde{\zeta}_m \widetilde{\u}_m),_\sigma\|_{H^\rk(\rU)} &\le \|\text{\bf\emph{w}}_\sigma\|_{H^\rk(\rU)} + \|A^\rT \nabla {\boldsymbol\phi}_\sigma\|_{H^\rk(\rU)} \nonumber\\
&\le C(|\bdy\Omega|_{H^{\rk+0.5}}) \Big[\|\f\|_{H^\rk(\Omega)} + \|g\|_{H^\rk(\Omega)} + \|h\|_{H^{\rk-0.5}(\bdy\Omega)} + \|\u\|_{H^{s+1}(\Omega)}\Big]. \label{xi_v_est_temp}
\end{align}

Finally, following the same procedure of Step 4 in the proof of Theorem \ref{thm:vector-valued_elliptic_regularity} (that is, using (\ref{xi_v_eq}a) to obtain an expression of $\widetilde{\zeta}_m\p^{\rk+1-j} \nabla^j \widetilde{\u}_m,_{33}$) and then arguing by induction on $j$, we find that
$$
\|\widetilde{\zeta}_m \widetilde{\u}_m\|_{H^{\rk+1}(\rU)} \le C(|\bdy\Omega|_{H^{\rk+0.5}}) \Big[\|\f\|_{H^\rk(\Omega)} + \|g\|_{H^\rk(\Omega)} + \|h\|_{H^{\rk-0.5}(\bdy\Omega)} + \|\u\|_{H^{s+1}(\Omega)}\Big].
$$
The estimate above and estimate (\ref{div_curl_est_temp1}) provide us with
$$
\|\u\|_{H^{\rk+1}(\rU)} \hspace{-1pt}\le\hspace{-1pt} C(|\bdy\Omega|_{H^{\rk+0.5}}) \Big[\hspace{-1pt}\|\f\|_{H^\rk(\Omega)} \hspace{-1pt}+\hspace{-1pt} \|g\|_{H^\rk(\Omega)} \hspace{-1pt}+\hspace{-1pt} \|h\|_{H^{\rk-0.5}(\bdy\Omega)} \hspace{-1pt}+\hspace{-1pt} \|\u\|_{H^{s+1}(\Omega)}\hspace{-1pt}\Big].
$$
Since $0 < s+1 < \rk+1$, by interpolation and Young's inequality,
$$
\|\u\|_{H^s(\Omega)} \le C_\delta \|\u\|_{L^2(\Omega)} + \delta \|\u\|_{H^{\rk+1}(\Omega)}\qquad\Forall \delta>0\,,
$$
so by choosing $\delta > 0$ small enough we conclude (\ref{Hodge_elliptic_estimate1}).
\end{proof}
By studying the vector-valued elliptic equation (\ref{vector-valued_elliptic_eq}),
with the help of Theorem \ref{thm:vector-valued_elliptic_eq_Sobolev_coeff_v2} we can also conclude (\ref{Hodge_elliptic_estimate2}).

\begin{remark}
Suppose that $\Omega$ is a bounded $H^{\rk+2}$-domain for some $\rk > \novertwo$\,. Since $\bdygrad \u \cdot \bN = \bdygrad (\u\cdot \bN) - \bdygrad \bN \cdot \u$, by interpolation we find that
\begin{align*}
\|\bdygrad \u \cdot \bN\|_{H^{\rk-0.5}(\bdy\Omega)} &\le\hspace{-1pt} \|\u\cdot \bN\|_{H^{\rk+0.5}(\bdy\Omega)} \hspace{-1pt}+\hspace{-1pt} \|\bdygrad\bN \cdot \u\|_{H^{\rk-0.5}(\bdy\Omega)} \\
&\le\hspace{-1pt} \|\u\cdot \bN\|_{H^{\rk+0.5}(\bdy\Omega)} \hspace{-1pt}+\hspace{-1pt} C(|\bdy\Omega|_{H^{\rk+1.5}}) \|\u\|_{H^\rk(\Omega)} \\
&\le\hspace{-1pt} \|\u\cdot \bN\|_{H^{\rk+0.5}(\bdy\Omega)} \hspace{-1pt}+\hspace{-1pt} C(|\bdy\Omega|_{H^{\rk+1.5}}, \delta) \|\u\|_{L^2(\Omega)} \hspace{-1pt}+\hspace{-1pt} \delta \|\u\|_{H^{\rk+1}(\Omega)}.
\end{align*}
Hence, by choosing $\delta>0$ small enough we conclude that there exists a generic constant $C = C(|\bdy\Omega|_{H^{\rk+1.5}})$ such that
$$
\|\u\|_{H^{\rk+1}(\Omega)} \le C \Big[\|\u\|_{L^2(\Omega)} + \|\curl \u\|_{H^\rk(\Omega)} + \|\div \u\|_{H^\rk(\Omega)} + \|\u\cdot \bN\|_{H^{\rk+0.5}(\bdy\Omega)}\Big].
$$
Similarly, we also have that
$$
\|\u\|_{H^{\rk+1}(\Omega)} \le C \Big[\|\u\|_{L^2(\Omega)} + \|\curl \u\|_{H^\rk(\Omega)} + \|\div \u\|_{H^\rk(\Omega)} + \|\u\times \bN\|_{H^{\rk+0.5}(\bdy\Omega)}\Big]
$$
for some constant $C = C(|\bdy\Omega|_{H^{\rk+1.5}})$.
\end{remark}

\section{The Proof of Theorem \ref{thm:main_thm1}}\label{sec:vector_decomp}
We begin with the following problem: find a vector field $\v$ such that
\begin{subequations}\label{ufromcurldiv_temp}
\begin{alignat}{2}
\curl \v &= \f \qquad&&\text{in}\quad \Omega\,,\\
\div \v &= g &&\text{in}\quad \Omega\,,\\
\v\cdot \bN &= h &&\text{on}\quad \bdy\Omega\,.
\end{alignat}
\end{subequations}
From the divergence theorem and the fact that $ \operatorname{div} \operatorname{curl} =0$ , we must require that
\begin{equation}\label{solvability_condition_div_curl_normal_trace}
\div \f = 0 \quad\text{and}\quad \int_\Omega g\, dx = \int_{\bdy\Omega} h\, dS\,.
\end{equation}

Since $g$ and $h$ satisfy the solvability condition (\ref{solvability_condition_div_curl_normal_trace}), there exists a solution $\phi$ to the Poisson equation
with Neumann boundary conditions:
\begin{subequations}\label{phi_eq}
\begin{alignat}{2}
\Delta \phi &= g \qquad&&\text{in}\quad \Omega\,,\\
\frac{\p \phi}{\p \bN} &= h &&\text{on}\quad \bdy\Omega\,.
\end{alignat}
\end{subequations}
Let $\u = \v - \nabla \phi$. Then $\u$ satisfies
\begin{subequations}\label{ufromcurldiv}
\begin{alignat}{2}
\curl \u &= \f \qquad&&\text{in}\quad \Omega\,,\\
\div \u &= 0 &&\text{in}\quad \Omega\,,\\
\u \cdot \bN &= 0 &&\text{on}\quad \bdy\Omega\,.
\end{alignat}
\end{subequations}
Hence, if (\ref{ufromcurldiv}) is solvable, then there exists a solution to (\ref{ufromcurldiv_temp}).

\subsection{Uniqueness of the solution}\label{sec:uniqueness_curl_div_normal_trace}
We show that under the assumptions of Theorem \ref{thm:main_thm1}, the solution to (\ref{ufromcurldiv_temp}) is unique.
We first assume that $\Omega$ is a bounded convex domain.  (Note that a convex set must be simply connected.) If $\Varphi \in \mC^2(\Omega)\cap \mC^1(\cls{\Omega})$, then for all $\u\in H^1(\Omega)$,
\begin{align}
\int_\Omega \curl \u \cdot \curl \Varphi\,dx &= \int_\Omega \u \cdot \curl\curl \Varphi\, dx + \int_{\bdy\Omega} (\bN \times \u) \cdot \curl \Varphi\, dS \nonumber\\
&= \int_\Omega \hspace{-2pt} \u \cdot (-\Delta \Varphi + \nabla \div \Varphi) \,dx + \int_{\bdy\Omega} (\bN \times \u) \cdot \curl \Varphi\, dS \nonumber\\
&= \int_\Omega \hspace{-2pt} \u \cdot (-\Delta \Varphi + \nabla \div \Varphi) \,dx + \int_{\bdy\Omega} \Big[\frac{\p \Varphi}{\p \bN} \cdot \u - \u^k \bN_j \Varphi^j,_k \Big] dS \nonumber\\
&= \int_\Omega (\nabla \u \hspace{-1pt}:\hspace{-1pt} \nabla \Varphi - \div \u \div \Varphi)\,dx + \int_{\bdy\Omega} \Big[(\u \cdot \bN) \div\Varphi - \u^k \bN_j \Varphi^j,_k \Big] dS \nonumber\,.
\end{align}
Using the notation introduced in the proof of Lemma \ref{lem:tangential_component_of_dwdN}, in any local chart $(\U,\vartheta)$ we have on $\bdy\Omega$,
\begin{align*}
(\u^k \bN_j \Varphi^j,_k)\circ \vartheta &= \widetilde{\u}^k \wtrN^j \big(\widetilde{\Varphi}^j,_\n \wtrN^k + g^{\alpha\beta} \widetilde{\Varphi}^j,_\alpha \vartheta^k,_\beta\big) \\
&= (\widetilde{\u} \cdot \wtrN) (\widetilde{\Varphi}^j,_\n \wtrN^j) + g^{\alpha\beta} (\widetilde{\u} \cdot \vartheta,_\beta) (\widetilde{\Varphi}\cdot\wtrN),_\alpha - g^{\alpha\beta} (\widetilde{\u} \cdot \vartheta,_\beta) \wtrN,_\alpha \cdot \widetilde{\Varphi} \\
&= (\widetilde{\u} \cdot \wtrN) (\widetilde{\Varphi}^j,_\n \wtrN^j) + g^{\alpha\beta} (\widetilde{\u} \cdot \vartheta,_\beta) (\widetilde{\Varphi}\cdot\wtrN),_\alpha - g^{\alpha\beta} g^{\gamma\delta} b_{\alpha\gamma} (\widetilde{\u} \cdot \vartheta,_\beta) (\widetilde{\Varphi} \cdot \vartheta,_\delta) \,,
\end{align*}
so that using (\ref{wid2}),
\begin{align}
\int_\Omega \curl \u \cdot \curl \Varphi\,dx &= \int_\Omega (\nabla \u : \nabla \Varphi - \div \u \div \Varphi)\,dx \nonumber\\
&\quad + \int_{\bdy\Omega} (\u \cdot \bN) \Big[\bdydiv  (\rP_{\bN^\perp} \Varphi) + 2 \rH (\Varphi \cdot \bN) \Big] dS \label{int_curlcurl_id}\\
&\quad + \sum_{m=1}^K \int_{\bdy\Omega\cap \U_m} \hspace{-2pt}\zeta_m \big[g^{\alpha\beta}_m g^{\gamma\delta}_m b_{m \alpha\gamma} \big((\u\circ \vartheta_m)\cdot \vartheta_m,_\beta)\big) \big((\Varphi \circ \vartheta_m) \cdot \vartheta_m,_\delta)\big]\circ \vartheta_m^{-1} dS \nonumber\\
&\quad - \sum_{m=1}^K \int_{\bdy\Omega\cap \U_m} \hspace{-2pt}\zeta_m \big[g^{\alpha\beta}_m \big((\u\circ\vartheta_m) \cdot \vartheta_m,_\beta\big) \big((\Varphi\cdot \bN)\circ \vartheta_m \big),_\alpha\,\big]\circ\vartheta_m^{-1} dS\,. \nonumber
\end{align}
Therefore, if $\v_1, \v_2\in H^1(\Omega)$ are two solutions to (\ref{ufromcurldiv_temp}), then $\v=\v_1 - \v_2$ satisfies
\begin{align*}
& \|\curl \v\|^2_{L^2(\Omega)} + \|\div \v\|^2_{L^2(\Omega)} \\
&\qquad = \|\nabla \v\|^2_{L^2(\Omega)} + \sum_{m=1}^K \int_{\bdy\Omega\cap \U_m} \hspace{-2pt}\zeta_m \big[g^{\alpha\beta}_m g^{\gamma\delta}_m b_{m \alpha\gamma} \big((\v \circ \vartheta_m)\cdot \vartheta_m,_\beta)\big) \big((\v \circ \vartheta_m) \cdot \vartheta_m,_\delta)\big]\circ \vartheta_m^{-1} dS \,.
\end{align*}
Since $\Omega$ is convex, $g^{\alpha\beta}_m g^{\gamma\delta}_m b_{m \alpha\gamma}$ is non-negative definite for all $m$; thus the \Poincare\ inequality (\ref{vectorPoincareineq}) shows that for some constant $c>0$,
$$
c \|\v\|^2_{H^1(\Omega)} \le \|\nabla \v\|^2_{L^2(\Omega)} \le \|\curl \v\|^2_{L^2(\Omega)} + \|\div \v\|^2_{L^2(\Omega)} = 0
$$
which implies that $\v = 0$. In other words, the $H^1$-solution to (\ref{ufromcurldiv_temp}) must be unique if $\Omega$ is bounded and convex.

Now we assume the more general case that $\Omega$ is a simply connected open set, and that there are two solutions $\v_1$ and $\v_2$ in $H^1(\Omega)$. Then $\v=\v_1-\v_2$ satisfies $\curl \v = 0$ in $\Omega$. By the simple connectedness of $\Omega$, $\v$ must be of the form $\v = \nabla p$ for some scalar potential $p$. Then the equation
\begin{subequations}\label{homogeneous_curl_div_eq1}
\begin{alignat}{2}
\curl \v &= {\bf 0} \qquad&&\text{in}\quad\Omega\,,\\
\div \v &= 0 &&\text{in}\quad\Omega\,,\\
\v \cdot \bN &= 0 &&\text{on}\quad \bdy\Omega\,,
\end{alignat}
\end{subequations}
has only the trivial solution $\v = {\bf 0}$.  In other words, if $\Omega$ is the disjoint union of simply connected open sets, then equation (\ref{div_curl}) with boundary condition (\ref{bc1}) has a unique solution.

\subsection{Existence of solutions}\label{sec:existence_curl_div_normal_trace}
We solve (\ref{ufromcurldiv}) by finding a solution $\u$ of the form $\u = \curl \w$ for a divergence-free vector field $\w$. Indeed, if such a $\w$ exists, then using (\ref{curlcurlu_identity}), $\w$ must solve
\begin{subequations}\label{w_eq_temp}
\begin{alignat}{2}
-\Delta \w &= \f \qquad&&\text{in}\quad \Omega\,,\\
\div \w &= 0 &&\text{in}\quad\Omega\,,\\
\curl \w\cdot \bN &= 0 &&\text{on}\quad \bdy\Omega\,.
\end{alignat}
\end{subequations}
We note that if $\w$ is sufficiently smooth, then the divergence-free condition (\ref{w_eq_temp}b) can instead be treated as a boundary condition
$$
\div \w = 0 \qquad\text{on}\quad \bdy\Omega\,. \eqno{\rm\text{(\ref{w_eq_temp}b')}}
$$
In fact, taking the divergence of (\ref{w_eq_temp}a) we find that
$$
\Delta \div \w = \div \f = 0 \qquad\text{in}\quad \Omega\,,
$$
where we use the solvability condition (\ref{solvability_condition_div_curl_normal_trace}) to establish the last equality; thus if $\w$ satisfies (\ref{w_eq_temp}a,b'), $\w$ automatically has zero divergence. In other words, we may instead assume that $\w$ satisfies (\ref{w_eq_temp}a,b',c). Our goal next is to find some suitable boundary condition to replace (\ref{w_eq_temp}b',c).

\subsubsection{The case that $\Omega = B(0,R)$}
Now we assume that $\Omega = B(0,R)$ for some $R>0$. Having obtained (\ref{wid2}) and (\ref{wid1}), in order to achieve (\ref{w_eq_temp}b',c) it is natural to consider the case $\rP_{\bN^\perp} \w = 0$. In other words, we consider the following elliptic problem (with a non-standard boundary condition)
\begin{subequations}\label{HodgeEQ1}
\begin{alignat}{2}
- \Delta \w &= \f \qquad&&\text{in}\quad\Omega\,,\\
\rP_{\bN^\perp} \w &= {\bf 0} &&\text{on}\quad\bdy\Omega\,,\\
\frac{\p \w}{\p \bN} \cdot \bN + 2\, \rH (\w\cdot \bN) &= 0 &&\text{on}\quad\bdy\Omega\,,
\end{alignat}
\end{subequations}
where we remark that $\rH = R^{-1}$ is a positive constant. We also note that (\ref{w_eq_temp}b') and (\ref{w_eq_temp}c) are direct consequence of (\ref{HodgeEQ1}b,c), and (\ref{w3,3_id}) shows that (\ref{HodgeEQ1}c) is in fact a Robin boundary condition for $\widetilde{\w}_3$. The goal is to find a solution to (\ref{HodgeEQ1}) in the Hilbert space
\begin{align*}
H^1_\tau(\Omega) \equiv \big\{ \w\in H^1(\Omega)\,\big|\, \rP_{\bN^\perp} \w = 0 \text{ on }\bdy\Omega\big\} = \big\{ \w\in H^1(\Omega)\,\big|\, \w \times \bN = 0 \text{ on }\bdy\Omega\big\}\,.
\end{align*}

In order to solve (\ref{HodgeEQ1}), we find the weak formulation first, and this amounts to computing
 $\smallexp{$\displaystyle{}\int_{\bdy\Omega} \frac{\p \w}{\p \bN}$} \cdot \Varphi\, dS$. If $\Varphi\in H^1_\tau(\Omega)$,
 then $\Varphi = (\Varphi\cdot \bN)\bN$; thus, if $\w$ satisfies (\ref{HodgeEQ1}c), then for all $\Varphi \in H^1_\tau(\Omega)$,
\begin{align}
- \int_{\bdy\Omega} \frac{\p \w}{\p \bN} \cdot \Varphi\, dS = - \int_{\bdy\Omega} \Big[\frac{\p \w}{\p \bN} \cdot \bN \Big] (\Varphi \cdot \bN)\, dS = 2 \int_{\bdy\Omega} \rH (\w\cdot \bN) (\Varphi\cdot \bN)\, dS\,. \label{ss1}
\end{align}
Using (\ref{ss1}), we can state the following
\begin{definition}
A vector-valued function $\w\in H^1_\tau(\Omega)$ is said to be a weak solution of {\rm(\ref{HodgeEQ1})} if
\begin{equation}\label{weakformHodgeEQ1}
\int_\Omega \hspace{-1.5pt}\nabla \w \hspace{-1pt}:\hspace{-1pt} \nabla \Varphi\, dx \hspace{-1pt}+\hspace{-1pt} 2 \int_{\bdy\Omega} \hspace{-1.5pt}\rH (\w\hspace{-1pt}\cdot\hspace{-1pt} \bN) (\Varphi \hspace{-1pt}\cdot\hspace{-1pt} \bN)\, dS \hspace{-1pt}=\hspace{-1pt} (\f,\Varphi)_{L^2(\Omega)} \quad\Forall\Varphi\in \hspace{-1pt} H^1_\tau(\Omega)\,,
\end{equation}
where $\nabla \w : \nabla \Varphi = \w^i,_j \Varphi^{\,i},_j$.
\end{definition}
Since $\rH > 0$, the left-hand side of (\ref{weakformHodgeEQ1}) obviously defines a bounded, coercive bilinear form on $H^1_\tau(\Omega)\times H^1_\tau(\Omega)$. In fact, using \Poincare's inequality (\ref{vectorPoincareineq}) we find that for some generic constant $c>0$,
$$
\int_\Omega \nabla \w:\nabla \w dx + 2 \int_{\bdy\Omega} \rH (\w\cdot \bN) (\w\cdot \bN)\,dS \ge c \|\w\|^2_{H^1(\Omega)} \qquad \w \in H^1_\tau(\Omega)\,;
$$
hence by the Lax-Milgram theorem, there exists a unique $\w\in H^1_\tau(\Omega)$ satisfying the weak formulation (\ref{weakformHodgeEQ1}) and the basic energy estimate
\begin{equation}\label{HodgeEQ_basic_estimate}
\|\w\|_{H^1(\Omega)} \le C \|\f\|_{L^2(\Omega)}\,.
\end{equation}

Before proceeding, we establish the corresponding regularity theory for equation (\ref{HodgeEQ1}).
\begin{lemma}\label{lem:curl_div_normal_trace_convex}
Let $\Omega = B(0,R) \subseteq \bbR^3$ for some $R>0$. Then for all $\f\in H^{\ell-1}(\Omega)$ for some $\ell \ge 1$, the weak solution $\w$ to {\rm(\ref{HodgeEQ1})} in fact belongs to $H^{\ell+1}(\Omega)$, and satisfies
\begin{equation}\label{estimate_for_given_vorticity}
\|\w\|_{H^{\ell+1}(\Omega)} \le C \|\f\|_{H^{\ell-1}(\Omega)}\,.
\end{equation}
\end{lemma}
\begin{proof}
As the proof of Theorem \ref{thm:vector-valued_elliptic_regularity} we prove this lemma by induction. The weak solution $\w$ indeed belongs to $H^1(\Omega)$ satisfies (\ref{HodgeEQ_basic_estimate}). Assume that $\w\in H^j(\Omega)$ for some $j \le \ell$. If $\chi$ is a smooth cut-off function so that $\supp(\chi) \cptsubset \Omega$, the same computation as in the proof of Theorem \ref{thm:vector-valued_elliptic_regularity} (with $a^{jk} = \delta^{jk}$) shows that
\begin{equation}\label{vector-valued_interior_estimate1}
\|\chi \nabla^{j + 1} \w\|_{L^2(\Omega)} \le C \Big[ \|\f\|_{H^{j-1}(\Omega)} + \|\w\|_{H^j(\Omega)}\Big]\,,
\end{equation}
where the constant $C$ depends on the distance between the support of $\chi$ and $\bdy\Omega$.

Now we focus on the estimate of $\w$ near $\bdy\Omega$. Let $\{\zeta_m,\U_m,\vartheta_m\}_{m=1}^K$ be defined as in the proof of Theorem \ref{thm:main_thm2}, and
$g_{\alpha\beta} = \vartheta_m,_\alpha\cdot\, \vartheta_m,_\beta$. Define
$$
\Varphi_1 = (-1)^j \zeta_m \bN \big[\Lambda_\epsilon \bp^{2j} \Lambda_\epsilon(\widetilde{\zeta}_m \widetilde{\w}_m\hspace{-1pt}\cdot\wtrN)\big] \circ \vartheta^{-1}_m\,,
$$
where $\widetilde{\zeta}_m = \zeta_m \circ \vartheta_m$, $\widetilde{\w}_m = \w \circ \vartheta$ and $\wtrN = \bN \circ \vartheta$. Since $\rP_{\bN^\perp} \Varphi_1 = 0$, $\Varphi_1$ can be used as a test function in (\ref{weakformHodgeEQ1}). Similar to the computations in Step 2 in the proof of Theorem \ref{thm:vector-valued_elliptic_regularity}, we find that
\begin{align}
\int_{\Omega} \hspace{-1pt} \nabla \w \hspace{-1pt}:\hspace{-1pt} \nabla \Varphi_1\, dx \hspace{-1pt}&\ge\hspace{-1pt} \frac{1}{2} \big\|\nabla \bp^j \Lambda_\epsilon (\widetilde{\zeta}_m \widetilde{\w}_m\hspace{-2pt}\cdot\wtrN) \big\|^2_{L^2(B^+_m)} \hspace{-2pt}-\hspace{-1pt} C \|\u\|_{H^j(\Omega)} \Big[\big\|\bp^j \Lambda_\epsilon (\widetilde{\zeta}_m \widetilde{\w}_m\hspace{-2pt}\cdot\wtrN) \big\|_{H^1(B^+_m)} \hspace{-2pt}+\hspace{-1pt} \|\u\|_{H^j(\Omega)} \hspace{-1pt}\Big] \nonumber\\
&\ge \frac{1}{2} \big\|\nabla \bp^j \Lambda_\epsilon (\widetilde{\zeta}_m \widetilde{\w}_m\hspace{-2pt}\cdot\wtrN) \big\|^2_{L^2(B^+_m)} \hspace{-1pt}-\hspace{-1pt} C_\delta \|\w\|^2_{H^j(\Omega)} \hspace{-1pt}-\hspace{-1pt} \delta \big\|\bp^j \Lambda_\epsilon (\widetilde{\zeta}_m \widetilde{\w}_m\hspace{-2pt}\cdot\wtrN) \big\|_{H^1(B^+_m)} \,, \label{pf_thm_12.21_est1}
\end{align}

Now we focus on the term $\smallexp{$\displaystyle{}\int_{\bdy\Omega}$} \rH (\w\cdot \bN) (\Varphi_1\cdot \bN)\, dS$. Making a change of variable and integrating by parts, we find that
\begin{align*}
\int_{\bdy\Omega} \rH (\w\cdot \bN)(\Varphi_1\cdot \bN)\, dS &= \int_{\{y_3= 0\}} \hspace{-2pt}\bp^j \Lambda_\epsilon \big[\sqrt{\rg}\,\rH (\widetilde{\zeta}_m \widetilde{\w}_m \hspace{-1pt}\cdot\hspace{-1pt} \wtrN)\big] \bp^j \big( \Lambda_\epsilon (\widetilde{\zeta}_m \widetilde{\w}_m \hspace{-1pt}\cdot\hspace{-1pt} \wtrN)\big)\,dS \nonumber\\
&= \int_{\{y_3= 0\}} \hspace{-2pt}\bp \Lambda_\epsilon \big[\sqrt{\rg}\,\rH \bp^{j-1} (\widetilde{\zeta}_m \widetilde{\w}_m \hspace{-1pt}\cdot\hspace{-1pt} \wtrN)\big] \bp^j \big( \Lambda_\epsilon (\widetilde{\zeta}_m \widetilde{\w}_m \hspace{-1pt}\cdot\hspace{-1pt} \wtrN)\big)\,dS 
\\
&\quad + \hspace{-2pt} \sum_{k=1}^{j-1} \hspace{-2pt}\footnoteexp{$\displaystyle{}{{j\hspace{-1pt}-\hspace{-1pt}1}\choose{k}}$}\hspace{-3pt} \int_{\{y_3= 0\}} \hspace{-4pt}\bp \Lambda_\epsilon \big[\hspace{-1pt}\bp^k \hspace{-1pt}(\hspace{-1pt}\sqrt{\rg}\,\rH ) \bp^{j-1-k} (\widetilde{\zeta}_m \widetilde{\w}_m \hspace{-2pt}\cdot\hspace{-2pt} \wtrN)\hspace{-1pt}\big] \bp^j \big( \Lambda_\epsilon (\widetilde{\zeta}_m \widetilde{\w}_m \hspace{-2pt}\cdot\hspace{-2pt} \wtrN)\hspace{-1pt}\big)\,dS\,. \nonumber
\end{align*}
Using the commutator notation $\comm{A}{B} f = A (Bf) - B(Af)$,
\begin{align}
\int_{\bdy\Omega} \rH (\w\cdot \bN)&(\Varphi_1\cdot \bN)\, dS = \int_{\{y_3= 0\}} \hspace{-2pt} \sqrt{\rg}\,\rH \big|\bp^j \Lambda_\epsilon (\widetilde{\zeta}_m \widetilde{\w}_m \hspace{-1pt}\cdot\hspace{-1pt} \wtrN)\big|^2 dS \nonumber\\
& + \int_{\{y_3= 0\}} \hspace{-2pt} \big[\bp (\sqrt{\rg}\, \rH ) \bp^{j-1} \Lambda_\epsilon(\widetilde{\zeta}_m \widetilde{\w}_m \hspace{-1pt}\cdot\hspace{-1pt} \wtrN\big)\big] \bp^j \big( \Lambda_\epsilon (\widetilde{\zeta}_m \widetilde{\w}_m \hspace{-1pt}\cdot\hspace{-1pt} \wtrN)\big)\,dS \nonumber\\
& + \int_{\{y_3= 0\}} \hspace{-2pt}\bp \big[\bigcomm{\Lambda_\epsilon}{\sqrt{\rg}\,\rH } \bp^{j-1} (\widetilde{\zeta}_m \widetilde{\w}_m \hspace{-1pt}\cdot\hspace{-1pt} \wtrN)\big] \bp^j \Lambda_\epsilon (\widetilde{\zeta}_m \widetilde{\w}_m \hspace{-1pt}\cdot\hspace{-1pt} \wtrN)\,dS \label{pf_thm_12.21_id1}\\
& + \sum_{k=1}^{j-1} \footnoteexp{$\displaystyle{}{{j\hspace{-1pt}-\hspace{-1pt}1}\choose{k}}$} \hspace{-2pt}\int_{\{y_3= 0\}} \hspace{-4pt} \Lambda_\epsilon \big[\bp^k (\hspace{-1pt}\sqrt{\rg}\,\rH ) \bp^{j-k} (\widetilde{\zeta}_m \widetilde{\w}_m \hspace{-1.5pt}\cdot\hspace{-1pt} \wtrN)\big] \bp^j \big( \Lambda_\epsilon (\widetilde{\zeta}_m \widetilde{\w}_m \hspace{-1.5pt}\cdot\hspace{-1pt} \wtrN)\hspace{-1pt}\big)\,dS \nonumber\\
& + \sum_{k=1}^{j-1} \footnoteexp{$\displaystyle{}{{j\hspace{-1pt}-\hspace{-1pt}1}\choose{k}}$} \hspace{-2pt}\int_{\{y_3= 0\}} \hspace{-4pt} \Lambda_\epsilon \big[\bp^{k+1} (\hspace{-1pt}\sqrt{\rg}\,\rH ) \bp^{j-1-k} (\widetilde{\zeta}_m \widetilde{\w}_m \hspace{-1.5pt}\cdot\hspace{-1pt} \wtrN)\big] \bp^j \big( \Lambda_\epsilon (\widetilde{\zeta}_m \widetilde{\w}_m \hspace{-1.5pt}\cdot\hspace{-1pt} \wtrN)\hspace{-1pt}\big)\,dS\,.\nonumber
\end{align}
The commutator estimate (\ref{commutator_estimates2}) and interpolation, as well as Young's inequality, show that
\begin{align*}
&\int_{\{y_3= 0\}} \hspace{-2pt}\bp \big[\bigcomm{\Lambda_\epsilon}{\sqrt{\rg}\,\rH } \bp^{j-1} (\widetilde{\zeta}_m \widetilde{\w}_m \hspace{-1pt}\cdot\hspace{-1pt} \wtrN)\big] \bp^j \Lambda_\epsilon (\widetilde{\zeta}_m \widetilde{\w}_m \hspace{-1pt}\cdot\hspace{-1pt} \wtrN)\,dS \\
&\qquad\ge - C \big\|\bp^{j-1} \Lambda_\epsilon (\widetilde{\zeta}_m \widetilde{\w}_m \hspace{-1pt}\cdot\hspace{-1pt} \wtrN)\big\|_{L^2(\{y_3 = 0\})} \big\|\bp^j \Lambda_\epsilon (\widetilde{\zeta}_m \widetilde{\w}_m \hspace{-1pt}\cdot\hspace{-1pt} \wtrN)\big\|_{L^2(\{y_3 = 0\})} \\
&\qquad\ge - C_\delta \|\w\|^2_{H^j(\Omega)} - \delta \big\|\bp^j \Lambda_\epsilon (\widetilde{\zeta}_m \widetilde{\w}_m \hspace{-1pt}\cdot\hspace{-1pt} \wtrN)\big\|^2_{H^1(B(0,r_m))}\,.
\end{align*}
Using \Holder's inequality to estimate the other terms we obtain that
\begin{equation}\label{pf_thm_12.21_est2}
\begin{array}{l}
\displaystyle{} \int_{\bdy\Omega} \rH (\w\cdot \bN)(\Varphi_1\cdot \bN)\, dS \ge \int_{\{y_3= 0\}} \hspace{-2pt} \sqrt{\rg}\,\rH \big|\bp^j \Lambda_\epsilon (\widetilde{\zeta}_m \widetilde{\w}_m \hspace{-1pt}\cdot\hspace{-1pt} \wtrN)\big|^2 dS \vspace{.2cm}\\
\displaystyle{} \qquad\qquad - C_\delta \|\w\|^2_{H^j(\Omega)} - \delta \big\|\bp^j \Lambda_\epsilon (\widetilde{\zeta}_m \widetilde{\w}_m \hspace{-1pt}\cdot\hspace{-1pt} \wtrN)\big\|^2_{H^1(B(0,r_m))}.
\end{array}
\end{equation}
Moreover, it is easy to see that
\begin{align}
\int_{\Omega} \f\cdot \Varphi_1\, dx &\le C \|\f\|_{H^j(\Omega)} \big\|\bp^{j+1} \Lambda_\epsilon (\widetilde{\zeta}_m \widetilde{\w}_m \hspace{-1pt}\cdot\hspace{-1pt} \wtrN)\big\|_{L^2(B(0,r_m))} \nonumber\\
&\le C_\delta \|\f\|^2_{L^2(\Omega)} + \delta \big\|\bp^{j+1} \Lambda_\epsilon (\widetilde{\zeta}_m \widetilde{\w}_m \hspace{-1pt}\cdot\hspace{-1pt} \wtrN)\big\|^2_{L^2(B(0,r_m))} \,. \label{pf_thm_12.21_est3}
\end{align}
Combining (\ref{pf_thm_12.21_est1}), (\ref{pf_thm_12.21_est2}) and (\ref{pf_thm_12.21_est3}),
\begin{equation}
\begin{array}{l}
\displaystyle{} \big\|\nabla \bp^j \Lambda_\epsilon (\widetilde{\zeta}_m \widetilde{\w}_m\hspace{-1pt}\cdot\wtrN) \big\|^2_{L^2(B^+_m)} \hspace{-1pt}+\hspace{-1pt} \int_{\{y_3= 0\}} \hspace{-2pt} \sqrt{\rg}\,\rH \big|\bp^j \Lambda_\epsilon (\widetilde{\zeta}_m \widetilde{\w}_m \hspace{-1pt}\cdot\hspace{-1pt} \wtrN)\big|^2 dS \vspace{.2cm}\\
\displaystyle{} \qquad\qquad \le\hspace{-1pt} C_\delta \Big[\|\f\|^2_{L^2(\Omega)} \hspace{-1.5pt}+\hspace{-1pt} \|\w\|^2_{H^j(\Omega)} \Big] \hspace{-1.5pt}+\hspace{-1pt} \delta \big\|\bp^j \Lambda_\epsilon (\widetilde{\zeta}_m \widetilde{\w}_m \hspace{-1pt}\cdot\hspace{-1pt} \wtrN)\big\|^2_{H^1(B(0,r_m))}.
\end{array}
\end{equation}
Using \Poincare's inequality, there exists a constant $c > 0$ such that
$$
c \|\bp^j \Lambda_\epsilon (\widetilde{\zeta}_m \widetilde{\w}_m\hspace{-1pt}\cdot\wtrN)\|^2_{H^1(\Omega)} \le \big\|\nabla \bp^j \Lambda_\epsilon (\widetilde{\zeta}_m \widetilde{\w}_m\hspace{-1pt}\cdot\wtrN) \big\|^2_{L^2(B^+_m)} \hspace{-1pt}+\hspace{-1pt} \int_{\{y_3= 0\}} \hspace{-2pt} \sqrt{\rg}\,\rH \big|\bp^j \Lambda_\epsilon (\widetilde{\zeta}_m \widetilde{\w}_m \hspace{-1pt}\cdot\hspace{-1pt} \wtrN)\big|^2 dS\,,
$$
so by choosing $\delta>0$ small enough we find that
$$
\big\|\bp^j \Lambda_\epsilon (\widetilde{\zeta}_m \widetilde{\w}_m\hspace{-1pt}\cdot\wtrN) \big\|_{H^1(B^+_m)} \le C \Big[\|\f\|_{L^2(\Omega)} + \|\w\|_{H^j(\Omega)} \Big]\,.
$$
Since the right-hand side of the estimate above is independent of $\epsilon$, we can pass $\epsilon$ to the limit and obtain that
\begin{equation}\label{pf_thm_12.21_est4}
\big\|\bp^j (\widetilde{\zeta}_m \widetilde{\w}_m\hspace{-1pt}\cdot\wtrN) \big\|_{H^1(B^+_m)} \le C \Big[\|\f\|_{L^2(\Omega)} + \|\w\|_{H^j(\Omega)} \Big]\,.
\end{equation}
The estimate above provides the regularity of $\w$ in the normal direction.

To see the regularity of the tangential component of $\w$, an alternative test function has to be employed. Define
$$
\Varphi_2 = (-1)^j \zeta_m \bN \times \big[\Lambda_\epsilon \bp^{2j} \Lambda_\epsilon (\widetilde{\zeta}_m \widetilde{\w}_m \times \wtrN)\big]\circ \vartheta^{-1}_m\,.
$$
We note that since $\w \times \bN = 0$ on $\bdy\Omega$, $\Varphi_2 = 0$ on $\bdy\Omega$ so $\Varphi_2$ may be used as a test function. Since $\u \cdot (\v \times \w) = (\u \times \v) \cdot \w$, with $J$ and $A$ denoting $\det(\nabla \vartheta_m)$ and $(\nabla \vartheta_m)^{-1}$ respectively, using $b^{rs}$ to denote $J A^r_k A^s_k$ we find that
\begin{align*}
\int_{\Omega} \hspace{-2pt}\nabla \w : \nabla \Varphi_2\, dx &= (-1)^j \int_{B^+_m} \hspace{-2pt}b^{rs} \widetilde{\w}^i_m,_s \big[\widetilde{\zeta}_m \wtrN \hspace{-1pt}\times\hspace{-1pt} \Lambda_\epsilon \bp^{2j} \Lambda_\epsilon (\widetilde{\zeta}_m \widetilde{\w}_m \hspace{-1pt}\times\hspace{-1pt} \wtrN)\big]^i,_r dy \\
&= (-1)^j \int_{B^+_m}\hspace{-2pt} b^{rs} (\widetilde{\zeta}_m \widetilde{\w}^i_m,_s\hspace{-1pt}\times\hspace{-1pt} \wtrN)^i \Lambda_\epsilon \bp^{2j} \Lambda_\epsilon (\widetilde{\zeta}_m \widetilde{\w}_m \hspace{-1pt}\times\hspace{-1pt} \wtrN)^i,_r dy \\
&\quad + (-1)^j \int_{B^+_m}\hspace{-2pt} b^{rs} \widetilde{\w}^i_m,_s \big[(\widetilde{\zeta}_m \wtrN),_r \hspace{-1pt}\times\hspace{-1pt} \Lambda_\epsilon \bp^{2j} \Lambda_\epsilon (\widetilde{\zeta}_m \widetilde{\w}_m \hspace{-1pt}\times\hspace{-1pt} \wtrN)\big]^i dy \\
&= \int_{B^+_m} \bp^j \Lambda_\epsilon \big[b^{rs} (\widetilde{\zeta}_m \widetilde{\w}_m\hspace{-1pt}\times\hspace{-1pt} \wtrN)^i,_s\big] \bp^j \Lambda_\epsilon (\widetilde{\zeta}_m \widetilde{\w}_m \hspace{-1pt}\times\hspace{-1pt} \wtrN)^i,_r dy \\
&\quad - \hspace{-1pt}\int_{B^+_m} \hspace{-3.5pt}\bp^j \Lambda_\epsilon \big[b^{rs} \hspace{-1pt}\big((\widetilde{\zeta}_m,_s \widetilde{\w},_m \hspace{-1pt}\times\hspace{-1pt} \wtrN)^i \hspace{-2pt}+\hspace{-2pt} (\widetilde{\zeta}_m \widetilde{\w},_m \hspace{-1pt}\times\hspace{-1pt} \wtrN,_s\hspace{-2pt})^i\big)\hspace{-1pt}\big] \bp^j \Lambda_\epsilon (\widetilde{\zeta}_m \widetilde{\w}_m \hspace{-1pt}\times\hspace{-1pt} \wtrN)^i,_r dy \\
&\quad + \int_{B^+_m} \bp^j \hspace{-2pt} \Lambda_\epsilon b^{rs} \big[\widetilde{\w}^i_m,_s \times (\widetilde{\zeta}_m \wtrN),_r\big] \bp^j \Lambda_\epsilon (\widetilde{\zeta}_m \widetilde{\w}_m \hspace{-1pt}\times\hspace{-1pt} \wtrN)^i dy \,.
\end{align*}
Similar to the procedure of deriving (\ref{vector-valued_elliptic_est_id1}), by Leibniz's rule,
\begin{align*}
\int_{B^+_m} \bp^j \Lambda_\epsilon b^{rs} &(\widetilde{\zeta}_m \widetilde{\w}_m\times \wtrN)^i,_s \bp^j \Lambda_\epsilon (\widetilde{\zeta}_m \widetilde{\w}_m \times \wtrN)^i,_r dy \\
&\quad = \int_{B^+_m} b^{rs} \bp^j \Lambda_\epsilon (\widetilde{\zeta}_m \widetilde{\w}_m\times \wtrN)^i,_s \bp^j \Lambda_\epsilon (\widetilde{\zeta}_m \widetilde{\w}_m \times \wtrN)^i,_r dy \\
&\qquad + \int_{B^+_m} \bp b^{rs} \bp^{j-1} \Lambda_\epsilon (\widetilde{\zeta}_m \widetilde{\w}_m\times \wtrN)^i,_s \bp^j \Lambda_\epsilon (\widetilde{\zeta}_m \widetilde{\w}_m \times \wtrN)^i,_r dy \\
&\qquad + \int_{B^+_m} \bp \bigcomm{b^{rs}}{\Lambda_\epsilon} \bp^{j-1}(\widetilde{\zeta}_m \widetilde{\w}_m\times \wtrN)^i,_s \bp^j \Lambda_\epsilon (\widetilde{\zeta}_m \widetilde{\w}_m \times \wtrN)^i,_r dy \\
&\qquad + \sum_{k=0}^{j-2} \footnoteexp{$\displaystyle{}{{j-1}\choose{k}}$} \int_{B^+_m} \bp \Lambda_\epsilon \big[\bp^{j-1-k} b^{rs} \bp^k (\widetilde{\zeta}_m \widetilde{\w}_m\times \wtrN)^i,_s \big] \bp^j \Lambda_\epsilon (\widetilde{\zeta}_m \widetilde{\w}_m \times \wtrN)^i,_r dy \,.
\end{align*}
Since $\{\vartheta_m\}_{m=1}^M$ is chosen so that $A \approx \id$, $b^{rs}$ is positive-definitive. As a consequence, by the commutator estimate (\ref{commutator_estimates2}) and Young's inequality,
\begin{align*}
\int_{\Omega} \nabla \w : \nabla \Varphi_2\, dx &\ge \frac{1}{2}\, \big\|\bp^j \nabla \Lambda_\epsilon (\widetilde{\zeta}_m \widetilde{\w}_m\hspace{-1pt}\times\hspace{-1pt} \wtrN)\big\|^2_{L^2(B^+_m)} \\
&\quad - C \big\|\bp^{j-1} \nabla \Lambda_\epsilon (\widetilde{\zeta}_m \widetilde{\w}_m\hspace{-1pt}\times\hspace{-1pt} \wtrN)\big\|_{L^2(B^+_m)}\big\|\bp^j \nabla \Lambda_\epsilon (\widetilde{\zeta}_m \widetilde{\w}_m\hspace{-1pt}\times\hspace{-1pt} \wtrN)\big\|_{L^2(B^+_m)} \\
&\ge \frac{1}{4}\, \big\|\bp^j \nabla \Lambda_\epsilon (\widetilde{\zeta}_m \widetilde{\w}_m\hspace{-1pt}\times\hspace{-1pt} \wtrN)\big\|^2_{L^2(B^+_m)} \hspace{-1.5pt}-\hspace{-1pt} C \|\w\|^2_{H^j(\Omega)}.
\end{align*}
On the other hand,
\begin{align*}
\int_{\Omega} \f\cdot \Varphi_2\, dx &\le C \|\f\|_{H^{j-1}(\Omega)} \big\|\bp^{j+1} \Lambda_\epsilon (\widetilde{\zeta}_m \widetilde{\w}_m \hspace{-1pt}\times\hspace{-1pt} \wtrN)\big\|_{L^2(B(0,r_m))}\\
&\le C_\delta \|\f\|^2_{L^2(\Omega)} + \delta \big\|\bp^{j+1} \Lambda_\epsilon (\widetilde{\zeta}_m \widetilde{\w}_m \hspace{-1pt}\times\hspace{-1pt} \wtrN)\big\|^2_{L^2(B(0,r_m))}\,;
\end{align*}
thus using $\Varphi_2$ as a test function in (\ref{weakformHodgeEQ1}) and choosing $\delta>0$ small enough, by the fact that $\smallexp{$\displaystyle{}\int_{\bdy\Omega}$} \rH (\w\cdot \bN) (\Varphi_2 \cdot \bN)\, dS = 0$ we conclude that
\begin{equation}\label{pf_thm_12.21_est4.5}
\big\|\bp^j \Lambda_\epsilon (\widetilde{\zeta}_m \widetilde{\w}_m \hspace{-1pt}\times\hspace{-1pt} \wtrN)\big\|_{H^1(B^+_m)} \le C \Big[ \|\f\|_{L^2(\Omega)} + \|\w\|_{H^j(\Omega)}\Big]\,.
\end{equation}
Since the right-hand side is $\epsilon$-independent, we can pass $\epsilon$ to the limit and obtain that
\begin{equation}\label{pf_thm_12.21_est5}
\big\|\bp^j (\widetilde{\zeta}_m \widetilde{\w}_m \hspace{-1pt}\times\hspace{-1pt} \wtrN)\big\|_{H^1(B^+_m)} \le C \Big[ \|\f\|_{L^2(\Omega)} + \|\w\|_{H^j(\Omega)}\Big]\,.
\end{equation}
The estimate above provides the regularity of $\w$ in the tangential direction.

Since every vector $u$ can be expressed as $\u = \wtrN \times (\u\times \wtrN) + (\u \cdot \wtrN) \wtrN$, the combination of (\ref{pf_thm_12.21_est4}) and (\ref{pf_thm_12.21_est5}) then shows that
$$
\big\|\widetilde{\zeta}_m \bp^j \widetilde{\w}_m \big\|_{H^1(B^+_m)} \le C \Big[ \|\f\|_{L^2(\Omega)} + \|\w\|_{H^j(\Omega)}\Big]\,.
$$
Finally, we follow Step 4 of Theorem \ref{thm:vector-valued_elliptic_regularity} or Step 3 of Theorem \ref{thm:vector-valued_elliptic_eq_Sobolev_coeff} to conclude that
\begin{equation}\label{pf_thm_12.21_est_temp1}
\big\|\widetilde{\zeta}_m \nabla^j \widetilde{\w}_m\big\|_{H^1(B^+_m)} \le C \Big[\|\f\|_{L^2(\Omega)} + \|\w\|_{H^j(\Omega)} \Big]\,.
\end{equation}
Estimate (\ref{estimate_for_given_vorticity}) is concluded from combining the $H^1$-estimate (\ref{HodgeEQ_basic_estimate}), the interior estimate (\ref{vector-valued_interior_estimate1}) and the boundary estimate (\ref{pf_thm_12.21_est_temp1}).
\end{proof}

\subsubsection{The case that $\Omega$ is a general $H^{\rk+1}$-domain}\label{sec:d=divw_is_div_free}
If $\Omega$ is a general $H^{\rk+1}$-domain, the mean curvature $\rH$ can be negative on some portion of $\bdy\Omega$,
leading to a problematic Robin boundary condition (\ref{HodgeEQ1}c), with the wrong sign. To overcome this difficulty, we instead consider a similar problem defined on a ball containing $\Omega$.

\def\F {\text{\bf\emph{F}}}

Let $B(0,R)$ be an open ball so that $\Omega \cptsubset B(0,R)$, and let $\F$ denote a divergence-free vector field on $B(0,R)$ such that $\F = \f$ in $\cls{\Omega}$; that is, $\F$ is a divergence-free extension of $\f$. For a vector field $\f\in L^2(\Omega)$, such an $\F$ (in $B(0,R)\backslash \Omega$) can be obtained by first solving the elliptic equation
\begin{subequations}\label{divF=0a}
\begin{alignat}{2}
\Delta \phi &=0 &&\text{in}\quad B(0,R)\backslash \Omega\,,\\
\frac{\p \phi}{\p \bN} &= \f\cdot \bN \qquad&&\text{on}\quad\bdy\Omega\,,\\
\frac{\p \phi}{\p \bN} &= 0 \qquad&&\text{on}\quad\bdy B(0,R)\,,
\end{alignat}
\end{subequations}
and setting $\F = \nabla \phi$ on $B(0,R)\backslash \Omega$. We note that (\ref{divF=0a}) is solvable if the solvability condition
(\ref{solvability_condition_for_normal_trace_problem}) holds. To see this, let $\Omega_i$ be one of the connected component of $\Omega$, and let $\O_i$ be one of bounded connected components of
$\Omega^\complement_i$ with boundary $\Gamma_i$. Then $\Gamma_i$ must be one of the connected component of $\bdy\Omega$, and (\ref{divF=0a}) implies that in particular,
\begin{alignat*}{2}
\Delta \phi &=0 &&\text{in}\quad \O_i\,,\\
\frac{\p \phi}{\p \bN} &= \f\cdot \bN \qquad&&\text{on}\quad\Gamma_i\,.
\end{alignat*}
Therefore, (\ref{solvability_condition_for_normal_trace_problem}) has to be satisfied in order for the above equation to be solvable.
We also note that $\F\in L^2(B(0,R))$ even if $\f\in H^{\ell-1}(\Omega)$; thus $\F$ must be less regular than $\f$ due to the lack of continuity of the derivatives of $\F$ across $\bdy\Omega$.

Next,  consider the following elliptic system:
\begin{subequations}\label{HodgeEQ2}
\begin{alignat}{2}
- \Delta \w &= \F \qquad&&\text{in}\quad B(0,R)\,,\\
\rP_{\bN^\perp} \w &= {\bf 0} &&\text{on}\quad \bdy B(0,R)\,,\\
\frac{\p \w}{\p \bN} \cdot \bN + 2\, \rH (\w\cdot \bN) &= 0 &&\text{on}\quad \bdy B(0,R)\,.
\end{alignat}
\end{subequations}
By Lemma \ref{lem:curl_div_normal_trace_convex}, there exists a strong solution $\w \in H^2(B(0,R))$ to (\ref{HodgeEQ2}) (so that (\ref{HodgeEQ2}) also holds in a pointwise sense).

Now we show that $\w$ has zero divergence. Let $d = \div \w \in H^1(B(0,R))$. We claim that $d$ is a weak solution to
\begin{subequations}\label{d=divw_eq}
\begin{alignat}{2}
\Delta d &= 0 \qquad&&\text{in}\quad B(0,R)\,,\\
d &= 0 &&\text{on}\quad \bdy B(0,R)\,;
\end{alignat}
\end{subequations}
that is, $d \in H^1_0(B(0,R))$ and $d$ satisfies
\begin{equation}\label{weak_form_for_d}
\int_{B(0,R)} \nabla d \cdot \nabla \varphi\, dx = 0 \qquad\Forall \varphi \in H^1_0(B(0,R))\,.
\end{equation}
The boundary condition $d=0$ on $\bdy B(0,R)$ is obvious because of (\ref{wid2}) and (\ref{HodgeEQ2}b,c). To see (\ref{weak_form_for_d}), we note that it suffices to show that $\Delta d = 0$ in the sense of distribution, since $\D(B(0,R))$ is dense in $H^1_0(B(0,R))$. Let $\varphi \in \D(B(0,R))$, and define ${\boldsymbol\psi} = \nabla\varphi$. Then ${\boldsymbol\psi}\in \D(B(0,R))$, and
\begin{align*}
- \int_{B(0,R)} \Delta \w \cdot {\boldsymbol\psi}\,dx &= \int_{B(0,R)} \F \cdot {\boldsymbol\psi}\,dx = \int_{B(0,R)\backslash \Omega} \F \cdot \nabla \varphi\,dx + \int_\Omega \f \cdot \nabla \varphi\,dx \\
&= \int_{\bdy\Omega} (\f\cdot \bN - \frac{\p \phi}{\p \bN})\, \varphi\,dS = 0\,.
\end{align*}
On the other hand, since $\w\in H^2(B(0,R))$, we have $d\in H^1(B(0,R))$ and
\begin{align*}
- \int_{B(0,R)} \Delta \w \cdot {\boldsymbol\psi}\,dx &= \int_{B(0,R)} \nabla \w : \nabla {\boldsymbol\psi}\, dx = \int_{B(0,R)} \w^i,_j \varphi,_{ij} dx \\
&= - \int_{B(0,R)} d,_j \varphi,_j dx = - \int_{B(0,R)} \nabla d \cdot \nabla \varphi\, dx\,;
\end{align*}
thus we conclude (\ref{weak_form_for_d}). Therefore, $d$ is the weak solution to (\ref{d=divw_eq}) and so $d$ must vanish in $\Omega$ which implies that $\div \w = 0$ in $\Omega$. Finally, since $\w \in H^2(\Omega)$, applying (\ref{curlcurlu_identity}), we find that $\v = \curl \w \in H^1(\Omega)$ satisfies $\curl \v = \f$ in $\Omega$.

So far, we have shown that there exists $\v\in H^1(B(0,R))$ satisfying
\begin{alignat*}{2}
\curl \v &= \F \qquad&&\text{in}\quad B(0,R)\,,\\
\div \v &= 0 &&\text{in}\quad B(0,R)\,,\\
\v \cdot \bN &= 0 &&\text{on}\quad \bdy B(0,R)\,,
\end{alignat*}
which in particular shows that $\curl \v = \f$ in $\Omega$. It is not clear that if $\v$ possesses better regularity, however, since $\v$ has been constructed using a non-smooth forcing function $\F$. Let $p$ be the $H^2$-solution to the elliptic equation
\begin{alignat*}{2}
\Delta p &= 0 \qquad&&\text{in}\quad\Omega\,,\\
\frac{\p p}{\p \bN} &= - \v \cdot \bN\qquad&&\text{on}\quad\bdy\Omega\,,
\end{alignat*}
and define $\u = \v + \nabla p$; then, $\u$ is a solution to (\ref{ufromcurldiv}). We note that $\u\in H^1(\Omega)$ and satisfies
\begin{align}
\|\u\|_{H^1(\Omega)} &\le \|\v\|_{H^1(\Omega)} + \|\nabla p\|_{H^1(\Omega)} \le C(|\bdy\Omega|_{H^{\rk+0.5}}) \|\w\|_{H^2(\Omega)} \nonumber\\
&\le C(|\bdy\Omega|_{H^{\rk+0.5}}) \|\f\|_{L^2(\Omega)}\,. \label{curlu=f_H2_est}
\end{align}
In the following lemma, we show that while $\v$ and $\nabla p$ both have low regularity, in fact, their sum, $\u$, possesses $H^\ell$-regularity if $\f\in H^{\ell-1}(\Omega)$ for some $\ell \ge 2$.

\begin{lemma}\label{lem:curlu=f_est}
Let $\Omega\subseteq \bbR^3$ be a bounded $H^{\rk+1}$-domain for some $\rk > \smallexp{$\displaystyle{}\frac{3}{2}$}$\,. Then for all $\f\in H^{\ell-1}(\Omega)$ for some $1\le \ell \le \rk$, there exists a solution $\u\in H^\ell(\Omega)$ to {\rm(\ref{ufromcurldiv})} satisfying
\begin{equation}\label{curlu=f_Hl_est}
\|\u\|_{H^\ell(\Omega)} \le C(|\bdy\Omega|_{H^{\rk+0.5}}) \|\f\|_{H^{\ell-1}(\Omega)}\,.
\end{equation}
\end{lemma}
\begin{proof}
We again show that $u\in H^\ell(\Omega)$ by induction. We have shown the validity of the lemma for the case that $\ell=1$. Now suppose that $\ell \ge 2$ and $u\in H^j(\Omega)$ for some $j \le \ell-1$. Since $\u = \curl \w \in H^1(\Omega)$ satisfies (\ref{ufromcurldiv}b,c), using (\ref{int_curlcurl_id}) we find that $\u$ satisfies
\begin{align*}
& \int_\Omega \curl \u \cdot \curl \Varphi\, dx = \int_\Omega \nabla \u : \nabla \Varphi\,dx - \int_{\bdy\Omega} \h \cdot \Varphi\, dS \qquad\Forall \Varphi \in H^1_n(\Omega)\,,
\end{align*}
where in local chart $(\U,\vartheta)$ $\h$ is given by $\h\circ \vartheta = - g^{\alpha\beta} g^{\gamma\delta} \big[(\u\circ \vartheta)\cdot \vartheta,_\beta\big] b_{\alpha\gamma} \, \vartheta,_\delta$. On the other hand,
\begin{align*}
\int_\Omega \f \cdot \curl \Varphi\, dx &= \int_\Omega \varepsilon_{ijk} \f^i \Varphi^k,_j dx = - \int_\Omega \varepsilon_{ijk} \f^i,_j \Varphi^k dx + \int_{\bdy\Omega} \varepsilon_{ijk} \f^i \bN_j \Varphi^k dS \\
&= \int_\Omega \curl \f \cdot \Varphi\,dx + \int_{\bdy\Omega} (\f\times \bN) \cdot \Varphi\, dS \qquad\Forall \Varphi\in H^1(\Omega).
\end{align*}
Using (\ref{ufromcurldiv}a), we find that $u$ satisfies
$$
\int_\Omega \nabla \u : \nabla \Varphi\,dx = \int_\Omega \curl \f \cdot \Varphi\,dx + \int_{\bdy\Omega} (\f\times \bN + \h) \cdot \Varphi\, dS \qquad\Forall \Varphi \in H^1_n(\Omega)\,;
$$
thus, $\u$ is a weak solution of the following elliptic system:
\begin{subequations}\label{curlcurlu=curlf}
\begin{alignat}{2}
- \Delta \u &= \curl\,\f \qquad&&\text{in}\quad\Omega\,,\\
\u \cdot \bN &= 0 &&\text{on}\quad\bdy\Omega\,,\\
\rP_{\bN^\perp}\Big(\frac{\p \u}{\p \bN} - \f \times \bN - \h\Big) &= {\bf 0}\qquad&&\text{on}\quad\bdy\Omega\,.
\end{alignat}
\end{subequations}

Let us first assume that $\rk \ge 3$. Then $\rk-1.5 > 1 = \smallexp{$\displaystyle{}\frac{2}{2}$}$\,. Moreover, $j-0.5 \le \rk-1.5$\,; thus Proposition \ref{prop:HkHl_product} shows that
\begin{equation}\label{h_est}
\|\h\|_{H^{j-0.5}(\bdy\Omega)} \le C(|\bdy\Omega|_{H^{\rk+0.5}}) \|b\|_{H^{\rk-1.5}(\{y_3=0\})} \|\u\|_{H^{j-0.5}(\bdy\Omega)} \le C(|\bdy\Omega|_{H^{\rk+0.5}}) \|\u\|_{H^j(\Omega)}\,.
\end{equation}
Therefore, by Corollary \ref{cor:vector-valued_elliptic_eq_Sobolev_coeff} (with $a^{jk} = \delta^{jk}$ and $\rw = \bN$) we conclude that
\begin{align*}
\|\u\|_{H^{j+1}(\Omega)} &\le C(|\bdy\Omega|_{H^{\rk+0.5}}) \Big[\|\curl\,\f\|_{H^{j-1}(\Omega)} + \|\f\times \bN + \h\|_{H^{j-0.5}(\bdy\Omega)}\Big] \\
&\le C(|\bdy\Omega|_{H^{\rk+0.5}}) \Big[\|\f\|_{H^j(\Omega)} + \|\u\|_{H^j(\Omega)}\Big]
\end{align*}
which implies $\u\in H^{j+1}(\Omega)$. Estimate (\ref{curlu=f_Hl_est}) then is concluded from estimate (\ref{curlu=f_H2_est}), interpolation and Young's inequality.

The case that $\rk = 2$ (and $\ell = 2$) is a bit tricky. In this case (\ref{h_est}) cannot be applied since $b,\u$ both belong to $H^{0.5}(\bdy\Omega)$ while $H^{0.5}(\bdy\Omega)$ is not a multiplicative algebra. To see why $\u$ indeed belongs to $H^2(\Omega)$ if $\f\in H^1(\Omega)$, let $\u^\epsilon$ to be the solution to
\begin{subequations}\label{ufromcurldiv_eps}
\begin{alignat}{2}
\lambda \u^\epsilon - \Delta \u^\epsilon &= \curl\,\f + \lambda \u \qquad&&\text{in}\quad\Omega\,,\\
\u^\epsilon \cdot \bN &= 0 &&\text{on}\quad\bdy\Omega\,,\\
\rP_{\bN^\perp}\Big(\dfrac{\p \u^\epsilon}{\p \bN} - \f \times \bN - \h_\epsilon\Big) &= {\bf 0}\qquad&&\text{on}\quad\bdy\Omega\,,
\end{alignat}
\end{subequations}
where $\u$ on the right-hand side of (\ref{ufromcurldiv_eps}a) is the solution to (\ref{ufromcurldiv}), $\h_\epsilon$ is a smooth version of $\h$ given by
$$
\h_\epsilon = - \sum_{m=1}^K \zeta_m \Big[g^{\alpha\beta}_m g^{\gamma\delta}_m \big((\u^\epsilon\circ \vartheta_m)\cdot \vartheta_m,_\beta\big) (\Lambda_\epsilon b_{m\alpha\gamma})\, \vartheta_m,_\delta \Big] \circ \vartheta_m^{-1}
$$
in which $\Lambda_\epsilon$ is the horizontal convolution defined in Section \ref{sec:convolution}, and $\lambda \gg 1$ is a big constant so that the bilinear form
$$
B(\u^\epsilon,\Varphi) = \lambda (\u^\epsilon,\Varphi)_{L^2(\Omega)} + (\nabla \u^\epsilon, \nabla \Varphi)_{L^2(\Omega)} + \int_{\bdy\Omega} \h_\varepsilon \cdot \Varphi\, dS
$$
is coercive on $H^1_n(\Omega) \times H^1_n(\Omega)$. Since $\Lambda_\epsilon b_m$ is smooth, we find that $\h_\epsilon \in H^{0.5}(\bdy\Omega)$ satisfying
\begin{align*}
\|\h_\epsilon\|_{H^{0.5}(\bdy\Omega)} &\le C(|\bdy\Omega|_{H^{2.5}}) \Big[\|\p_y \vartheta_m\|^6_{H^{1.25}(B(0,r_m)\cap \{y_3 = 0\})} \|\Lambda_\epsilon b_m\|_{H^{1.25}(B(0,r_m)\cap \{y_3 = 0\})} \|\u^\epsilon\|_{H^{0.5}(\bdy\Omega)} \Big] \\
&\le C_\epsilon \|\u^\epsilon\|_{H^1(\Omega)} \le C_\epsilon \Big[\|\f\|_{L^2(\Omega)} + \lambda \|\u\|_{L^2(\Omega)}\Big] \le C_\epsilon \|\f\|_{L^2(\Omega)} \,,
\end{align*}
where the dependence on $\epsilon$ in the constant $C_\epsilon$ is due to the horizontal convolution $\Lambda_\epsilon$. As a sequence, $\u^\epsilon \in H^2(\Omega)$, and this fact further shows that $\h_\epsilon$ satisfies
\begin{align*}
\|\h_\epsilon\|_{H^{0.5}(\bdy\Omega)} &\le C(|\bdy\Omega|_{H^{2.5}}) \Big[\|\p_y \vartheta_m\|^6_{H^{1.25}(B(0,r_m)\cap \{y_3 = 0\})} \|b_m\|_{H^{0.5}(B(0,r_m)\cap \{y_3 = 0\})} \|\u^\epsilon\|_{H^{1.25}(\bdy\Omega)} \Big] \\
&\le C(|\bdy\Omega|_{H^{2.5}}) \|\u^\epsilon\|_{H^{1.75}(\Omega)} \le C(|\bdy\Omega|_{H^{2.5}}) \|\u^\epsilon\|^\frac{1}{4}_{H^1(\Omega)} \|\u^\epsilon\|^\frac{3}{4}_{H^2(\Omega)}\,.
\end{align*}
By Young's inequality, we find that that $\u^\epsilon$ satisfies
\begin{align*}
\|\u^\epsilon\|_{H^2(\Omega)} &\le C(|\bdy\Omega|) \Big[\|\curl \f\|_{L^2(\Omega)} + \|\f \times \bN + \h_\epsilon\|_{H^{0.5}(\bdy\Omega)}\Big] \\
&\le C(|\bdy\Omega|_{H^{2.5}},\delta) \|\f\|_{H^1(\Omega)} + \delta \|\u^\epsilon\|_{H^2(\Omega)}\,.
\end{align*}
Choosing $\delta > 0$ small enough, we conclude that $\u^\epsilon$ has a uniform $H^2$ upper bound and possesses an $H^1$ convergent subsequence $\u^{\epsilon_j}$ with limit $\v$. This limit $\v$ must be $\u$ since $\u$ is also a weak solution to (\ref{ufromcurldiv_eps}) and the strong solution to (\ref{ufromcurldiv_eps}) is unique (by the Lax-Milgram theorem). Moreover, $\u$ satisfies (\ref{curlu=f_Hl_est}) (for $\ell = 2$).
\end{proof}

Lemma \ref{lem:curlu=f_est} together with the elliptic estimate
$$
\|\nabla \phi\|_{H^{j+1}(\Omega)} \le C \big[\|g\|_{H^j(\Omega)} + \|h\|_{H^{j-0.5}(\bdy\Omega)}\big]
$$
for the solution $\phi$ to (\ref{phi_eq}) then concludes the first part of Theorem \ref{thm:main_thm1}.

\subsection{Solutions with prescribed tangential trace} Having considered the boundary condition $\v \cdot \bN =h$, we now establish the existence and uniqueness of the following problem:
\begin{subequations}\label{ufromcurldiv_tan}
\begin{alignat}{2}
\curl \v &= \f \qquad&&\text{in}\quad \Omega\,,\\
\div \v &= g &&\text{in}\quad \Omega\,,\\
\v \times \bN &= \h &&\text{on}\quad \bdy\Omega\,,
\end{alignat}
\end{subequations}
in which (\ref{ufromcurldiv_tan}c) prescribes the tangential trace of $\v$. We impose the following conditions on the forcing functions $\f$ and $\h$:
\begin{subequations}\label{solvability_condition_for_curl_div_tangential_trace}
\begin{align}
\div \f = 0 \text{ \ in \ }\Omega,\quad \f \text{ satisfies (\ref{solvability_condition_for_normal_trace_problem})}, \quad\text{and}\quad \h\cdot \bN = 0 \text{ \ on \ } \bdy\Omega\,.
\end{align}
\end{subequations}
Moreover, using (\ref{wid1}) and the identity $\bN \times (\v \times \bN) = \v - (\v \cdot \bN) \bN$ on $\bdy\Omega$, we find that $\f$ and $\h$ must have the relation
$$
\f \cdot \bN = \bdydiv  \h \quad\text{on}\quad\bdy\Omega\,. \eqno{\rm(\ref{solvability_condition_for_curl_div_tangential_trace}b)}
$$

For (\ref{ufromcurldiv_tan}) to have a solution, one additional solvability condition has to be imposed. Let $\u$ be a solution to (\ref{ufromcurldiv}), and $\phi$ be the solution to
\begin{alignat*}{2}
\Delta \phi &= g \qquad&&\text{in}\quad\Omega\,,\\
\phi &= 0 &&\text{on}\quad\bdy\Omega\,.
\end{alignat*}
Then $\w = \v - \u - \nabla \phi$ satisfies
\begin{subequations}\label{ufromcurldiv_tan1}
\begin{alignat}{2}
\curl \w &= {\bf 0} &&\text{in}\quad \Omega\,,\\
\div \w &= 0 &&\text{in}\quad \Omega\,,\\
\w \times \bN &= \h - \u\times \bN \qquad&&\text{on}\quad \bdy\Omega\,.
\end{alignat}
\end{subequations}
Taking the cross product of $\bN$ with (\ref{ufromcurldiv_tan1}c), we find that
$$
\w - (\w\cdot \bN)\bN = \bN \times \h - \big[\u - (\u \cdot \bN)\bN\big] \qquad\text{on}\quad \bdy\Omega\,.
$$
If $C$ is a closed curve on $\bdy\Omega$ enclosing a surface $\Sigma \subseteq \cls{\Omega}$ so that $C = \bdy\Sigma$ with a parametrization ${\mathbf r}$ and $\bfn$ is the unit normal on $\Sigma$ compatible with the orientation of $C$, then the Stokes theorem implies that
\begin{align*}
0 &= \int_\Sigma \curl \w \hspace{-1pt}\cdot\hspace{-1pt} \bfn\, dS = {\oint}_{\hspace{-3pt}C} \w \hspace{-1pt}\cdot\hspace{-1pt} d {\mathbf r} = {\oint}_{\hspace{-3pt}C} [\w - (\w \hspace{-1pt}\cdot\hspace{-1pt} \bN)\bN] \hspace{-1pt}\cdot\hspace{-1pt} d {\mathbf r} = {\oint}_{\hspace{-3pt}C} (\bN \times \h - \u)\hspace{-1pt}\cdot\hspace{-1pt} d {\mathbf r} \\
&= {\oint}_{\hspace{-3pt}C} (\bN \times \h)\hspace{-1pt}\cdot\hspace{-1pt} d {\mathbf r} - \int_\Sigma \curl \u \cdot \bfn\, dS = {\oint}_{\hspace{-3pt}C} (\bN \times \h)\hspace{-1pt}\cdot\hspace{-1pt} d {\mathbf r} - \int_\Sigma \f \cdot \bfn\, dS\,,
\end{align*}
As a consequence, 
$$
\int_\Sigma \f \cdot \bfn\, dS = {\oint}_{\hspace{-3pt}\bdy\Sigma} (\bN\times \h)\cdot d {\mathbf r} \qquad \Forall \Sigma \subseteq \cls{\Omega}, \bdy\Sigma \subseteq \bdy\Omega\,. \eqno{\rm(\ref{solvability_condition_for_curl_div_tangential_trace}c)}
$$
Conditions
(\ref{solvability_condition_for_curl_div_tangential_trace}a), (\ref{solvability_condition_for_curl_div_tangential_trace}b) and (\ref{solvability_condition_for_curl_div_tangential_trace}c)
constitute the solvability conditions for equation (\ref{ufromcurldiv_tan}). We remark that
(\ref{solvability_condition_for_curl_div_tangential_trace}c) follows from (\ref{solvability_condition_for_curl_div_tangential_trace}b) if $\Omega$ is
simply connected; however, if $\Omega$ is not simply connected, then (\ref{solvability_condition_for_curl_div_tangential_trace}c) is not necessarily true even if (\ref{solvability_condition_for_curl_div_tangential_trace}b) holds.

\subsubsection{Uniqueness of solutions}
The kernel of the tangential trace problem (\ref{ufromcurldiv_tan}) has been well studied in  \cite{FoTe1978, Ge1979, Picard1982, BeDoGa1985, AmBeDaGi1998, KoYa2009, AmSe2013} and references therein. We shall  establish uniqueness for the case that $\Omega$ is of class $H^{\rk+1}$ for some $\rk > \dfrac{3}{2}$. Without loss of generality, we can assume  that $\Omega$ is a connected bounded open set. We let $\{\Gamma_i\}_{i=0}^I$ denote the connected components of $\bdy\Omega$ in which $\Gamma_0$ is the boundary of the unbounded connected component of $\Omega^\complement$. To establish uniqueness of  solutions to (\ref{ufromcurldiv_tan}), we look for solutions to the equation
\begin{subequations}\label{homogeneous_curl_div_eq2}
\begin{alignat}{2}
\curl \v &= {\bf 0} \qquad&&\text{in}\quad\Omega\,,\\
\div \v &= 0 &&\text{in}\quad\Omega\,,\\
\v \times \bN &= {\bf 0} &&\text{on}\quad \bdy\Omega\,.
\end{alignat}
\end{subequations}
If $I\ge 1$ (which means $\bdy\Omega$ has multiple connected components), let $\{r_i\}_{i=1}^I$ solve
\begin{subequations}\label{kernel_of_the_tangential_trace_problem}
\begin{alignat}{2}
- \Delta r_i &= 0 &&\text{in}\quad\Omega\,,\\
r_i &= 0 &&\text{on}\quad \Gamma_0\,,\\
r_i &\text{ is constant } &&\text{on}\quad\Gamma_j \ \ \Forall 1\le j\le I\,,\\
\int_{\Gamma_j} \frac{\p r_i}{\p \bN} dS &= \delta_{ij}\,, &&\hspace{-.5cm}\int_{\Gamma_0} \frac{\p r_i}{\p \bN} dS = -1\,,
\end{alignat}
\end{subequations}
whose existence is guaranteed by the Lax-Milgram theorem applied to the variational problem
\begin{equation}\label{variational_problem}
\begin{array}{l}
\displaystyle{}\int_\Omega \nabla r_i \cdot \nabla \varphi\, dx = \varphi\big|_{\Gamma_i} \ \ \\
\displaystyle{}\qquad\qquad \Forall \varphi \in \big\{q \in H^1(\Omega)\,\big|\, q|_{\Gamma_0} = 0 \text{ and }  q \text{ is constant on each $\Gamma_i$ for all $1\le i\le I$}\big\}\,.
\end{array}
\end{equation}
In fact, let $r_i$ be the solution to the variation problem above. Define $C^\ell_i = r_i\big|_{\Gamma_\ell}$, and let $q_i$ be the solution to (\ref{q_eq}). Then $q_i\in H^{\rk+1}(\Omega)$ by Corollary \ref{cor:scalar_elliptic_eq_Sobolev_coeff}. Moreover, $r_i - C^\ell_i q_\ell \in H^1_0(\Omega)$ can be used as a test function in (\ref{variational_problem}); thus
$$
\int_\Omega \nabla r_i \cdot \nabla (r_i - C^\ell_i q_\ell) dx = 0\,.
$$
As a consequence,
\begin{align*}
& \int_\Omega \nabla (r_i - C^j_i q_j) \cdot \nabla (r_i - C^\ell_i q_\ell) dx = - C^j_i \int_\Omega \nabla q_j \cdot \nabla (r_i - C^\ell_i q_\ell) dx \\
&\qquad = - C^j_i \int_{\bdy\Omega} \frac{\p q_j}{\p \bN} r_i dS + C^j_i C^\ell_i \int_{\bdy\Omega} \frac{\p q_j}{\p \bN} q_\ell dS = - C^j_i C^\ell_i \int_{\Gamma_\ell} \frac{\p q_j}{\p \bN} dS + C^j_i C^\ell_i \int_{\Gamma_\ell} \frac{\p q_j}{\p \bN} dS = 0\,;
\end{align*}
hence $r_i - C^j_i q_j$ is a constant. This constant must be zero since $r_i$ and $q_j$ both vanish on $\Gamma_0$; thus we establish the identity
\begin{equation}\label{r_in_terms_of_q}
r_i = C^\ell_i q_\ell \qquad\Forall 1\le i\le I\,.
\end{equation}
Therefore, $r_i \in H^{\rk+1}(\Omega)$. Integrating by parts shows that
\begin{align*}
&- \int_\Omega \varphi\Delta r_i\, dx + \sum_{j=1}^I \varphi\big|_{\Gamma_j}\int_{\Gamma_j} \frac{\p r_i}{\p \bN} \, dS = \varphi\big|_{\Gamma_i} \\
&\qquad\qquad\Forall \varphi \in \big\{q \in H^1(\Omega)\,\big|\, q|_{\Gamma_0} = 0 \text{ and } q \text{ is constant on each $\Gamma_i$ for all $1\le i\le I$}\big\}\,.
\end{align*}
Choosing $\varphi \in H^1_0(\Omega)$ arbitrarily, we conclude that $r_i$ satisfies (\ref{kernel_of_the_tangential_trace_problem}a). Since $\varphi$ can be chosen as an arbitrary constant on $\Gamma_k$, we must have that $\smallexp{$\displaystyle{}\int_{\Gamma_j} \frac{\p r_i}{\p \bN}$} dS = \delta_{ij}$; thus
$$
\int_{\Gamma_0} \frac{\p r_i}{\p \bN} dS = \int_{\Gamma_0} \frac{\p r_i}{\p \bN} dS + \sum_{j=1}^I \int_{\Gamma_j} \frac{\p r_i}{\p \bN} dS - 1 = \int_{\bdy\Omega} \frac{\p r_i}{\p \bN} dS - 1 = \int_\Omega \Delta r_i dx - 1 = -1\,.
$$
In other words, $r_i$ satisfies (\ref{kernel_of_the_tangential_trace_problem}d); hence $r_i$ is a strong solution to (\ref{kernel_of_the_tangential_trace_problem}). We note that $\nabla r_i$ is not identically zero in $\Omega$ since $r_i$ cannot be constant in $\Omega$.

Let $\v \in H^1(\Omega)$ be a solution to (\ref{homogeneous_curl_div_eq2}). Define $\sbF \in H^1(\Omega)$ by
$$
\sbF = \v - \sum_{i=1}^I \left(\int_{\Gamma_i} \v\cdot \bN dS\right) \nabla r_i\,.
$$
Then $\curl \sbF = {\bf 0}$ in $\Omega$. Using (\ref{kernel_of_the_tangential_trace_problem}b,c), we find that
$$
\sbF \times \bN = \v \times \bN - \sum_{i=1}^I \left(\int_{\Gamma_i} \v\cdot \bN dS\right) (\nabla r_i \times \bN) = - \sum_{i=1}^I \left(\int_{\Gamma_i} \v\cdot \bN dS\right) (\bdygrad r_i \times \bN) = 0 \ \ \text{on}\ \ \bdy\Omega\,.
$$
Moreover, for $1\le k\le I$,
$$
\int_{\Gamma_k} \sbF \cdot \bN dS = \int_{\Gamma_k} \v \cdot \bN dS - \sum_{i=1}^I \left(\int_{\Gamma_i} \v\cdot \bN dS\right) \int_{\Gamma_k} \frac{\p r_i}{\p \bN} dS = 0\,.
$$
Since $\div \F = 0$ in $\Omega$,
$$
\int_{\Gamma_0} \sbF \cdot \bN dS = \int_{\Gamma_0} \sbF \cdot \bN dS + \sum_{k=1}^I \int_{\Gamma_k} \sbF \cdot \bN dS = \int_{\bdy\Omega} \sbF \cdot \bN dS = \int_{\bdy\Omega} \div \sbF dx = 0\,;
$$
hence,  $\F$ satisfies (\ref{solvability_condition_for_normal_trace_problem}). Lemma \ref{lem:curlu=f_est} then guarantees the existence of a vector $\Varpsi\in H^2(\Omega)$ satisfying
\begin{alignat*}{2}
\curl \Varpsi &= \F \qquad&&\text{in}\quad \Omega\,,\\
\div \Varpsi &= 0 &&\text{in}\quad \Omega\,,\\
\Varpsi \cdot \bN &= 0 &&\text{on}\quad \bdy\Omega\,.
\end{alignat*}
As a consequence,
$$
\|\sbF\|^2_{L^2(\Omega)} = \int_\Omega \sbF \cdot \curl \Varpsi dx = \int_\Omega \curl \sbF \cdot \Varpsi dx + \int_{\bdy\Omega} (\bN \times \sbF) \cdot \Varpsi dS = 0
$$
which implies that $\sbF = 0$. In other words, if $\v$ is a solution to (\ref{homogeneous_curl_div_eq2}), then
\begin{equation}\label{kernel_representation}
\v = \sum_{i=1}^I \left(\int_{\Gamma_i} \v\cdot \bN dS\right) \nabla r_i \qquad\text{in}\quad\Omega\,;
\end{equation}
that is, $\{\nabla r_i\}_{i=1}^I$ spans the solution space of (\ref{homogeneous_curl_div_eq2}). Therefore, as long as the boundary of a connected component of $\Omega$ has only one connected component, then only the trivial solution to (\ref{homogeneous_curl_div_eq2}) exists, and
uniqueness is established.

\begin{remark}
The identities {\rm(\ref{r_in_terms_of_q})} and {\rm(\ref{kernel_representation})} show that $\{\nabla q_i\}_{i=1}^I$ span the solution space of {\rm(\ref{homogeneous_curl_div_eq2})}; see, for example, \cite{Ge1979, Picard1982, BeDoGa1985, AmBeDaGi1998, KoYa2009, AmSe2013}.
\end{remark}

\subsubsection{Existence of solutions}
Let the pair $(\w,p)$ denote the solutions to the following elliptic problems:
\begin{alignat*}{2}
\Delta \w &= {\bf 0} &&\text{in}\quad\Omega\,,\\
\w &= \bN \times \h \qquad&&\text{on}\quad\bdy\Omega\,,
\end{alignat*}
and
\begin{alignat*}{2}
\Delta p &= g - \div \w \qquad&&\text{in}\quad\Omega\,,\\
p &= 0 &&\text{on}\quad\bdy\Omega\,.
\end{alignat*}
Then $(\w,p)$ satisfies
\begin{equation}\label{wp_div_Lap_eq_estimate}
\|\w\|_{H^\ell(\Omega)} + \|p\|_{H^{\ell+1}(\Omega)} \le C(|\bdy\Omega|_{H^{\rk+0.5}}) \big[\|g\|_{H^{\ell-1}(\Omega)} + \|\h\|_{H^{\ell-0.5}(\bdy\Omega)} \big]\,.
\end{equation}
We note that if $\u$ is a solution to the equation
\begin{subequations}\label{ufromcurldiv_tan2}
\begin{alignat}{2}
\curl \u &= \f - \curl \w \qquad&&\text{in}\quad\Omega\,,\\
\div \u &= 0 &&\text{in}\quad\Omega\,,\\
\u \times \bN &= {\bf 0} &&\text{on}\quad\bdy\Omega\,,
\end{alignat}
\end{subequations}
then $\v = \u + \w + \nabla p$ is a solution to (\ref{ufromcurldiv_tan}).

We first establish existence of an $H^1(\Omega)$ solution to (\ref{ufromcurldiv_tan2}), and then employ our regularity theory on Sobolev class domains to show
that  solutions to (\ref{ufromcurldiv_tan}) have the desired $H^\ell(\Omega)$-regularity stated in Theorem \ref{thm:main_thm1}.
We use the following lemma, which is Theorem 3.17 in \cite{AmBeDaGi1998}, to establish the existence of a  $\u \in H^1(\Omega)$ solving
(\ref{ufromcurldiv_tan2}).
\begin{lemma}\label{thm:existence_to_tangential_trace_problem}
Suppose that $\Omega \subseteq \bbR^3$ is a Lipschitz domain, and $\{\Sigma_j\}_{j=1}^J$ are cuts of $\Omega$; that is,
$\Sigma_j \subseteq \Omega$ for all
$j\in \{1,\cdots, J\}$ are connected smooth $2$-manifolds with unit normal $\bfn$ such that $\Omega^0 \equiv \Omega \backslash \bigcup\limits_{j=1}^J \Sigma_j$ is simply connected, where $J$ is the minimal number of cuts required. Then a function $\text{\bf\emph{F}} \in H^1(\Omega)$ satisfies $\div \text{\bf\emph{F}} = 0$ in $\Omega$, $\text{\bf\emph{F}}\cdot \bN = 0$ on $\bdy\Omega$, and
\begin{align*}
\int_{\Sigma_j} \text{\bf\emph{F}} \cdot \bfn \,dS &= 0 \text{ for each cut $\Sigma_j$, $1\le j\le J$};
\end{align*}
if and only if there exists a unique a vector potential $\u \in H^1(\Omega)$ satisfying
\begin{alignat*}{2}
\curl \u &= \text{\bf\emph{F}} \qquad&&\text{in}\quad\Omega\,,\\
\div \u &= 0 &&\text{in}\quad\Omega\,,\\
\u \times \bN &= {\bf 0} &&\text{on}\quad\bdy\Omega\,,\\
\int_\Gamma \u \cdot \bN \,dS &= 0 &&\text{for each connected component $\Gamma$ of $\bdy\Omega$}\,.
\end{alignat*}
\end{lemma}
Now we prove the existence of a solution $\u\in H^1(\Omega)$ to (\ref{ufromcurldiv_tan2}). First, noting that our domain $\Omega$ is of class
$H^{\rk+1}$ with $\rk > \novertwo$, by the  Sobolev embedding theorem,  $\Omega$ is a Lipschitz domain. Let $\text{\bf\emph{F}} = \f - \curl \w$. It is clear that $\div\text{\bf\emph{F}} = 0$ in $\Omega$. By the fact that
$$
\h\cdot \bN =0,\quad \w = \bN \times \h \quad\text{and}\quad \f \cdot \bN = \bdydiv  \h \quad\text{on \ $\bdy\Omega$}\,,
$$
(\ref{wid1}) implies that
$$
\text{\bf\emph{F}} \cdot \bN = \f\cdot \bN - \bdydiv  (\w \times \bN) = \f \cdot \bN - \bdydiv  \h = 0 \quad\text{on \ $\bdy\Omega$}.
$$
Finally, let $\text{\bf\emph{t}}$ denote the unit tangent vector on $\bdy\Sigma$ such that $\widetilde{\bfn} = \text{\bf\emph{t}} \times \bfn$ is the outward pointing unit normal to $\Sigma_j$. Since $\f$ satisfies (\ref{solvability_condition_for_curl_div_tangential_trace}c), using (\ref{wid1}) again we conclude that
\begin{align*}
\int_{\Sigma_j} \text{\bf\emph{F}} \cdot \bfn \,dS &= \int_{\Sigma_j} \f \cdot \bfn \,dS - \int_{\Sigma_j} \div_{\Sigma_j} (\w\times \bfn) \,dS = \int_{\Sigma_j} \f \cdot \bfn \,dS - \int_{\bdy\Sigma_j} (\w \times \bfn) \cdot \widetilde{\bfn} \,ds \\
&= \int_{\Sigma_j} \f \cdot \bfn \,dS - \int_{\bdy\Sigma_j} \w \cdot (\bfn \times \widetilde{\bfn}) \,ds = \int_{\Sigma_j} \f \cdot \bfn \,dS - \int_{\bdy\Sigma_j} \w \cdot \text{\bf\emph{t}} \,ds\\
&= \int_{\Sigma_j} \f \cdot \bfn \,dS - {\oint}_{\hspace{-3pt}\bdy\Sigma_j} (\bN\times \h)\cdot d {\mathbf r} = 0\,.
\end{align*}
Therefore, \text{\bf\emph{F}} satisfies all the conditions of Lemma \ref{thm:existence_to_tangential_trace_problem}, and so we have established the
existence of a solution $\u\in H^1(\Omega)$ to (\ref{ufromcurldiv_tan2}).  We next establish the regularity of this solution.

\subsubsection{Regularity of solutions}
We follow the proof of Lemma \ref{lem:curlu=f_est} to establish the following
\begin{lemma}
Let $\Omega\subseteq \bbR^3$ be a bounded $H^{\rk+1}$-domain for some $\rk > \smallexp{$\displaystyle{}\frac{3}{2}$}$\,. Then for all $(\f,\w) \in H^{\ell-1}(\Omega)\times H^\ell(\Omega)$ for some $1\le \ell \le \rk$, there exists a solution $\u\in H^\ell(\Omega)$ to {\rm(\ref{ufromcurldiv_tan2})} satisfying
\begin{equation}\label{curlu=f_tangential_Hl_est}
\|\u\|_{H^\ell(\Omega)} \le C(|\bdy\Omega|_{H^{\rk+0.5}}) \big[\|\f\|_{H^{\ell-1}(\Omega)} + \|\w\|_{H^\ell(\Omega)}\big] \,.
\end{equation}
\end{lemma}
\begin{proof}
It suffices to prove the case $\ell \ge 2$. Using (\ref{int_curlcurl_id}) we find that if $\Varphi \in H^1_\tau(\Omega)$,
\begin{align*}
\int_\Omega \curl (\f - \curl \w) \cdot \Varphi \,dx &= \int_{\bdy\Omega} (\f - \curl \w) \cdot (\Varphi \times \bN) \,dS + \int_\Omega \f \cdot \curl \Varphi \,dx = \int_\Omega \curl \u \cdot \curl \Varphi \,dx \\
&= \int_\Omega \nabla \u : \nabla \Varphi + \int_{\bdy\Omega} (\u \cdot \bN) \Big[\bdydiv  \Varphi + 2 \rH (\Varphi \cdot \bN) \Big] \,dS \\
&= \int_\Omega \nabla \u : \nabla \Varphi \,dx + \int_{\bdy\Omega} 2 \rH (\u \cdot \bN) (\Varphi \cdot \bN)\,dS\,.
\end{align*}
In other words, $\u$ is a weak solution of the following elliptic system:
\begin{subequations}\label{curlcurlu=curlf2}
\begin{alignat}{2}
- \Delta \u &= \curl\,\f + \Delta \w \qquad&&\text{in}\quad\Omega\,,\\
\u \times \bN &= 0 &&\text{on}\quad\bdy\Omega\,,\\
\frac{\p \u}{\p \bN} \cdot \bN + 2 \rH (\u \cdot \bN) &= 0 &&\text{on}\quad\bdy\Omega\,.
\end{alignat}
\end{subequations}
Similar to the proof of Lemma \ref{lem:curlu=f_est}, by induction we conclude that $\u \in H^\ell(\Omega)$ satisfies
$$
\|\u\|_{H^\ell(\Omega)} \le C \|\curl\, \f + \Delta \w\|_{H^{\ell-2}(\Omega)}
$$
which concludes the lemma.
\end{proof}
The regularity of $\u$ together with the inequality (\ref{wp_div_Lap_eq_estimate}) concludes the proof of regularity of the tangential trace problem, hence we have finished the proof of Theorem \ref{thm:main_thm1}.

\section{The Proof of Theorem {\rm\ref{thm:main_thm3}}}\label{sec6}
Now we proceed to the proof of Theorem {\rm\ref{thm:main_thm3}}. We only prove (\ref{Hodge_elliptic_estimate1b}) since the proof of (\ref{Hodge_elliptic_estimate2b}) is similar.
By assumption $\bdy \Omega $ is in a small tubular neighborhood of the normal bundle over $\bdy \mathcal{D} $; hence, there is height function
$h(x,t)$ such that each point on $\bdy \Omega$ is given by $x+ h(x)\nn(x)$, $x \in \bdy \mathcal{D} $, where $\nn$ is the outward-pointing unit normal to $\bdy
\mathcal{D} $.
 Let $\psi:\D \to \bbR^2$ solve
\begin{alignat*}{2}
\Delta \psi &= 0 &&\text{in}\quad\D\,,\\
\psi &= e + h \nn \qquad&&\text{on}\quad\bdy\D\,,
\end{alignat*}
where $e$ is the identity map. Then $\psi:\bdy\D \to \bdy\Omega$, and standard elliptic estimates show that for some constant
$C = C(|\bdy\D|_{H^{\rk+0.5}})$,
\begin{equation}\label{psi_est}
\|\nabla \psi - \id\|_{H^\rk(\D)} \le C \|h\|_{H^{\rk+0.5}(\bdy\D)} \le C \epsilon \ll 1
\end{equation}
which further shows that $\psi:\D \to \Omega$ is an $H^{\rk+1}$-diffeomorphism since $\|h\|_{H^{\rk+0.5}(\bdy\D)} < \epsilon \ll 1$.
We note that according to the proofs of Corollary \ref{cor:JA_est} and Theorem \ref{thm:HkHl_product2}, there exist generic constants $c_1$ and $C_1$ independent of $|\bdy\Omega|_{H^{\rk+0.5}}$ such that if $j\le k+1$,
\begin{equation}
c_1(1-\epsilon) \|f\|_{H^j(\Omega)} \le \|f\circ \psi\|_{H^j(\D)} \le C_1(1 + \epsilon) \|f\|_{H^j(\Omega)}\qquad\Forall f\in H^j(\Omega)\,. \label{another_f_comp_psi_est}
\end{equation}
As a consequence, letting $A = (\nabla \psi)^{-1}$ we obtain that
\begin{align*}
\|(\curl \u)\circ \psi\|_{H^\rk(\D)} &= \|\varepsilon_{ijk} A^r_j (\u^k\circ \psi),_r\|_{H^\rk(\D)} = \|\varepsilon_{ijk} (A^r_j - \delta^r_j) (\u^k\circ \psi),_r + \varepsilon_{ijk} (\u^k\circ \psi),_j\|_{H^\rk(\D)} \\
&\ge \|\curl (\u \circ \psi)\|_{H^\rk(\D)} - C \|A - \id\|_{H^\rk(\D)} \|\u\circ \psi\|_{H^{\rk+1}(\D)}\,,
\end{align*}
where the constant $C = C(|\bdy\D|_{H^{\rk+0.5}})$. Therefore,
\begin{subequations}\label{div-curl_est}
\begin{align}
\|\curl (\u \circ \psi)\|_{H^\rk(\D)} &\le \|(\curl \u)\circ \psi\|_{H^\rk(\D)} + C \epsilon \|\u\circ \psi\|_{H^{\rk+1}(\D)} \nonumber\\
&\le C_1 \|\curl \u\|_{H^\rk(\Omega)} + (C_1 + C) \epsilon \|\u\|_{H^{\rk+1}(\Omega)}\,.
\end{align}
\end{subequations}
Similarly,
$$
\|\div (\u \circ \psi)\|_{H^\rk(\D)} \le C_1 \|\div \u\|_{H^\rk(\Omega)} + (C_1 + C) \epsilon \|\u\|_{H^{\rk+1}(\Omega)} \,. \eqno{\rm(\ref{div-curl_est}b)}
$$
Let $\nn$ be the outward-pointing unit normal to $\bdy\D$. Then by the identity $\bN \circ \psi = \dfrac{A^\rT \nn}{|A^\rT \nn|}$, we find that
$$
\|(\bN \circ \psi) - \nn\|_{H^{\rk-0.5}(\bdy D)} \le C_2(|\bdy\D|_{H^{\rk+0.5}}) \epsilon \,.
$$
Therefore, in addition to estimate (\ref{div-curl_est}a,b), we also have
\begin{align*}
\|\bpD (\u\circ \psi)\cdot \nn\|_{H^{\rk-0.5}(\bdy\D)} &\le \|\nabla_{\bdy \D} (\u\circ \psi) \cdot (\bN \circ \psi)\|_{H^{\rk-0.5}(\bdy\D)} + C_2 \epsilon \|\u\|_{H^{\rk+1}(\bdy\Omega)} \\
&\le C_1 (1+\epsilon) \|\bdygrad \u \cdot \bN\|_{H^{\rk-0.5}(\bdy\Omega)} + C_2 \epsilon \|\u\|_{H^{\rk+1}(\Omega)} \\
&\le C_1 \|\bdygrad \u \cdot \bN\|_{H^{\rk-0.5}(\bdy\Omega)} + (C_1 + C_2) \epsilon \|\u\|_{H^{\rk+1}(\Omega)}\,.
\end{align*}
Finally, by Theorem \ref{thm:main_thm2}, there exists a generic constant $C_3 = C_3(|\bdy\D|_{H^{\rk+0.5}})$ such that
$$
\|\v\|_{H^{\rk+1}(\D)} \hspace{-1pt}\le\hspace{-1pt} C_3 \Big[\|\v\|_{L^2(\D)} \hspace{-1pt}+\hspace{-1pt} \|\curl \v\|_{H^\rk(\D)} \hspace{-1pt}+\hspace{-1pt} \|\div \v\|_{H^\rk(\D)} \hspace{-1pt}+\hspace{-1pt} \|\bpD \v \cdot \nn \|_{H^{\rk-0.5}(\bdy\D)} \Big] \quad \Forall \v \in H^{\rk+1}(\D)\,.
$$
Letting $\v = \u\circ \psi$, using (\ref{another_f_comp_psi_est}) and (\ref{div-curl_est}) we find that
\begin{align*}
c_1 (1-\epsilon) \|\u\|_{H^{\rk+1}(\Omega)} &\le C_3 C_1 \Big[\|\u\|_{L^2(\Omega)} \hspace{-1pt}+\hspace{-1pt} \|\curl \u\|_{H^\rk(\Omega)} \hspace{-1pt}+\hspace{-1pt} \|\div \u\|_{H^\rk(\Omega)} \hspace{-1pt}+\hspace{-1pt} \|\bdygrad \u \cdot \bN \|_{H^{\rk-0.5}(\bdy\Omega)} \Big] \\
&\quad + C_3(C_1 + C_2 + C)\epsilon \|\u\|_{H^{\rk+1}(\Omega)} \qquad \Forall \u\in H^{\rk+1}(\Omega)\,.
\end{align*}
Since $\epsilon \ll 1$, the last term on the right-hand side can be absorbed by the left-hand side, yielding a linear inequality. The conclusion
of Theorem {\rm\ref{thm:main_thm3}} then follows by linear interpolation.

\appendix
\section{Proofs of the inequalities in Section \ref{sec:useful_ineq}}\label{appendixA}
\begin{proof}[Proof of Proposition {\rm\ref{prop:Cr_domain_charts}}]
Let $p \in \bdy\Omega$. Since $\bdy\Omega$ is a $\mC^\rk$-surface, there exists a tangent space $T_p \bdy\Omega$ to the boundary $\bdy\Omega$ at $p$, and in a neighborhood of $p$ $\bdy\Omega$ can be view as a graph of a function $\phi$ defined on that neighborhood. In other words, there exist 
an (\smallexp{$(\n-1)$}-dimensional) ball $D(p,R) \subseteq T_p \bdy\Omega$ and a $\mC^\rk$-map $\phi:D(p,R) \to \bdy\Omega$ such that
$$
\big\{(y',\phi(y'))\in \bbR^\n\,\big|\, y'\in D(p,R)\big\} \subseteq \bdy\Omega\,.
$$
Choose an othornormal basis $\{\re_1,\cdots,\re_{\n-1}\}$ on the tangent space $T_p \bdy\Omega$ such that every $y'\in T_p\bdy\Omega$ can be written as $y' = (y_1,\cdots,y_{\n-1})$ in the sense that $y' = y_1 \re_1 + \cdots y_{\n-1} \re_{\n-1}$, and let $\bN$ denote the inward-pointing unit normal to the tangent plane $T_{p} \bdy\Omega$. Then $\{\re_1,\cdots,\re_{\n-1},\bN\}$ as the orthonormal basis in $\bbR^\n$. We define $\vartheta:B(0,R) \to \bbR^\n$, where $B(0,R)$ is a $\n$-dimensional ball, by
$$
\vartheta(y',y_\n) = y' + \big(y_\n + \phi(y')\big) \bN \qquad\Forall y' \in D(p,R)\,, (y',y_\n) \in B(0,R)\,,
$$
or equivalently,
$$
\vartheta(y_1,\cdots,y_\n) = \big(y_1,\cdots,y_{\n-1}, y_\n + \phi(y')\big)\,.
$$

Since $T_p \bdy\Omega$ is a tangent space, $\vartheta,_\alpha(0) \cdot \bN = 0$ for all $1\le \alpha\le \n-1$; thus $\phi,_\alpha(0) = 0$ for all $1\le \alpha \le \n-1$. By the continuity of $\nabla \phi$, there exists $r> 0$ such that
$$
|\phi,_\alpha(y')| < \varepsilon \qquad\text{whenever}\quad |y'| < r \text{ and } 1\le \alpha \le \n-1\,.
$$
It is clear that $\vartheta$ is injective on $B(0,r)$, and the inverse function theorem implies that $\vartheta:B(0,r) \to \U := \vartheta(B(0,r))$ is a $\mC^\rk$-diffeomorphism between open sets $B(0,r)$ and $\U$. Moreover, $\vartheta: B(0,r) \cap \{y_\n = 0\} \to \bdy\Omega$ and by the choice of $\bN$, $\vartheta: B(0,r) \cap \{y_\n > 0\} \to \Omega$. Finally, since
$$
\nabla \vartheta(y) = \left[\begin{array}{ccccc}
1 & 0 & \cdots & \cdots & 0\\
0 & \ddots & \ddots & & \vdots \\
\vdots & & \ddots & \ddots & 0 \\
0 & \cdots & 0 & 1 & 0\\
\phi,_1(y') &\cdots & \cdots & \phi,_{\n-1}(y') & 1
\end{array}
\right]\,,
$$
we immediately conclude that $\det(\nabla \vartheta) = 1$ and $\|\nabla \vartheta - \id\|_{L^\infty(B(0,r))} < \varepsilon$. Therefore, we establish that for each $p\in \bdy\Omega$, there exist $\vartheta:B(0,r) \to \U$ satisfying
\begin{enumerate}
\item $\vartheta: B(0,r) \to \U \text{ is a } \mC^\rk\text{-diffeomorphism} $;
\item $\vartheta: B(0,r) \cap \{y_n=0\} \to \U \cap \p \Omega $;
\item $\vartheta: B(0,r) \cap \{y_\n>0\} \to \U \cap \Omega $;
\item $\det (\nabla \vartheta) =1 $;
\item $\|\nabla \vartheta - \id\|_{L^\infty(B(0,r))} \le \varepsilon$.
\end{enumerate}
The proposition is concluded by the fact that $\bdy\Omega$ is compact.
\end{proof}

\begin{proof}[Proof of Proposition {\rm\ref{prop:HkHl_product}}]
We estimate $\nabla^j f \nabla^{\ell-j} g$ for $j=1,\cdots,\ell$ as follows:
\begin{enumerate}
\item[(1)] If $1\le j\le \novertwo$, by the Sobolev inequalities
  \begin{align*}
  \|w\|_{L^\frac{\n}{j-\sigma}(\Omega)} \hspace{-2pt}&\le 
  C_\sigma \|w\|_{H^{\frac{\n}{2} - j + \sigma}(\Omega)} \quad\text{(\hspace{1pt}if $0< \sigma < 1$)}\,,\\
  \|w\|_{L^\frac{2\n}{\n-2(j-\sigma)}(\Omega)} \hspace{-2pt}&\le C \|w\|_{H^{j-\sigma}(\Omega)}\,,
  \end{align*}
  we find that
  \begin{align*}
  \qquad\quad\|\nabla^j f \nabla^{\ell-j} g\|_{L^2(\Omega)} &\le \|\nabla^j f\|_{L^\frac{\n}{j-\sigma}(\Omega)} \|\nabla^{\ell-j} g\|_{L^\frac{2\n}{\n-2(j-\sigma)}(\Omega)} \le C_\sigma \|f\|_{H^{\frac{\n}{2}+\sigma}(\Omega)} \|g\|_{H^{\ell-\sigma}(\Omega)} \,.
  \end{align*}

\item[(2)] If $j = \ell$, by the Sobolev inequality
  $$
  \qquad\qquad\quad \|w\|_{L^\infty(\Omega)} \hspace{-3pt}\le C_\sigma \|w\|_{H^{\frac{\n}{2} + \sigma}(\Omega)}\,,
  $$
  we find that
  $$
  \|\nabla^j f \nabla^{\ell-j} g\|_{L^2(\Omega)} 
  \le C_\sigma \|f\|_{H^\ell(\Omega)} \|g\|_{H^{\frac{\n}{2} + \sigma}(\Omega)} \,.
  $$

\item[(3)] If $\novertwo < j < \ell$ {\rm(}this happens only when $\novertwo < \ell \le \rk${\rm)}, we consider the following two sub-cases:
\begin{enumerate}
\item {\it The case $\ell \le \n$}\,: Similar to the previous case, by the Sobolev inequalities
  $$
  \|w\|_{L^\frac{2\n}{\n - 2(\ell-j)}(\Omega)} \hspace{-3pt}\le C \|w\|_{H^{\ell-j}(\Omega)} \ \ \text{and}\ \
  \|w\|_{L^\frac{\n}{\ell-j}(\Omega)} \hspace{-3pt}\le C \|w\|_{H^{\frac{\n}{2} - \ell + j}(\Omega)} \,,
  $$
  we obtain that
  \begin{align}
   \|\nabla^j f \nabla^{\ell-j} g\|_{L^2(\Omega)} &\le \|\nabla^j f\|_{L^\frac{2\n}{n - 2(\ell-j)}(\Omega)} \|\nabla^{\ell-j} g\|_{L^\frac{\n}{\ell-j}(\Omega)} \nonumber\\
  &\le C \|f\|_{H^\ell(\Omega)} \|g\|_{H^{\frac{\n}{2}}(\Omega)}\,. \nonumber
  \end{align}

\item {\it The case $\n < \ell \le \rk$}\,: If $j > \rk- \novertwo$\,, by the Sobolev inequalities
  $$
  \|w\|_{L^\frac{2\n}{\n - 2(\rk-j)}(\Omega)} \hspace{-3pt}\le C \|w\|_{H^{\rk-j}(\Omega)}\ \ \text{and}\ \
  \|w\|_{L^\frac{\n}{\rk-j}(\Omega)} \hspace{-3pt}\le C \|w\|_{H^{\frac{\n}{2} - \rk + j}(\Omega)}\,,
  $$
  we obtain that
  \begin{align}
  \qquad\qquad\quad\|\nabla^j f \nabla^{\ell-j} g\|_{L^2(\Omega)} &\le \|\nabla^j f\|_{L^\frac{2\n}{n - 2(\rk-j)}(\Omega)} \|\nabla^{\ell-j} g\|_{L^\frac{\n}{\rk-j}(\Omega)} \le C \|f\|_{H^\rk(\Omega)} \|g\|_{H^{\frac{\n}{2}-\rk+\ell}(\Omega)}\,. \nonumber
  \end{align}
  Now suppose that $\novertwo < j \le \rk - \novertwo$\,. Note that if $0 < \sigma < \smallexp{$\displaystyle{}\frac{1}{2}$}$,
  \begin{align*}
   \|w\|_{H^{\frac{\n}{2} + \sigma}(\Omega)} \hspace{-2pt}&\le C_\sigma \|w\|_{W^{j,\infty}(\Omega)} \le C_\sigma \|w\|_{H^\rk(\Omega)} \,, \\
   \|w\|_{H^{\frac{\n}{2}-\rk+\ell}(\Omega)} \hspace{-2pt}&\le C \|w\|_{H^{\ell-j}(\Omega)} \le C \|w\|_{H^{\ell-\sigma}(\Omega)}\,.
  \end{align*}
   Therefore, by the Gagliardo-Nirenberg-Sobolev interpolation inequality we obtain that
  \begin{align*}
  \|\nabla^j f \nabla^{\ell-j} g\|_{L^2(\Omega)} &\le \|f\|_{W^{j,\infty}(\Omega)} \|g\|_{H^{\ell-j}(\Omega)} \\
   & \le C_\sigma \|f\|^{1-\alpha_j}_{H^{\frac{\n}{2}+\sigma}(\Omega)} \|f\|^{\alpha_j}_{H^\rk(\Omega)} \|g\|^{\alpha_j}_{H^{\frac{\n}{2}-\rk+\ell}(\Omega)} \|g\|^{1-\alpha_j}_{H^{\ell-\sigma}(\Omega)}
  \end{align*}
  for some $\alpha_j \in (0,1)$; thus Young's inequality implies that
  $$
  \qquad\qquad\|\nabla^j f \nabla^{\ell-j}g\|_{L^2(\Omega)} \le C_\sigma \Big[\|f\|_{H^{\frac{\n}{2}+\sigma}(\Omega)} \|g\|_{H^{\ell-\sigma}(\Omega)} + \|f\|_{H^\rk(\Omega)} \|g\|_{H^{\frac{\n}{2}-\rk+\ell}(\Omega)} \Big]\,.
  $$
\end{enumerate}
\end{enumerate}
Summing over all the possible $\ell$, we conclude that for $0 < \sigma < \smallexp{$\displaystyle{}\frac{1}{2}$}$,
\begin{align*}
\sum_{j=1}^\ell \|\nabla^j f \nabla^{\ell-j} g\|_{L^2(\Omega)} \le \left\{\begin{array}{ll}
 C_\sigma \|f\|_{H^{\frac{\n}{2}+\sigma}(\Omega)} \|g\|_{H^{\ell - \sigma}(\Omega)} & \text{if $\ell \le \novertwo$}\,,\vspace{.2cm}\\
 C_\sigma \Big[\|f\|_{H^{\frac{\n}{2}+\sigma}(\Omega)} \|g\|_{H^{\ell-\sigma}(\Omega)} + \|f\|_{H^\rk(\Omega)} \|g\|_{H^{\frac{\n}{2}+\sigma}(\Omega)} \Big] & \text{otherwise}\,.
\end{array}
\right.
\end{align*}
Estimate (\ref{commutator_estimate_elliptic_est_temp}) is then concluded by the fact that for all $\sigma \in \big(0,\smallexp{$\displaystyle{}\frac{1}{4}$}\big)$,
$$
\frac{\n}{2} + \sigma \le \rk \quad\text{and}\quad
\frac{\n}{2} + \sigma \le \ell - \sigma \text{ \ if (in addition)\ }\ell > \frac{\n}{2}\,.
$$
Finally, we conclude estimate (\ref{HkHl_product}) by an additional estimate \vspace{.2cm}\\
$
\displaystyle{}\hspace{60pt} \|f \nabla^\ell g\|_{L^2(\rO)} \le \|f\|_{L^\infty(\rO)} \|g\|_{H^\ell(\rO)} \le C \|f\|_{H^\rk(\rO)} \|g\|_{H^\ell(\rO)}\,.
$
\end{proof}

\begin{proof}[Proof of Corollary {\rm\ref{cor:JA_est}}]
By the definition of determinant, inequality (\ref{HkHl_product}) shows that
$$
\|J\|_{H^\rk(\rO)} \le C \|\nabla \psi\|^\n_{H^\rk(\rO)}
\,.
$$
By the Sobolev embedding $H^\rk(\rO) \subseteq \mC^{0,\alpha}(\rO)$, we find that $J$ is uniformly continuous on $\cls{\rO}$. Since $J\ne 0$ in $\cls{\rO}$ (by the virtue of that $\psi$ being a diffeomorphism), $\|1/J\|_{L^\infty(\rO)} < \infty$ and $|J| \ge 1/\|1/J\|_{L^\infty(\rO)} > 0$. Let $\delta = 1/\|1/J\|_{L^\infty(\rO)}$. The cofactor formula for the inverse of matrices then provides that
\begin{equation}\label{A_L2_est}
\|A\|_{L^2(\rO)} \le \Big\|\dfrac{1}{J}\Big\|_{L^\infty(\rO)} \|JA\|_{L^2(\rO)} \le \frac{C}{\delta} \|\nabla \psi\|^{\n-2}_{H^\rk(\rO)} \|\nabla\psi\|_{L^2(\rO)} \,.
\end{equation}
Therefore, by interpolation and Young's inequality, using (\ref{commutator_estimate_elliptic_est_temp}) with $\sigma = \dfrac{1}{8}$ we find that
\begin{align*}
\|\nabla^\rk A\|_{L^2(\rO)} &\le \frac{1}{\delta} \|J \nabla^\rk A\|_{L^2(\rO)} \le \frac{1}{\delta} \Big[\|\nabla^k (J A)\|_{L^2(\rO)} + \sum_{j=1}^\rk \smallexp{$\displaystyle{}{{k}\choose{j}}$} \|\nabla^j J \nabla^{\rk-j} A\|_{L^2(\rO)}\Big] \\
&\le \frac{C}{\delta} \Big[ \|\nabla \psi\|^{\n-1}_{H^\rk(\rO)} + \|J\|_{H^\rk(\rO)} \|A\|_{H^{\rk-\frac{1}{8}}(\rO)} \Big] \\
&\le \frac{C}{\delta} \Big[\|\nabla \psi\|^{\n-1}_{H^\rk(\rO)} + \|J\|_{H^\rk(\rO)} \|A\|^{1-\frac{1}{8\rk}}_{H^\rk(\rO)} \|A\|^{\frac{1}{8\rk}}_{L^2(\rO)} \Big]\\
&\le C_{\delta,\delta_1} \big(\|\nabla \psi\|_{H^\rk(\rO)}\big) + \delta_1 \|A\|_{H^\rk(\rO)}\,.
\end{align*}
Combining the estimate above with (\ref{A_L2_est}), by choosing $\delta_1 > 0$ sufficiently small we conclude (\ref{JA_est}b).
\end{proof}

\begin{proof}[Proof of Corollary {\rm\ref{cor:f_comp_psi}}]
We prove (\ref{f_comp_psi_est}) by induction. Let $J = \det(\nabla \psi)$ and $A = (\nabla \psi)^{-1}$. With the help of (\ref{JA_est}), the case that $\ell=0$ is concluded by
\begin{equation}\label{f_comp_psi_L2_est1}
\|f\|^2_{L^2(\Omega)} = \int_{\rO} |(f\circ \psi)(y)|^2 |J(y)|\, dy \le C(\|\nabla \psi\|_{H^\rk(\rO)}) \|f\circ \psi\|^2_{L^2(\rO)}
\end{equation}
and
\begin{equation}\label{f_comp_psi_L2_est2}
\|f\circ\psi\|^2_{L^2(\rO)} = \int_\Omega |f(x)|^2 \frac{1}{(J\circ \psi^{-1})(x)}\, dx \le \frac{1}{\delta} \|f\|^2_{L^2(\Omega)}\,,
\end{equation}
where $\delta = 1/\|1/J\|_{L^\infty(\rO)} >0$ is a lower bound for $\|J\|_{L^\infty(\rO)}$. Suppose that (\ref{f_comp_psi_est}) holds for $\ell=j \,(\le \rk)$. Then for $\ell=j+1$, by (\ref{HkHl_product}) and (\ref{JA_est}) we obtain that
\begin{align*}
\|\nabla^{j+1}f\|_{L^2(\Omega)} &\le \|\nabla f\|_{H^j(\Omega)} \le C(\|\nabla\psi\|_{H^\rk(\rO)}) \|(\nabla f)\circ \psi\|_{H^j(\rO)} \\
&\le C(\|\nabla\psi\|_{H^\rk(\rO)}) \big\|A^\rT \nabla (f\circ\psi)\|_{H^j(\rO)} \\
\text{\tiny\color{blue}{($\le$ holds if $j\le \rk$)}} &\le C(\|\nabla\psi\|_{H^\rk(\rO)}) \|A\|_{H^\rk(\rO)} \|D(f\circ\psi)\|_{H^j(\rO)} \\
&\le C(\|\nabla\psi\|_{H^\rk(\rO)}) \|f\circ\psi\|_{H^{j+1}(\rO)}
\end{align*}
and
\begin{align*}
\|\nabla^{j+1}(f\circ \psi)\|_{L^2(\rO)} &= \big\|\nabla^j \big[(\nabla f)\circ\psi D\psi\big] \big\|_{L^2(\rO)} \le \big\|(\nabla f)\circ \psi D\psi\big\|_{H^j(\rO)} \\
\text{\tiny\color{blue}{($\le$ holds if $j\le \rk$)}} &\le C \|\nabla\psi\|_{H^\rk(\rO)} \|(\nabla f)\circ\psi \|_{H^j(\rO)} \\
&\le C \|\nabla\psi\|_{H^\rk(\rO)} \|\nabla f\|_{H^j(\Omega)} \le C \|\nabla \psi\|_{H^\rk(\rO)} \|f\|_{H^{j+1}(\Omega)}\,,
\end{align*}
which, together with the (\ref{f_comp_psi_L2_est1}) and (\ref{f_comp_psi_L2_est2}), concludes the case that $\ell = j+1$.
\end{proof}
\vspace{.1in}

\noindent {\bf Acknowledgments.} AC was supported by the Ministry of Science and Technology (Taiwan) under grant 103-2115-M-008-010-MY2 and by the National Center of Theoretical Sciences. SS was supported by the National Science Foundation under grants DMS-1001850 and DMS-1301380, by OxPDE at the University of Oxford,
and by the Royal Society Wolfson Merit Award. We would like to thank the referees for carefully reading our manuscript and for making extremely useful
showions and corrections which have improved the paper.

\bibliography{Hodge-elliptic}
\bibliographystyle{plain}

\end{document}